\documentclass [final]{amsart}

\usepackage{amsmath,amssymb,amsmath,amsfonts,amssymb,graphicx,latexsym,float}
\usepackage[all]{xy}
\usepackage{color}
\usepackage{hyperref}
\usepackage[open]{bookmark}
\usepackage{cite,bbm,verbatim}

\newtheorem{theorem}{Theorem}

\newtheorem{corollary}[theorem]{Corollary}
\newtheorem{lemma}[theorem]{Lemma}
\newtheorem{proposition}[theorem]{Proposition}
\newtheorem{definition}{Definition}
\newtheorem{remark}{Remark}

\usepackage{showkeys}

\theoremstyle{remark}

\theoremstyle{definition}

\newcommand{\rth}{\mathbb{R}^3}

\newcommand*{\toccontents}{\@starttoc{toc}}

\allowdisplaybreaks

\begin{document}
\title[The Landau equation with the specular boundary condition]{The Landau equation with the specular reflection boundary condition}

\author{Yan Guo}
\address{Brown University, Providence RI 02912, USA }
\email{yan\_guo@brown.edu}

\author{Hyung Ju Hwang}
\address{Department of Mathematics, POSTECH, Pohang 37673, Republic of Korea }
\email{hjhwang@postech.ac.kr}

\author{Jin Woo Jang}
\address{Center for Geometry and Physics, Institute for Basic Science (IBS), Pohang 37673, Republic of Korea }
\email{jangjinw@ibs.re.kr}

\author{Zhimeng Ouyang}
\address{Brown University, Providence RI 02912, USA }
\email{zhimeng\_ouyang@brown.edu}
\let\svthefootnote\thefootnote
\let\thefootnote\relax
	\addtocounter{footnote}{-1}\let\thefootnote\svthefootnote
\newcommand{\eqdef  }{\overset{\mbox{\tiny{def}}}{=}}
\thispagestyle{empty}
\begin{abstract}The existence and stability of the \textit{Landau equation} (1936) in a general bounded domain with a physical boundary condition is a long-outstanding open problem. This work proves the global stability of the \textit{Landau equation} with the Coulombic potential in a general smooth bounded domain with the \textit{specular reflection} boundary condition for initial perturbations of the Maxwellian equilibrium states. The highlight of this work also comes from the low-regularity assumptions made for the initial distribution. This work generalizes the recent global stability result for the \textit{Landau equation} in a periodic box \cite{2016arXiv161005346K}. Our methods consist of the generalization of the wellposedness theory for the \textit{Fokker-Planck equation}  \cite{hwang2014fokker, MR3788197} and the extension of the boundary value problem to a whole space problem, as well as the use of a recent extension of \textit{De Giorgi-Nash-Moser} theory for the \textit{kinetic Fokker-Planck equations} \cite{golse2016harnack} and the \textit{Morrey} estimates \cite{polidoro1998sobolev} to further control the velocity derivatives, which ensures the uniqueness.
Our methods provide a new understanding of the grazing collisions in the Landau theory for an initial-boundary value problem.
\end{abstract}
\maketitle
\tableofcontents 
 \section{Introduction}\label{section: introduction}
The  \textit{Landau equation}, which was proposed by Landau in 1936, is a fundamental mathematical model that describes collisions among charged particles interacting via Coulombic force.  Especially, this model describes the dynamics of a large number of particles when all collisions tend to be grazing. 
The equation takes the form of 
 \begin{equation}\label{Landau}
\partial_t F + v\cdot\nabla_{\!x} F = Q(F,F),
 \end{equation}
 where the unknown $F=F(t,x,v)$ is a non-negative function. For each time $t\geq 0$, $F(t,\cdot,\cdot)$ represents the density of particles in phase space. Throughout this research, the spatial coordinates are $x\in \Omega \subset\rth$ and the velocities are $v\in\rth$, where $\Omega \subset \rth$ is a bounded domain. The \textit{Landau collision operator} $Q$ is defined as 
 \begin{align*}
 Q(F,G)(v) \eqdef & \;\nabla_v\cdot\left\{\int_{\mathbb{R}^3}\phi(v\!-\!v')\big[F(v')\nabla_{\!v} G(v)-G(v)\nabla_{\!v} F(v')\big]dv' \right\},
 \end{align*}
 where the collision kernel for the Coulombic particle interactions
 \begin{align*}
 \phi(z) \eqdef & \left\{I-\frac{z}{|z|}\otimes\frac{z}{|z|}\right\}\cdot|z|^{-1},
 \end{align*} is a symmetric and non-negative matrix such that $\phi_{ij}(z)z_iz_j=0$. 

The \textit{Landau equation} is a limit case of the \textit{Boltzmann equation}$$
\partial_t F + v\cdot\nabla_{\!x} F = Q_B(F,F),$$ with the \textit{Boltzmann collision operator}
\begin{align*}
 Q_B(F,G)(v) \eqdef \int_{\mathbb{R}^3}dv_*\int_{\mathbb{S}^2}d\sigma\ B(v-v_*,\sigma)\big[F(v')G(v'_*)-F(v)G(v_*)\big],
 \end{align*} in the sense of the phenomenological arguments by Landau that solutions to the \textit{Boltzmann equation} tend to solutions of the \textit{Landau equation} if all collisions tend to be grazing; in order words, $B(v-v_*,\sigma)$ tends to be more singular for $\langle v-v_*,\sigma\rangle\approx0$. The examples of such physical potentials include the inverse-power-law potential and other long-range potentials, whose dynamics can further be described via the
\textit{Boltzmann equation} without the classical \textit{angular cutoff} assumptions; i.e., $B(v-v_*,\cdot)\notin L^1_{loc}(\mathbb{S}^2)$. It is worth mentioning that the dynamics that the \textit{Landau equation} describes are different from those of the \textit{Boltzmann equation} with the classical \textit{angular cutoff} in the sense that the former describes the dynamics of grazing collisions whereas the latter describes the dynamics neglecting grazing collisions. 
More detailed study of connection with the \textit{Boltzmann equation} is given in the literature
\cite{MR0434265, MR996783,  MR1167768,MR1165528}.

The \textit{Landau equation} is very interesting not just in itself but also in that we hope we can have a better understanding on the
\textit{Boltzmann equation} without the classical \textit{angular cutoff} assumptions in the case when grazing collisions are not neglected.

\subsection{Historical remarks on the well-posedness theory}The well-posedness theory for the \textit{Landau equation} is strongly related to \textit{a priori} estimates on the non-linear \textit{Landau collision operator} $Q(F,F)$ and the regularity conditions of the solutions. Unfortunately, the only ``easily-granted" \textit{a priori} estimates that one can expect from the  \textit{Landau collision operator} are the physical $L^1$-type conservation laws. Historically, this difficulty that arises from the lack of strong \textit{a priori} estimates of the solutions has been resolved via the following techniques:
\begin{itemize}
\item (Spatially homogeneous equation) The brief list of the results that considered the spatially homogeneous situation includes \cite{ MR3884792, MR3375485, 2018arXiv180608720S, desvillettes2015entropy, MR2514370, MR2745513, MR3365830, MR3407515,  MR3158719, MR1737548, MR1737547, MR2502525, MR2718931, MR3599518, MR3158719, MR1650006, MR1646502, MR1750572, MR2288534}. 
\item (Renormalized equation) One can also consider renormalizing the equation and obtain additional \textit{a priori} estimates. The brief list includes  \cite{MR972541,MR1392006,MR1278244,MR1014927,MR1127927 }. 
\item (Linearized equation nearby the Maxwellian equilibrium) The list of results includes \cite{carrapatoso2016cauchy, chen2009smoothing, guo2002landau,  2016arXiv161005346K, ha2015l2, herau2011anisotropic, herau2013global, strain2006almost, liu2014regularizing, strain2013vlasov,luo2016spectrum}. 
\end{itemize}
Unfortunately, all of the results listed above are the wellposedness theory on the simple torus or the whole space. To the best of authors' knowledge, there is no global wellposedness result yet on the initial-boundary value problem for the \textit{Landau equation} in a general bounded domain with nontrivial physical boundary conditions.

\subsection{Introduction to an initial-boudary value problem for the Landau equation}
Our main concern in this paper is on the global wellposedness and the decay of the weak solution to the \textit{Landau equation} in a general bounded domain with the \textit{specular reflection} boundary conditions in the nearby-equilibrium setting without employing high order Sobolev norms. One may say that the \textit{Landau equation}, which models the behavior of charged particles, is intrinsically an equation for the particles in a bounded domain. In spite of the importance of the theory for the initial-boundary value problem, no global wellposedness theory has been developed for nontrivial physical boundary conditions due to the difficulties arising nearby the boundary.
 When the particles approach the boundary of a bounded domain $\Omega\subset \rth$, we must impose a specific boundary condition for the probability density function $f(t,x,v)$. The boundary assumptions that we impose are as follows.
 \subsubsection{Boundary conditions}
Throughout this paper, our domain $\Omega=\{x:\zeta(x)<0\}$ is connected and bounded with $\zeta(x)$ being a smooth function. We also assume that $\nabla\zeta(x) \neq 0$ at the boundary $\zeta(x)=0.$ We define the outward normal vector $x(x)$ on the boundary $\partial \Omega$ as 
 \begin{equation}\label{out}
 n_x\eqdef  \frac{\nabla \zeta(x)}{|\nabla\zeta(x)|}.
 \end{equation}
We say that $\Omega$ has a rotational symmetry if there exist vectors $x_0$ and $\omega$ such that 
 \begin{equation}\label{rot}
 \{(x-x_0)\times \omega\}\cdot n_x=0,
 \end{equation}for all $x\in\partial \Omega$.
 
 Throughout this paper, we will denote the phase boundary of $\Omega\times \rth$ as $\gamma\eqdef  \partial\Omega\times \rth.$ Additionally we split this boundary into an outgoing boundary $\gamma_+$, an incoming boundary $\gamma_-$, and a singular boundary $\gamma_0$ for grazing velocities, defined as
 \begin{equation}
 \begin{split}
 \gamma_+&\eqdef  \{(x,v)\in\Omega\times \rth :n_x\cdot v >0\},\\
 \gamma_-&\eqdef  \{(x,v)\in\Omega\times \rth :n_x\cdot v <0\},\\
 \gamma_0&\eqdef  \{(x,v)\in\Omega\times \rth :n_x\cdot v =0\}.
 \end{split}
 \end{equation}

 In terms of the probability density function $F$, we formulate the \textit{specular reflection boundary condition} as \begin{equation}\label{specular}
 F(t,x,v)|_{\gamma_-}=F(t,x,v-2n_x(n_x\cdot v))=F(t,x,R_xv),
 \end{equation} for all $x\in\partial \Omega$
 where $$R_xv\eqdef  v-2n_x(n_x\cdot v).$$

 \subsection{Linearization and conservation laws}
 We will study the linearization of \eqref{Landau} around the Maxwellian equilibrium state
 \begin{equation}\label{perturbation}
 F(t,x,v) = \mu(v) + \mu^{1/2}(v)f(t,x,v),
\end{equation}where without loss of generality $$
 \mu(v) = (2\pi)^{-3/2}e^{-\frac{|v|^2}{2}}.
$$
 It is well-known that under the specular reflection boundary condition (\ref{specular}), both mass and energy are conserved for the \textit{Landau equation} (\ref{Landau}). Without loss of generality, we assume the mass-energy conservation laws hold for $t\geq 0$ in terms of the perturbation $f$:
 \begin{equation}
 \label{conservation laws}
 \int_{\Omega\times\rth}f(t,x,v)\sqrt{\mu}dxdv=
 \int_{\Omega\times\rth}|v|^2f(t,x,v)\sqrt{\mu}dxdv=0.
 \end{equation}
 Additionally, we further assume a corresponding conservation law of angular momentum for all $t\geq 0$ if the domain has a rotational symmetry:
 \begin{equation}
 \label{angcon}
 \int_{\Omega\times\rth}\{(x-x_0)\times \omega\}\cdot v f(t,x,v)\sqrt{\mu}dxdv=0.
 \end{equation}

\subsection{Main theorem and our strategy}
We may now state our main result as follows:
\begin{theorem}
[Main theorem]\label{Thm : main result} There exist $\vartheta^{\prime}$ and 
$0<\varepsilon_{0}\ll 1$ such that for some $\vartheta \ge \vartheta^{\prime}$ if $f_{0}$ satisfies
\begin{equation}
\label{Eq : initial condition}
	\|f_{0}\|_{\infty,\vartheta} \le \varepsilon_0, \quad \|f_{0t}\|_{\infty, \vartheta} + \|D_v f_0\|_{\infty, \vartheta}  < \infty,
\end{equation}
where $f_{0t} \eqdef -v\cdot \nabla_x f_{0} + \bar A_{f_{0}}f_{0}$ and $\bar{A}_g$ is defined in \eqref{A_g}.
\begin{itemize}
\item (Existence and Uniqueness) Then there exists a unique weak solution $f$ of \eqref{linearized eq},
\eqref{initial}, and \eqref{specular} on $(0,\infty)\times\Omega \times\mathbb{R}^{3}$.
\item (Positivity) Let $F(t,x,v) = \mu(v) + \sqrt \mu(v) f(t,x,v)$.
If $F(0) \ge 0$, then $F(t) \ge 0$ for every $t\ge 0$.

\item (Decay of solutions in $L^2$ and $L^\infty$) Moreover, for any $t>0$, $\vartheta_{0}\in\mathbb{N}$, and $\vartheta\ge \vartheta'$, there exist $C_{\vartheta, \vartheta_{0}}>0$ and
$l_0(\vartheta_0)>0$ such that $f$ satisfies
\begin{equation*}
\label{Eq : energy estimate}
	\sup_{0 \le s\le\infty}\mathcal{E}_{\vartheta}(f(s)) \le C 2^{2\vartheta}\mathcal{E}_{\vartheta}(0),
\end{equation*}
\begin{equation*}
\label{Eq : decay estimate}
	\|f(t)\|_{2,\vartheta} \le C_{\vartheta,\vartheta_{0}} \mathcal{E}_{\vartheta+\vartheta_{0}/2}(0)^{1/2}\left(  1+ \frac{t}{\vartheta_{0}}\right)  ^{-\vartheta_{0}/2},
\end{equation*}and
\begin{equation*}
\label{Eq : L^infty estimate}
	\|f(t)\|_{\infty, \vartheta} \le C_{\vartheta,\vartheta_{0}}(1+t)^{-\vartheta_{0}}\|f_{0}\|_{\infty,\vartheta+l_{0}}.
\end{equation*}
\item (Boundedness in $C^{0,\alpha}$ and $W^{1,\infty}$) In addition, there exist $C>0$ and $0<\alpha<1$ such that $ f$ satisfies
\begin{equation*}
\label{Eq : Holder}
	\|f\|_{C^{0,\alpha}\big((0,\infty)\times\Omega \times\mathbb{R}^{3})\big)} \le C\left(\|f_{0t}\|_{\infty, \vartheta} + \|f_{0}\|_{\infty, \vartheta}\right),
\end{equation*}
and
\begin{equation*}
\label{Eq : D_v f bddness}
	\|D_{v} f\|_{L^{\infty}\big((0,\infty)\times\Omega \times\mathbb{R}^{3})\big)}  \le C\left(\|f_{0t}\|_{\infty, \vartheta} + \|D_v f_0\|_{\infty, \vartheta} + \|f_{0}\|_{\infty, \vartheta}\right).
\end{equation*}
\end{itemize}
\end{theorem}

We will now make a few comments on Theorem \ref{Thm : main result}. Our main concern throughout this paper is on the study of (low-regularity) global well-posedness for the Landau equation in
a general bounded domain with a physical boundary condition: namely, the specular reflection boundary condition. The \textit{Landau equation} has been extensively studied in either a simple periodic domain or the whole domain (See \cite{boblylev2013particle}, \cite{bobylev2015some}, \cite{carrapatoso2016cauchy}, \cite{chen2009smoothing}, \cite{desvillettes2015entropy}, \cite{guo2002landau}, \cite{ha2015l2}, \cite{herau2011anisotropic}, \cite{herau2013global}, \cite{liu2014regularizing}, \cite{luo2016spectrum}, \cite{strain2006almost}, \cite{strain2013vlasov}, and \cite{MR3158719}), and these were via employing Sobolev norms of sufficiently high orders.  However, in a bounded domain, the solutions
cannot be smooth up to the grazing set \cite{hwang2014fokker} even though we have the diffusion term in $v$ variable. Hence some new mathematical
tools involving much weaker norms must be developed. As the first step, a $L^{2}\rightarrow L^{\infty }$ framework has been
developed to construct the unique global solution in a periodic box in \cite{2016arXiv161005346K}. Our work generalizes \cite{2016arXiv161005346K} to the problem on a general bounded domain with a more physical boundary condition, the specular reflection boundary condition.

Our starting point is to linearize the \textit{Landau equation} \eqref{Landau} around the perturbation \eqref{perturbation}. Then the first step is to construct a weak solution to the \textit{linearized Landau equation} \eqref{linearized eq}, as the notion of a weak solution is no longer simple if we consider a nontrivial boundary. An alternative form of the \textit{linearized Landau equation} is given as 
\begin{equation}\label{g-eq}
\partial_t f + v\cdot\nabla_{\!x}f = \bar{A}_g f + \bar{K}_g f,\end{equation} for some given function $g$, where  $\bar{A}_g f$  consists of the terms which contain at least one momentum-derivative of $f$ while  $\bar{K}_g f$ consists of the rest. Then, in Section \ref{section: linear} through Section \ref{section: wellposedness lin}, we first consider constructing weak solutions for the linearized equation \begin{equation}\label{eq Agf only}
\partial_t f + v\cdot\nabla_{\!x}f = \bar{A}_g f,\end{equation} without the presence of the lower order term $\bar{K}_g f$. The wellposedness for the linearized equation \eqref{eq Agf only} is then obtained via regularizing the problem, constructing approximate solutions, and showing $L^1$ and $L^\infty$ estimates for the adjoint problem to the approximate problem. This method is a generalization of the work \cite{hwang2014fokker} and \cite{MR3788197} for the Fokker-Planck equation with the absorbing boundary condition.

Once we are equipped with Theorem \ref{main-thm}, i.e., the wellposedness of \eqref{eq Agf only} with the \textit{specular reflection boundary condition} in the sense of distribution, we may associate a continuous semigroup of linear and bounded operators $U(t)$ such that $f(t) = U(t)f_0$ is the unique weak solution of \eqref{eq Agf only}.
Then by the Duhamel principle, the solution $\bar{f}$ of the whole linearized equation \eqref{linearized eq} can further be written as
\begin{equation}
\label{Eq : Duhamel principle}
\bar{f}(t) = U(t)\bar{f}_{0} + \int_{0}^{t} U(t-s) \bar K_{g} \bar{f}(s) ds.
\end{equation}
After we construct the notion of solutions to the linearized Landau equation, we continue developing further estimates as in the following diagram:
\begin{figure}[H]
\begin{equation*}
\xymatrix@C=0.3pc{
	& *+[F]{\txt{\small Energy estimate\\ \small \& $L^2$ time-decay}} \ar[d]_{{\color{red} \txt{\footnotesize \emph{Boundary flattening}\\ \footnotesize \emph{+ Moser's iteration}}}}
	& *+[F--]{\txt{\small $L^2$ weak-sol.\! of  (\ref{g-eq}) :\\ \small Well-posedness\\ \small (Theorem \ref{main-thm})}} \ar[l] \ar@{=>}@(r,ul)[dr] \\
	& *+[F-:<5pt>]{\txt{$L^2\!\rightarrow\!L^\infty$ estimate}} \ar[d]_{{\color{red} \txt{\footnotesize \emph{De Giorgi's method}}}} \ar@{-->}@/^/[drr] \ar@{.>}@/_1pc/[u] \ar@{.>}[rr]^(.45){\scriptsize \big( g\eqdef f^{(m)},\;  \|g\|_{\infty}< \varepsilon  \big)}
	&
	& *+[F.:<7pt>]{\txt{\small {\color{red} Iteration argument :}\\ \small construct iterating   \\ \small $f^{(m)}$ through  (\ref{linearized-eq})}} \ar@{->>}[d]^(.46){\txt{\scriptsize Pass to the limit\\ \scriptsize (Cauchy sequence)}} \\
	& *+[F]{\txt{H\"{o}lder estimate}} \ar[d]_{\color{red} \txt{\footnotesize \emph{Linear}\\ \footnotesize \emph{$L^2$ and $L^\infty$ decay}\\ \footnotesize \emph{+ interpolation}}} \ar@{.>}@/^1pc/[d]^(.5){\txt{\scriptsize $\big( g\eqdef f^{(m)},$ \\ \scriptsize $\|g\|_{C^{0,\alpha}}\!\leq\! C  \rightarrow \|\sigma_{\!G}\|_{C^{0,\alpha}}\!\leq\! C  \big)$}}
	&
	& *+[F=]{\txt{\small Global $L^2\cap L^{\!\infty}$ weak-sol.\\ \small of  (\ref{linearized eq}) :\\ \small  Existence \& Uniqueness \\ \small (Theorem \ref{Thm : main result})}} \\
	& *+[F]{\txt{$S^p$ estimate}} \ar@{=>}[r]
	& *+[F-:<5pt>]{\txt{$L^\infty$ bound of\;$\nabla_{\!v}f$\;}} \ar@{-->}@(r,dl)[ur]}
\end{equation*}
\caption{$L^2\rightarrow L^\infty$ approach.}
\end{figure}
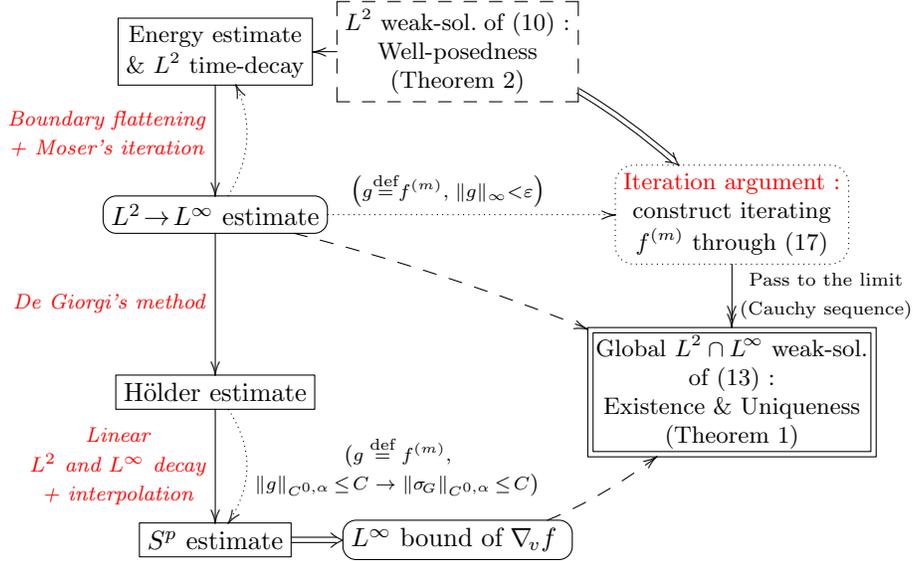
We may make a few comments on the diagram. Unfortunately, to the best of our knowledge, there is
still no construction for $L^{\infty }$ global weak solutions to the Landau
equation with nontrivial physical boundary conditions. One of the major difficulties behind the $L^\infty$ theory for the initial-boundary value problem for the \textit{Landau equation} arises from the fact that one must obtain the boundedness of a higher-order derivative norm $||\nabla_v f||_{L^\infty}$ in order to prove the uniqueness. It is worth comparing it with the $L^\infty$ wellposedness theory for the initial-boundary value problem for the \textit{Boltzmann equation} with \textit{angular cutoff} \cite{MR2679358}, which does not involve the estimates on the higher order derivative norms.  

Since our goal throughout this paper is to construct a $L^\infty$ global weak solution to the nonlinear \textit{Landau equation} with the specular-reflection boundary condition, the main approach that we take is mainly the $L^2-L^\infty$ bootstrapping after the $L^2$ energy estimates. In order for this, we first obtain the $L^2$ decay estimates in Section \ref{sec: L2decay}, motivated by the constructive method in \cite{MR3712934} where some extra efforts are put into the corresponding elliptic problem with certain boundary conditions. The key idea for the $L^2-L^\infty$ bootstrapping is to flatten the general $C^1$ boundary and to adapt the mirror extension of the boundary so that our argument can be analyzed as the one in the whole space. 

However, in order to close the wellposedness argument for the nonlinear problem, we must have the boundedness of $||\nabla_v f||_{L^\infty}$ as we discussed above. The key idea that we are based on is to follow so-called \textit{Morrey estimates} \cite{polidoro1998sobolev} (what we call as \textit{$S_p$ estimate} throughout this paper), which results in the boundedness of $||\nabla_v f||_{L^\infty}$ from the boundedness of the $C^{0,\alpha}$ norm and the $L^\infty$ bounds of the terms in the initial data. Hence we also adapt the $L^\infty-C^{0,\alpha}$ bootstrap via recently developed so-called \textit{De Giorgi}-type methods: local Harnack-type arguments from \cite{golse2016harnack}, \cite{golse2015holder}, and \cite{mouhot2015holder}, and the uniform $L^\infty-C^{0,\alpha}$ bootstrap from \cite{2016arXiv161005346K}. With all of the ingredients ready, we close the global wellposedness argument for the nonlinear \textit{Landau equation} in a bounded domain with the specular reflection boundary condition.

Regarding the \textit{De Giorgi-Nash-Moser} iteration methods, we briefly summarize the recent development. The theory for elliptic or parabolic equations in divergence form \cite{MR0093649, MR0082045, MR0170091, MR0159139, MR0100158} has been extended to the study of hypoelliptic PDEs of divergent types to obtain H\"older regularity in \cite{ MR1825690, MR1386366, MR1662349, MR2068847,  mouhot2015holder, MR2530175, MR2409660, golse2016harnack, MR2773175, MR2277064, 2016arXiv160807571I, MR2386472, MR1289901, MR1463798}. The full summary of the story is introduced in \cite{2018arXiv180800194M}.

\subsection{Notations and function spaces}
We may now introduce the notations and the function spaces that we use throughout this paper. Below is the table of notations:

\paragraph{$\bullet$ Domain:}
\begin{align*}
& \Omega\subset\mathbb{R}_x^3  \text{: bounded domain with smooth boundary} \\
& \bar{Q}_T \eqdef  [0,T] \times \bar{\Omega} \times \mathbb{R}_v^3 \;\;\ni  (t,x,v)
\end{align*}

\paragraph{$\bullet$ Phase-boundary:}
\begin{align*}
& \gamma \eqdef  \partial\Omega \times \mathbb{R}_v^3 & \\
& \gamma_\pm \eqdef  \left\{ (x,v)\in\gamma:  \pm  (v\!\cdot\! n_x) >0 \right\} & \text{outcoming/incoming set} \\
& \gamma_0 \eqdef  \left\{ (x,v)\in\gamma:  v\cdot n_x =0 \right\} & \text{grazing (singular) set} \\
&  & (n_x\!: \text{outward unit normal at } x\!\in\!\partial\Omega) \\
& \Sigma^{ T} \eqdef  [0,T] \times \gamma & \\
& \Sigma_\pm^{ T} \eqdef  [0,T] \times \gamma_\pm &
\end{align*}

\paragraph{$\bullet$ Trace:}
\begin{equation*}
\gamma f \eqdef  f|_{\gamma}, \quad \gamma_\pm f \eqdef  f|_{\gamma_\pm}
\end{equation*}

\paragraph{$\bullet$ Spaces \& norms:}
\begin{align*}
& L^p\!\left(\Sigma_\pm^{ T}\right) \eqdef  L^p \!\left( \Sigma_\pm^{ T}  ;  |v\!\cdot\! n_x|  dS_xdvdt  \right) \\[5pt]
& L^p(\gamma_\pm) \eqdef  L^p (\gamma_\pm  ;  |v\!\cdot\! n_x|  dS_xdv) \\
& \left\| \gamma_\pm f \right\|_{L^p(\gamma_\pm)}^p \eqdef  \iint_{\gamma_\pm} \left| \gamma_\pm f \right|^p |v\!\cdot\! n_x|  dS_xdv
\end{align*}
\paragraph{$\bullet$ Weighted spaces \& norms:}
\begin{equation*}		\label{weighted L^p}
w\eqdef  (1+|v|),\quad|f|_{p,\vartheta}^{p}\eqdef  \int_{\mathbb{R}^{3}}w^{p\vartheta}|f|^{p}dv,\quad\Vert f\Vert_{p,\vartheta}^{p}\eqdef  \int_{\Omega\times\mathbb{R}^{3}}w^{p\vartheta}|f|^{p}dxdv.
\end{equation*}
\begin{equation*}		\label{weighted sigma 1}
|f|_{\sigma,\vartheta}^{2}\eqdef  \int_{\mathbb{R}^{3}}w^{2\vartheta}\left[\sigma^{ij}\partial_{i}f\partial_{j}f+\sigma^{ij}v_{i}v_{j}f^{2}\right]  dv,
\end{equation*}
\begin{equation*}		\label{weighted sigma 2}
\Vert f\Vert_{\sigma,\vartheta}^{2}\eqdef  \iint_{\Omega\times\mathbb{R}^{3}}w^{2\vartheta}\left[  \sigma^{ij}\partial_{i}f\partial_{j}f+\sigma^{ij}v_{i}v_{j}f^{2}\right]  dvdx,
\end{equation*}
\begin{equation*}		\label{Rw : weighted L^infty}
|f|_{\infty,\vartheta}=\sup_{\mathbb{R}^{3}}w^{\vartheta}(v)f(v),\quad\Vert f\Vert_{\infty,\vartheta}=\sup_{\Omega\times\mathbb{R}^{3}}w^{\vartheta}(v)f(x,v).
\end{equation*}

\[%
\begin{split}
|f|_{2}\eqdef  |f|_{2,0},  &  \quad\Vert f\Vert_{2}\eqdef  \Vert f\Vert_{2,0},\\
|f|_{\sigma}\eqdef  |f|_{\sigma,0},  &  \quad\Vert f\Vert_{\sigma}\eqdef  \Vert
f\Vert_{\sigma,0},\\
|f|_{\infty}\eqdef  |f|_{\infty,0},  &  \quad\Vert f\Vert_{\infty}\eqdef  \Vert
f\Vert_{\infty,0},
\end{split}
\]%

\paragraph{$\bullet$ $L^2$ inner products \& $L^2$ energy:}
\begin{equation*}		\label{inner L^2}
\langle f,g\rangle\eqdef  \int_{\mathbb{R}^{3}}fgdv,\quad(f,g)\eqdef  \int_{\Omega\times\mathbb{R}^{3}}fgdxdv,
\end{equation*}
\begin{equation*}		\label{inner sigma 1}
\langle f,g \rangle_{\sigma}\eqdef  \int_{\mathbb{R}^{3}}\left[\sigma^{ij}\partial_{i}f\partial_{j}g+\sigma^{ij}v_{i}v_{j}fg\right]  dv,
\end{equation*}
\begin{equation*}		\label{inner sigma 2}
( f, g)_{\sigma}\eqdef  \iint_{\Omega\times\mathbb{R}^{3}}\left[  \sigma^{ij}\partial_{i}f\partial_{j}g+\sigma^{ij}v_{i}v_{j}fg\right]  dvdx,
\end{equation*}

\begin{equation*}		\label{E}
\mathcal{E}_{\vartheta}(f(t)) \eqdef   \big|  f(t)\big|_{2,\vartheta}^{2} + \int_{0}^{t} \big|  f(s)\big|_{\sigma,\vartheta}^{2} ds.
\end{equation*}

\subsection{Reformulation}\subsubsection{Linearization}
We linearize the \textit{Landau equation} \eqref{Landau} around the perturbation \eqref{perturbation}. This grants an equation for the perturbation $f(t,x,v)$ as
 \begin{equation}\label{linearized eq}
 \partial_t f+v\cdot\nabla_{x}f+Lf=\Gamma(f,f),
 \end{equation} and 
 \begin{equation}
 f(0,x,v)=f_{0}(x,v), \label{initial}%
 \end{equation}
 where $f_{0}$ is the
 initial data satisfying the mass-energy conservation laws:$$
 \int_{\Omega\times\mathbb{R}^{3}}f_{0}(x,v)\sqrt{\mu}=\int
 _{\Omega\times\mathbb{R}^{3}}|v|^{2}f_{0}(x,v)\sqrt{\mu}=0.$$
 The linear operator $L$ is further defined as%
 \begin{equation}		\label{L}
 L = -A -K,
 \end{equation} where the linear operator $A$ consists of the terms with at least one momentum derivative on $f$ as$$
 Af \eqdef \mu^{-1/2}\partial_{i}\left\{  \mu^{1/2}\sigma^{ij}[\partial_{j}f+v_{j}f]\right\} =\partial_{i}[\sigma^{ij}\partial_{j}f]-\sigma^{ij}v_{i}v_{j}f+\partial_{i}\sigma^{i}f,
$$ whereas the linear operator $K$ consists of the rest of the operator $L$ which does not contain any momentum derivative of $f$ as$$
 Kf\eqdef -\mu^{-1/2}\partial_{i}\left\{  \mu\left[  \phi^{ij}\ast\left\{  \mu^{1/2}[\partial_{j}f+v_{j}f]\right\}  \right]  \right\}.
$$On the other hand, the nonlinear operator $\Gamma$ is defined as
 \begin{equation}		\label{Gamma}
 \begin{split}
 \Gamma\lbrack g,f]  &  \eqdef \partial_{i}\left[  \left\{  \phi^{ij}\ast\lbrack\mu^{1/2}g]\right\}  \partial_{j}f\right]  -\left\{  \phi^{ij}\ast\lbrack	v_{i}\mu^{1/2}g]\right\}  \partial_{j}f\\
 &  \quad-\partial_{i}\left[  \left\{  \phi^{ij}\ast\lbrack\mu^{1/2}%
 \partial_{j}g]\right\}  f\right]  +\left\{  \phi^{ij}\ast\lbrack v_{i}\mu^{1/2}\partial_{j}g]\right\}  f,
 \end{split}
 \end{equation}where the diffusion matrix (collision frequency) $\sigma^{ij}_u$ is defined as
$$
 \sigma_{u}^{ij}(v) \eqdef  \phi^{ij}*u = \int_{\mathbb{R}^{3}}\phi^{ij}(v-v^{\prime})u(v^{\prime})dv^{\prime}.
 $$We also denote the special case when $u=\mu$ as
$$
 \sigma^{ij} = \sigma^{ij}_{\mu}, \quad \sigma^{i} = \sigma^{ij}v_{j}.
$$
 \subsubsection{Alternative formulation}
 In Section \ref{section: linear} \-- \ref{section: wellposedness lin} where we prove the wellposedness of the linear problem, we use an alternative representation \eqref{linearized-eq} of the \textit{Landau equation}:
 \begin{equation}
  \partial_t f + v\cdot\nabla_{\!x} f = \bar{A}_g f + \bar{K}_g f 
 \label{linearized-eq}
 \end{equation}
 where $\bar{A}_g f$  consists of the terms which contain at least one momentum-derivative of $f$ while  $\bar{K}_g f$ consists of the rest as follows:
 \begin{equation}\begin{split}\label{A_g} \bar{A}_g f \eqdef &\; \partial_i\left[\left\{\phi^{ij}\!\ast[\mu+\mu^{\!1/2}g]\right\}\partial_j f  \right] \\
 & - \left\{\phi^{ij}\!\ast[v_j\mu^{1/2}g]\right\}\partial_i f 
 - \left\{\phi^{ij}\!\ast[\mu^{1/2}\partial_j g]\right\}\partial_i f \\
 =:&\; \nabla_v\cdot\big(\sigma_{\!G}\nabla_{\!v} f\big)+a_g\cdot\nabla_{\!v}f, \ \text{and}\\[3pt]
 \bar{K}_g f \eqdef &\; K f + \partial_i \sigma^i f - \sigma^{ij}v_i v_j  f \\
 & -\partial_i\left\{\phi^{ij}\!\ast[\mu^{1/2}\partial_j g]\right\}f + \left\{\phi^{ij}\!\ast[v_i\mu^{1/2}\partial_j g]\right\}f.
 \end{split}\end{equation}
 \subsubsection{Weighted equation.}
 We will also utilize the extra weights on $v$ variable when we show the $L^\infty$ decay estimates. If $f$ satisfies \eqref{linearized-eq}, then the weighted function $f^\theta\eqdef w^\theta f$ satisfies the following equation:
 \begin{equation}
 \text{ (\ref{linearized-eq})}  \xrightarrow{\cdot  w^\theta} \;
 \partial_t f^{ \theta} + v\cdot\nabla_{\!x} f^{ \theta} = \bar{A}_g^{ \theta} f^{ \theta} + \bar{K}_g^{ \theta} f  ,
 \label{weighted-eq}
 \end{equation}where the terms are defined as
 \begin{equation*}
 f^{ \theta}\eqdef w^{\theta}f,\quad w\eqdef (1+|v|),\quad (\theta\in\mathbb{R})
 \end{equation*}
 \begin{align*} 
 \bar{A}_g^{ \theta} \eqdef &\; \bar{A}_g - 2\frac{_{\partial_{i} w^{\theta}}}{^{w^{\theta}}}\sigma_{\!G}^{ij}\partial_j = \bar{A}_g - 2\frac{_{\nabla w^{\theta}}}{^{w^{\theta}}}\sigma_{\!G} \cdot\nabla_{\!v}, \\[3pt]
 \bar{A}_g^{ \theta} f =&\; \nabla_v\cdot\big(\sigma_{\!G}\nabla_{\!v} f\big) + \left(a_g-2\frac{_{\nabla w^{\theta}}}{^{w^{\theta}}}\sigma_{\!G} \right)\cdot\nabla_{\!v}f \\
 =:&\; \nabla_v\cdot\big(\sigma_{\!G}\nabla_{\!v} f\big)+a_g^{ \theta}\cdot\nabla_{\!v}f, \ \text{and} \\[5pt]
 \bar{K}_g^{ \theta} f \eqdef &\; w^{\theta}\bar{K}_g f 
  + \left(2\frac{_{\partial_i w^{\theta}\partial_j w^{\theta}}}{^{w^{2\theta}}}\sigma_{\!G}^{ij} - \frac{_{\partial_{ij} w^{\theta}}}{^{w^{\theta}}}\sigma_{\!G}^{ij} - \frac{_{\partial_{j} w^{\theta}}}{^{w^{\theta}}}\partial_i\sigma_{\!G}^{ij} - \frac{_{\partial_{i} w^{\theta}}}{^{w^{\theta}}}a_g^{ i} \right) f^{ \theta}.
 \end{align*}For some estimates, we will also utilize the \textit{homogeneous} part of the weighted equation only:
 \begin{equation}
  \partial_t h + v\cdot\nabla_{\!x}h = \bar{A}_g^{ \theta}h  .
 \label{weighted-eq-h}
 \end{equation}

 \subsection{Further difficulties and ideas}
 \subsubsection{Further difficulties of the boundary value problem.}
 In the paper \cite{2016arXiv161005346K}, the authors worked on a simplest periodic box. When generalizing it to the bounded domain case, we would encounter some major difficulties of ``getting things {\em uniform}'' as we approach the boundary. More precisely, in the proof of the $L^\infty$ estimate, we transfer the regularity from $v$ to $t$ and $x$ by an averaging lemma, which is based on some Fourier analysis methods (e.g. the Fourier transform) and only makes sense in the whole domain. Also, for the H\"{o}lder estimate, the local estimate requires a {\em full} neighborhood at each point in the bounded domain including the boundary. Additionally, the $S^p$ estimate is in general developed in the whole space as well. All of above would become technically difficult if we have to directly work on a closed domain with boundary.

 \subsubsection{Extending outside the boundary.}
 To overcome these difficulties, we consider a local-flattening of the boundary and a proper extension of the interior domain beyond the flattened boundary, so we can see that the boundary extension still keeps the modified equation in a similar form.  Then this would allow us to see that the previous arguments for the whole space case can be applied to our extended domain. 

 \subsubsection{A few principles.}
 When designing the proper extension, we cannot just choose an arbitrary ``smooth'' one. There are some guidelines we should follow:
 \begin{itemize}
 	\item[$\bullet$] First of all, the transformation-extension at the boundary is supposed to be compatible with the original specular reflection boundary condition, making it a well-defined and {\em continuous} extension in the whole space.
 	 
 	\item[$\bullet$] Secondly, the transformed equation with respect to the new variables is expected to be in the same form as the original one, at least the characteristics/transport operator should be {\em invariant} under change of coordinates.
 	 
 	\item[$\bullet$] Moreover, we should check the (H\"{o}lder) continuity of the highest-order coefficient to make sure the existing theory can apply. 
 \end{itemize}
  \subsubsection{Our strategy: the mirror extension of the locally flattened domain}In Section \ref{boundary flattening and mirror extension}, the authors implement the (local) flattening of the boundary and observe how the original equation is converted via the boundary flattening and a proper extension. It turns out that, for the specular reflection boundary condition, a reflection-type extension (or so-called the \textit{mirror} extension) allows one to enjoy a similar kinetic equation to the \textit{Landau equation} in the whole space situation, and hence one can apply previously known techniques for the \textit{Landau equation} in a whole space or in a periodic box. The authors believe that the work developed in Section \ref{boundary flattening and mirror extension} would suggest a new method of approaching to the specular-reflection boundary value problem for varied diffusive kinetic equations, including the \textit{Landau equation}, the \textit{Fokker-Planck equation}, and the \textit{Boltzmann equation} without \textit{angular cutoff}.
 \section{Linear problem}\label{section: linear}
 As the start of the study of the Landau boundary value problem, we first construct a weak solution of the linearized problem.

 \subsection{The Landau initial-boundary value problem}
 \subsubsection{A homogeneous equation}Since the operator $\bar{A}_g f$  contains all terms with the derivatives of  $f$, we will first consider constructing weak solutions for the homogeneous equation
 \begin{equation} \label{linearized-eq-homo}
 \partial_t f + v\cdot\nabla_{\!x}f = \bar{A}_g f 
 \end{equation}
 with the same initial condition (\ref{initial}) and the following boundary condition, as introduced in Section \ref{section: introduction} above \eqref{Eq : Duhamel principle}.

 \subsubsection{Specular-reflection boundary condition.}
 We can rewrite the boundary condition~(\ref{specular}) as
 \begin{equation*}
 \gamma_- f = \mathcal{R}\left[\gamma_+ f\right] 
 \end{equation*} 
 in terms of the specular-reflection operator $\mathcal{R}$ defined by
 \begin{equation*}
 \mathcal{R}\!\left[\gamma_+ f\right] (t,x,v) = \gamma_+ f(t,x,R_x v)
 \end{equation*}
 for any $(t,x,v)\!\in\!\Sigma_{-}^{ T}$.

 \subsubsection{Basic estimates.}We would like to remark that (\ref{linearized-eq}) is in the form of a class of kinetic Fokker-Planck equations (also known as hypoelliptic   or ultraparabolic   of Kolmogorov type) with rough coefficients  (cf. \cite{golse2016harnack}): 
 \begin{equation*}
 \partial_t f + v\cdot\nabla_{\!x} f = \nabla_v\cdot(\mathbf{A} \nabla_{\!v} f ) + \mathbf{B}\cdot\nabla_{\!v}f + \mathbf{C}f
 \end{equation*}
 When  $\|g\|_{\infty}$  is sufficiently small, we have the following results regarding the coefficients:

 $\mathbf{A} \eqdef  \sigma_{\!G}(t,x,v)$  is a  $3\times3$  non-negative matrix (but \textit{not uniformly} elliptic in $v$)  satisfying
 \;(cf. Lemma 2.4 in \cite{2016arXiv161005346K})
 $$0 < (1\!+\!|v|)^{-3} \mathbf{I}  \lesssim  \mathbf{A}(v)  \lesssim  (1\!+\!|v|)^{-1} \mathbf{I} .$$  
 
 $\mathbf{B} \eqdef  a_g(t,x,v)$  is a uniformly bounded $3$-dimensional vector with \;(cf. Appendix A in \cite{golse2016harnack}) 
 $$\|\mathbf{B}[g]\|_{\infty}  \lesssim  \|g\|_{\infty}^{2/3}  \ll  1 .$$
 
 $\mathbf{C} \eqdef  \bar{K}_g$  is an operator bounded on $L^\infty$  with  $\|\mathbf{C}\|_{L^\infty \rightarrow L^\infty}\lesssim 1$. \;(cf. Lemma 2.9 in \cite{2016arXiv161005346K})

 \subsection{Notion of weak solutions}
 We first give the notion of weak solutions to the Landau initial-boundary value problem (\ref{linearized-eq-homo}), (\ref{initial}), and (\ref{specular}).
 \begin{definition} [Weak solutions] \label{Def:weak-sol}
 	$f \in L^\infty \!\left([0,T]  ; L^1\cap L^\infty (\Omega \times \mathbb{R}^3) \right)$ is a weak solution of the Landau initial-boundary value problem (\ref{linearized-eq-homo}), (\ref{initial}), and (\ref{specular}) if for any test function $\psi \in C_{t,x,v}^{1,1,2}\left((0,T)\times\Omega \times \mathbb{R}^3\right) \cap  C\!\left([0,T]\times\bar{\Omega} \times \mathbb{R}^3\right) $ such that $\psi(t)$ is compactly supported in $\bar{\Omega}\!\times\!\mathbb{R}^3$ for all $t\in [0,T]$ and with the dual specular reflection boundary condition
 	\begin{equation} \label{dual-specular-BC}
 	\begin{array}{rcl} 
 	& \gamma_+ \psi = \mathcal{R^*}\!\left[\gamma_- \psi\right] & \\[5pt]
 	\text{i.e.,}\; & \gamma_+ \psi(t,x,v) = \gamma_- \psi(t,x,R_x v) & \quad \forall\; (t,x,v)\in\Sigma_{+}^{ T},
 	\end{array}
 	\end{equation}
 	it satisfies that the function
 	\begin{equation*}
 	t  \mapsto \iint_{\Omega\times\mathbb{R}^3} f(t,x,v)\psi(t,x,v) dxdv
 	\end{equation*}
 	is continuous on $[0,T]$, and for every  $t\in [0,T]$,
 	\begin{equation} \label{Def:weak-formulation}
 	\begin{split}
 	& \iint_{\Omega\times\mathbb{R}^3} f(t,x,v) \psi(t,x,v) dxdv  - \iint_{\Omega\times\mathbb{R}^3} f_0(x,v) \psi(0,x,v) dxdv \\[3pt]
 	& = \int_0^t\! \iint_{\Omega\times\mathbb{R}^3} f(\tau,x,v) \Big[ \partial_t \psi + v\cdot\nabla_{\!x}\psi - \nabla_v\cdot\big(a_g \psi \big) + \nabla_v\cdot\big(\sigma_{\!G}\nabla_{\!v}\psi \big) \Big] dxdvd\tau. 
 	\end{split}
 	\end{equation}
 \end{definition}

 \subsection{Main theorem}
 The main theorem of Section \ref{sec: reg}--\ref{section: wellposedness lin} concerns the global well-posedness of the weak solution to (\ref{linearized-eq-homo}), (\ref{initial}), and (\ref{specular}).
 \begin{theorem} [Well-posedness of weak solutions] \label{main-thm}
 	Let  $T\!>\!0$  and  $f_0\in L^1\cap L^\infty (\Omega \times \mathbb{R}^3)$  be given. Then there exists a unique weak solution $f \in L^\infty \!\big([0,T]  ; $\newline$L^1\cap L^\infty (\Omega \times \mathbb{R}^3) \big)$ of the linearized \textit{Landau equation} with the specular-reflection boundary condition (\ref{linearized-eq-homo}), (\ref{initial}), and (\ref{specular}) in the sense of Definition \ref{Def:weak-sol}. Moreover, it satisfies the $L^\infty$ bound
 	\begin{equation*}
 	\|f(t)\|_{L^\infty (\Omega \times \mathbb{R}^3)}  \leq  \|f_0\|_{L^\infty (\Omega \times \mathbb{R}^3)}
 	\end{equation*}
 	and the $L^1$ bound
 	\begin{equation*}
 	\|f(t)\|_{L^1 (\Omega \times \mathbb{R}^3)}  \leq  \|f_0\|_{L^1 (\Omega \times \mathbb{R}^3)}
 	\end{equation*}
 	for each $t\in [0,T]$.
 \end{theorem}

 \subsection{Ideas for the linear problem}
 Our goal is to establish a well-posedness theory for the Landau initial-boundary value problem, in particular, the existence of weak solutions to the linearized \textit{Landau equation}. For that purpose, we adapt the mechanism of \textit{a priori }estimates for weak solutions, followed by a limiting process based on the compactness and weak convergence of approximate solutions. More precisely, we will follow the steps outlined below: 
 \begin{enumerate}
 	\item Design a sequence of regularized approximate problem and solve the problem by using the method of characteristics combined with a standard fixed-point argument.
 	\item Derive the uniform $L^\infty$ and $L^1$ estimates for weak solutions of the approximate problem by studying its adjoint problem to exploit the maximum principle property.
 	\item Pass to the limit in the approximate problem based on the weak compactness, which is ensured by the corresponding uniform boundedness.
 \end{enumerate}

 \section{Regularization \& approximation}\label{sec: reg}
 The first step for constructing a weak solution is to regularize the problem (\ref{linearized-eq-homo}), (\ref{initial}), and (\ref{specular}), so that the resulting approximate problem is relatively easier to solve, allowing us to find a sequence of approximate solutions.

 \subsection{Regularization of the problem}
\subsubsection{ Parametrization near the boundary.} 
 We follow the parametrization introduced in Section 3.2 of \cite{MR3788197}.
 Suppose we have defined the local (flattening) charts near the boundary $\big\{\varphi_i  ;  \Omega\cap B_{\delta_i}(x_i)\big\}$. 
 Since  $\partial\Omega$  is compact, we can choose \newline  $x_1,x_2,\cdots,x_n \in\partial\Omega$  such that  $\partial\Omega \subset \bigcup_{i=1}^{ n} B_{\delta_i}(x_i)$. Take  $\delta_0 >0$  to be the maximum value satisfying
 \begin{equation} \label{delta_0}
 \big\{ x\!\in\!\bar{\Omega}:{\rm dist}(x,\partial\Omega) < \delta_0 \big\}  \subset  \bigcup_{i=1}^n B_{\delta_i}(x_i) .
 \end{equation}
 Then divide $\bar{\Omega}$ into two parts. Let $\Omega_{{\rm bd}}$ be a tubular region near the boundary where the flattening coordinates is well-defined:
 \begin{align*}
 \Omega_{{\rm bd}} &\eqdef   \big\{ x\!\in\!\bar{\Omega}:{\rm dist}(x,\partial\Omega) < \delta_0 \big\}, \\[3pt]
 \Omega_{{\rm in}} &\eqdef   \bar{\Omega} \backslash \Omega_{{\rm bd}} .
 \end{align*}
 For $(x,v)\in \Omega_{{\rm bd}}\times\mathbb{R}^3_v$, define
 \begin{align} \label{x-v-perp}
 \begin{split}
 x_\perp &\eqdef   {\rm dist}(x,\partial\Omega) , \\[2pt]
 v_\perp &\eqdef   v\cdot n_{\hat{x}},
 \end{split}
 \end{align}
 where $\hat{x}\in\partial\Omega$  is the boundary point closest to $x$, and $n_{\hat{x}}$ is the outward unit normal vector at $\hat{x}\in\partial\Omega$. 
 Let us first fix $x\in \partial\Omega$. In a small neighborhood $\Omega\cap B_{\delta_{x}}(x)$, we take a local coordinate $x_{\parallel}(\mu_1,\mu_2)$ and denote the inward unit normal vector at each boundary point $x_{\parallel}(\mu_1,\mu_2)$ by $N(\mu_1,\mu_2)$. Then we choose $\delta_x>0$ small enough such that $$x=x_{\parallel}(\mu_1,\mu_2)+x_\perp N(\mu_1,\mu_2),\ v=w_1u_1+w_2u_2+v_\perp N(\mu_1,\mu_2),$$ are well-defined in $(\Omega\cap B_{\delta_x}(x))\times \rth$ where $u_i\eqdef \frac{\partial x_{\parallel}}{\partial \mu_i}(\mu_1,\mu_2)$ forms an orthogonal basis of the tangent plane $T_{x_\parallel}\partial\Omega$ at each $x_\parallel(\mu_1,\mu_2)$.
 
 Then the singular set of the phase-boundary can also be represented in terms of $x_\perp$ and $v_\perp$ by
 $$\gamma_0 \eqdef  \left\{ (x,v)\in\bar{\Omega}\times\mathbb{R}^3 :  x_\perp = v_\perp =0 \right\}.$$
 
 \subsubsection{Cutoff away from the grazing set for the transport term.}
 We first need to regularize the transport term  $v\cdot\nabla_{\!x}f$  to avoid possible obstacles coming from the singular boundary set (e.g. grazing collisions), especially in the analysis of the adjoint problem (see the proof of Lemma~\ref{Lem:adj-pb}).
 
 Let  $\varepsilon\in \left(0,\min\{\delta_0,\frac{1}{2}\} \right)$ be a small regularization parameter, where $\delta_0$ is given in (\ref{delta_0}), be fixed. Recall from (\ref{x-v-perp}) that $x_\perp$ and $v_\perp$ represent the normal components of $x$ and $v$, respectively, when $x\in \Omega_{{\rm bd}}$. We choose a nonnegative cutoff function  $\lambda_\varepsilon \in C^\infty(\mathbb{R})$  away from the origin which is monotonous on $(-\infty,0)$ and $(0,\infty)$: 
 \begin{equation*} 
 \lambda_\varepsilon(s) 
  =  \left\{ 
 \begin{array}{ll}
 \;0,&\quad \text{if }\;  |s|\leq \varepsilon^4 \\[5pt]
 \text{smooth},&\quad \text{if }\;  \varepsilon^4 <|s|< 2\varepsilon^4 \\[5pt]
 \;1,&\quad \text{if }\;  |s|\geq 2\varepsilon^4.
 \end{array}\right.
 \end{equation*}
 For $(x,v)\in \bar{\Omega}\times\mathbb{R}^3_v$, define
 \begin{equation*} 
 \beta_\varepsilon(v) 
  =  \left\{ 
 \begin{array}{ll}
 \lambda_\varepsilon(v_\perp) v,
 &\quad \text{if }\; x\in\Omega_{{\rm bd}} \\[5pt]
 v,
 &\quad \text{if }\; x\in\Omega_{{\rm in}}
 \end{array}\right.
 \end{equation*}
 and
 \begin{align*} 
 \eta_\varepsilon(x) 
 & =  \left\{ 
 \begin{array}{ll}
 \lambda_\varepsilon\big(x_{\!\perp}^{ 4}\big),
 &\quad \text{if }\; x\in\Omega_{{\rm bd}} \\[5pt]
  1,
 &\quad \text{if }\; x\in\Omega_{{\rm in}}
 \end{array}\right. \\[5pt]
 & =  \left\{ 
 \begin{array}{ll}
 \;0,&\quad \text{if }\;  x_{\!\perp} \leq \varepsilon \\[5pt]
 \text{smooth},&\quad \text{if }\;  \text{else} \\[5pt]
 \;1,&\quad \text{if }\;  x_{\!\perp} \geq 2^{1/4}\varepsilon \;\; \text{or}\;\; x\in\Omega_{{\rm in}}.
 \end{array}\right.
 \end{align*}
 
 We make a smooth truncation away from the grazing set by replacing the term  $v\cdot\nabla_{\!x}f$  with 
 $$\big\{ \beta_\varepsilon(v) + [v\!-\!\beta_\varepsilon(v)]  \eta_\varepsilon(x) \big\} \cdot\nabla_{\!x}f,$$
 which formally converges to  $v\cdot\nabla_{\!x}f$  as  $\varepsilon$  approaches  $0$.

 \subsubsection{``Discretization'' of the diffusion term.}
 To handle the second-order differential operator, we first diagonalize (standardize) the coefficient-matrix $\mathbf{A}\eqdef$\newline$ \sigma_{\!G}(t,x,v)$  by making a change of variables. Since $\mathbf{A}$ is positive definite, it is congruent to the identity matrix. Then we design a discretized approximation of the Laplacian as a difference-quotient in the integral form:
 \begin{equation*}
 Q^\varepsilon_{\!\Delta}[f](t,x,v) \eqdef   \frac{2}{\varepsilon^2} \int_{\mathbb{R}^3}\! \big[f(t,x,v\!+\!\varepsilon u)-f(t,x,v)\big] \xi(u) du,
 \end{equation*}
 where  $\xi(u)\in C_c^\infty(\mathbb{R}^3)$  is a nonnegative bump function satisfying
 \begin{align*}
 & \int_{\mathbb{R}^3} \xi(u) du  =  1 , \\[2pt]
 & \int_{\mathbb{R}^3} u_{i }\xi(u) du  =  0 ,\   (i=1,2,3) , \\[2pt]
 & \int_{\mathbb{R}^3} u_iu_{j }\xi(u) du  =  \delta_{ij}  ,\  (i,j=1,2,3) .
 \end{align*}
 We may construct such  $\xi(u)$  by setting
 \begin{equation*}
 \xi(u) = \xi(u_1,u_2,u_3) \eqdef   \xi_1(u_1) \xi_2(u_2) \xi_3(u_3)
 \end{equation*}
 with each  $\xi_i(u_i)\in C_c^\infty(\mathbb{R})$  being a nonnegative even bump function and for each  $i\in \{1,2,3\}$,
 \begin{align*}
 & \int_{\mathbb{R}} \xi_i(u_i) du_i  =  1 , \\[3pt]
 & \int_{\mathbb{R}} u_{i }\xi_i(u_i) du_i  =  0 , \\[3pt]
 & \int_{\mathbb{R}} u_i^{ 2 } \xi_i(u_i) du_i  =  1 .
 \end{align*}
 
 From Taylor's formula, we can see that  $Q^\varepsilon_{\!\Delta}[f]$ formally converges to our diffusion term  $\nabla_{\!v}\!\cdot\!\big(\sigma_{\!G}\nabla_{\!v} f\big)$  up to a congruence transformation. For simplicity, we may adapt the notation $Q^\varepsilon[f]$  and assume that 
  $Q^\varepsilon[f]  \rightarrow  \nabla_{\!v}\!\cdot\!\big(\sigma_{\!G}\nabla_{\!v} f\big)$  as  $\varepsilon \rightarrow 0$.
 
 It is worth pointing out that with the discretized diffusion operator $Q^\varepsilon[f]$  defined in this way, we are able to follow the method of characteristics to construct an approximate solution and show the existence by the fixed-point argument. It is also convenient to carry out the ``transfer of derivatives'' from one factor to the other with the integral representation.
 
 Furthermore, the subtlety of this approximate operator lies in that, on one hand, it preserves some ``positivity/ellipticity'' of the Laplacian, which plays a key role in the derivation of the maximum principle and  $L^\infty$ estimate; on the other hand, its mean-zero property (the fact that  $\iint_{\Omega\times\mathbb{R}^3} \bar{Q}^\varepsilon_{\!\Delta}[\psi] dxdv = 0$, an inheritance feature of the original elliptic operator in divergence form) also guarantees the $L^1$ estimate of the solutions; see the proof of Lemma \ref{Lem:adj-pb}.

 \subsubsection{Modified specular reflection boundary condition: a Dirichlet boundary condition} 
 In order to get the trace estimates to control the ``boundary term'' in our construction (see the proof of Lemma \ref{existence:mild-sol}), we have to modify the original problem by a Dirichlet boundary condition (or so-called \textit{inflow} boundary condition) $\gamma_- f^n = \alpha_{{\rm r}} \mathcal{R}\left[\gamma_+ f^{n-1}\right]$ for a sequence of approximate solutions $f^n$ with $n \in \mathbb{N}$ and  $\alpha_{{\rm r}} \!<\!1$, because for the specular-reflection case ($\alpha_{{\rm r}} \!=\!1$), the boundary contributions of $\gamma_-$ and $\gamma_+$ will cancel each other out in the estimates.
 
 Let $n\in \mathbb{N}$ and $a\in (0,1)$ be fixed, and approximate the specular reflection boundary condition by a modified one:
 \begin{equation*}
 \gamma_- f^n = (1\!-\!a)  \mathcal{R}\left[\gamma_+ f^{n-1}\right],
   \forall\; (t,x,v)\in \Sigma_{-}^{ T}, \ \text{if}\ n\geq 2,
 \end{equation*}and $\gamma_-f^{1}=0$.

 \subsubsection{The approximate problem.}
 To sum up, we choose and fix $\varepsilon>0$, $n\in \mathbb{N}$ and $a\in(0,1)$. Then we consider the approximate (regularized) equation:
 \begin{equation} \label{app-eq}
 \partial_t f^{\varepsilon,a,n} + \big\{ \beta_\varepsilon(v) + [v\!-\!\beta_\varepsilon(v)]  \eta_\varepsilon(x) \big\} \cdot\nabla_{\!x}f^{\varepsilon,a,n} - \mathbf{B}\cdot\nabla_{\!v}f^{\varepsilon,a,n} 
 = Q^\varepsilon[f^{\varepsilon,a,n}] 
 \end{equation}
 with the same initial condition
 \begin{equation} \label{app-initial}
 f^{\varepsilon,a,n}(0,x,v) = f_0(x,v),
 \end{equation} 
 and the modified specular reflection boundary condition
 \begin{equation} \label{modified-BC}
 \gamma_- f^{\varepsilon,a,n} = (1\!-\!a)  \mathcal{R}\left[\gamma_+ f^{\varepsilon,a,n-1}\right],\ \text{if}\ n\geq 2,
 \end{equation} and $\gamma_-f^{\varepsilon,a,1}=0$. 
 
 Note that the approximate equation (\ref{app-eq}) is essentially a transport equation combined with the ``jump process'' $Q^\varepsilon[f]$, where the transport term is truncated in a small neighborhood of the grazing set whose area is of $O\big(\varepsilon^5\big)$.
 Also, the approximate problem (\ref{app-eq}) \-- (\ref{modified-BC}) formally converges to the original problem (\ref{linearized-eq-homo}) \-- (\ref{specular}) as both  $\varepsilon$  and  $a$  go to  $0$ and $n\rightarrow \infty$.
 
 In the following sections we will first construct (weak) solutions  $F \eqdef  f^{\varepsilon,a,n}$  to the approximate problem and obtain uniform \textit{a priori }estimates for those approximate solutions.

 \subsection{Method of characteristics}
 Since the left-hand side of  (\ref{app-eq}) is a first-order differential operator, we can solve the approximate problem (\ref{app-eq}) \-- (\ref{modified-BC}) by the method of characteristics.

 \subsubsection{Characteristics.}
 Denote by $$\big( T(s),X(s),V(s) \big) \eqdef  \big( T(s;t,x,v),X(s;t,x,v),V(s;t,x,v) \big)$$ a solution of the (backward) characteristic equations for  $s\leq t$:
 
 First, solving $$\frac{d T(s)}{ds} = 1, \quad T(t)=t$$ gives  $T(s)=s$, and so the equations of  $\big(X(s),V(s)\big)$ read
 \begin{align}
 & \frac{d X(s)}{ds}  =  \beta_\varepsilon\big(V(s)\big) + \big[V(s) - \beta_\varepsilon\big(V(s)\big)\big]  \eta_\varepsilon\big(X(s)\big),\quad 
 X(t) = x  \label{characteristics-X} \\[5pt]
 & \frac{d V(s)}{ds}  =  -\mathbf{B}\big(s,X(s),V(s)\big),\quad
 V(t) = v  \label{characteristics-V}
 \end{align}

 \paragraph{$\bullet$ \textit{The backward stopping-time.}}
 Due to the cutoff function away from the singular boundary and the modified specular reflection boundary condition,  it could be too complicated to write out the solution of the backward characteristics explicitly. 
 To get around it, for a given  $(t,x,v)$, we define  $t_0 = t_0(t,x,v) \in [0,T]$  to be
 \begin{enumerate}
 	\item[(i)] The time when the particle starting with $(t,x,v)$  first hits the boundary  $\partial\Omega$  along the backward characteristics: \\[3pt]
 	$t_0 \eqdef  \max\Big\{ \tau\in (0,t ]: X(\tau)= x - \int_\tau^t \big\{\beta_\varepsilon\big(V(s)\big) \\ + \big[V(s) - \beta_\varepsilon\big(V(s)\big)\big] \eta_\varepsilon\big(X(s)\big) \big\} ds \in \partial\Omega \Big\}$,  if it exists.
 	 
 	\item[(ii)] Otherwise (if there is no such value i.e.,the particle never hits the boundary $\partial\Omega$  back in time until $\tau\!=\!0$ ), then set  $t_0=0$.
 \end{enumerate}

 \paragraph{$\bullet$ \textit{The Jacobian.}}
 We next record and calculate the Jacobian $J_\mathcal{C}(s ;t)$  for  $t_0\!<\!s\!<\!t$  of the transformation
 $$\mathcal{C}:\;  \big(X(t),V(t)\big) = (x,v) \;\mapsto\; \big(X(s),V(s)\big) ,$$
 which is given by 
 \begin{equation*}
 J_\mathcal{C}(s ;t) \eqdef   \frac{\partial\big(X(s),V(s)\big)}{\partial\big(X(t),V(t)\big)}
  =  \det \left(
 \begin{array}{c|c}
 \frac{\partial X(s)}{\partial x}  &  \frac{\partial X(s)}{\partial v} \\[3pt]
 \hline \\[-10pt]
 \frac{\partial V(s)}{\partial x}  &  \frac{\partial V(s)}{\partial v}
 \end{array}
 \right) ,
 \end{equation*}
 where $\big(X(s),V(s)\big)$ is the characteristics defined in (\ref{characteristics-X}) and (\ref{characteristics-V}).
 
 Geometrically,  $J_\mathcal{C}(s ;t)$  measures the rate of change of the unit volume in the phase space along the characteristics, and it is bounded near $1$ in terms of  $\varepsilon$  and  $T$, as shown in the following lemma.
 
 \begin{lemma} [Estimate of the Jacobian] \label{Lem:jacobian-est}
 	Let  $s\in (t_0,t)\subset [0,T]$, and  $(x,v)\in \Omega \times\mathbb{R}^3$  be given. Then the Jacobian $J_\mathcal{C}(s ;t)$  satisfies the estimates
 	\begin{equation} \label{jacobian-est-1}
 	e^{-CT} \leq  e^{-C(1+\varepsilon^3)|t-s|}  \leq  \big|J_\mathcal{C}(s ;t)\big|  \leq  e^{C(1+\varepsilon^3)|t-s|}  \leq  e^{CT},
 	\end{equation}
 	and 
 	\begin{equation} \label{jacobian-est-2}
 	1-O(T)  \leq  \big|J_\mathcal{C}(s ;t)\big|  \leq  1+O(T) ,
 	\end{equation}
 	where  $O(T) = CT e^{CT}$, for some  $C>0$  independent of  $x, v, T$ and  $\varepsilon$.
 \end{lemma}
 
 \begin{proof}
 	We denote the right-hand side of  (\ref{characteristics-X}) by  $\mathbf{W_{\!\varepsilon}}\big(X(s),V(s)\big)$  for brevity, where
 	$$\mathbf{W_{\!\varepsilon}}(x,v) \eqdef   \beta_\varepsilon(v) + [v\!-\!\beta_\varepsilon(v)]  \eta_\varepsilon(x) .$$
 	Differentiating the characteristic equations (\ref{characteristics-X}) and (\ref{characteristics-V}) with respect to $(x,v)$  by the chain rule yields a first-order system of homogeneous ODEs of dimension six
 	\begin{equation} \label{system-ODE}
 	\frac{d}{ds} \left(
 	\begin{array}{c|c}
 	\frac{\partial X(s)}{\partial x}  &  \frac{\partial X(s)}{\partial v} \\[3pt]
 	\hline \\[-10pt]
 	\frac{\partial V(s)}{\partial x}  &  \frac{\partial V(s)}{\partial v}
 	\end{array}
 	\right)
 	= \left. \left(
 	\begin{array}{c|c}
 	 &  \\[-12pt]
 	\nabla_{\!x}\mathbf{W_{\!\varepsilon}}  &  \nabla_{\!v}\mathbf{W_{\!\varepsilon}} \\[3pt]
 	\hline \\[-10pt]
 	-\nabla_{\!x}\mathbf{B}  &  -\nabla_{\!v}\mathbf{B} \\[-12pt]
 	 & 
 	\end{array}
 	\right) \!\right\vert_{\left(X(s),V(s)\right)}
 	\!\!\!\left(
 	\begin{array}{c|c}
 	\frac{\partial X(s)}{\partial x}  &  \frac{\partial X(s)}{\partial v} \\[3pt]
 	\hline \\[-10pt]
 	\frac{\partial V(s)}{\partial x}  &  \frac{\partial V(s)}{\partial v}
 	\end{array}
 	\right) ,
 	\end{equation}
 	where we call the coefficient-matrix  $\mathbf{M}(s)$.
 	
 	Observe that the Jacobian matrix is a matrix-valued solution of (\ref{system-ODE}) on $[t_0,t]$. 
 	By applying the Liouville's formula, its determinant $J_\mathcal{C}(s ;t)$  satisfies the identity
 	\begin{equation} \label{Liouville-formula}
 	J_\mathcal{C}(s ;t)  =  J_\mathcal{C}(t ;t)  \exp\!\left(\int_t^s{\rm tr}\left[\mathbf{M}(\tau)\right]d\tau \right) ,
 	\end{equation}
 	where  $J_\mathcal{C}(t ;t) = \det( \mathbf{I}_{ 6\times 6}) =1$, noting that  $\big(X(t),V(t)\big) = (x,v)$.
 	
 	Now it remains to compute and estimate the trace of  $\mathbf{M}(s)$, the sum of its diagonal entries:
 	\begin{align*}
 	{\rm tr}\left[\mathbf{M}(\tau)\right]
 	&  =  \big[\nabla_{\!x}\cdot\mathbf{W_{\!\varepsilon}} - \nabla_{\!v}\cdot\mathbf{B}\big] \big(\tau,X(\tau),V(\tau)\big) \\[2pt]
 	&  =  \big[V(\tau) - \beta_\varepsilon\big(V(\tau)\big)\big]\cdot \nabla_{\!x}\eta_\varepsilon\big(X(\tau)\big) \\
 	&  \; -\big(\nabla_{\!v}\cdot\mathbf{B}\big) \big(\tau,X(\tau),V(\tau)\big)
 	\end{align*}
 	For the first term, since it survives only near the grazing set and $\eta_\varepsilon$ depends only on the normal component of  $X(\tau)$, 
 	\begin{align*}
 	\Big| \big[V(\tau) - \beta_\varepsilon\big(V(\tau)\big)\big]\cdot \nabla_{\!x}\eta_\varepsilon\big(X(\tau)\big) \Big|
 	&  =  \big[v\!-\!\beta_\varepsilon(v)\big]_{\perp} \big(V(\tau)\big) 
 	\left|\frac{\partial\eta_\varepsilon}{\partial x_{\!\perp}} \big(X(\tau)\big) \right| \\[3pt]
 	&  =  O(\varepsilon^4)  O(1/\varepsilon)  =  O(\varepsilon^3) .
 	\end{align*}
 	Also the second term is uniformly bounded recalling the definition of  $\mathbf{B}\eqdef a_g(t,x,v)$  in (\ref{A_g}). 
 	So we get (\ref{jacobian-est-1}) from (\ref{Liouville-formula}). 
 	
 	Finally, (\ref{jacobian-est-2}) follows from (\ref{jacobian-est-1}) with Taylor's expansion of the exponential function.
 \end{proof}

 \subsubsection{Mild solutions.}
 For a fixed $(\varepsilon,a,n)$, let  $F \eqdef  f^{\varepsilon,a,n}$  denote the approximate solution. Also, denote the given function $G=f^{\varepsilon,a,n-1}$ for each fixed $(\varepsilon,a,n)$. With the definition of the characteristics,  (\ref{app-eq}) can be rewritten as an ODE (via chain rule and (\ref{characteristics-X}) and (\ref{characteristics-V}))
 \begin{equation*}
 \frac{d}{ds}  F\big(s,X(s),V(s)\big)  =  Q^\varepsilon[F]\big(s,X(s),V(s)\big) .   
 \end{equation*}
 Integrating this equation from $t_0$ to $t$ in $s$  along the characteristics, we give the following definition of the mild solutions represented by an explicit formula. 
 
 \begin{definition} [Mild solutions to approximate problem] \label{Def:mild-sol-app-pb}
 	 $F \eqdef f^{\varepsilon,a,n} \in L^\infty \big([0,T]  ;$\newline $  L^\infty (\Omega \times \mathbb{R}^3) \big)$ is a mild solution of the approximate problem (\ref{app-eq}) \-- (\ref{modified-BC}) if it satisfies that for all  $t\in [0,T]$,
 	\begin{align} \label{mild-sol-integral-formula}
 	\begin{split}
 	F(t,x,v) & =  \bar{f_0}\big(X(t_0),V(t_0)\big)  + \int_{t_0}^t Q^\varepsilon[F]\big(s,X(s),V(s)\big)ds \\ 
 	& =:  \mathcal{T}[F] (t,x,v),
 	\end{split}
 	\end{align}
 	where
 	\begin{align*} 
 	& \bar{f_0}\big(X(t_0),V(t_0)\big) \\[7pt]
 	=& \left\{ 
 	\begin{array}{ll}
 	& f_0\big(X(0),V(0)\big), 
 	\ \text{if}\;\;\; t_0=0 \;\;\text{and}\;\; X(0)\in\Omega \\[7pt]
 	& \gamma_- F\big(t_0,X(t_0),V(t_0)\big) = (1\!-\!a) \gamma_+ G\big(t_0,X(t_0),R_{X(t_0)}V(t_0)\big), 
 	\ \text{otherwise.}
 	\end{array}\right.
 	\end{align*}
 \end{definition}

 \subsection{Construction of approximate solutions}

 Now we aim to construct weak solutions to the approximate problem by first showing that a mild solution in the sense of Definition~\ref{Def:mild-sol-app-pb}  indeed exists. We therefore need to deduce some \emph{a priori} estimates concerning the trace of the approximate solutions, especially in view of the expression for  $\bar{f_0}\big(X(t_0),V(t_0)\big)$  in the formula (\ref{mild-sol-integral-formula}).
 
 Thanks to the modified specular reflection boundary condition, we are able to bound the trace term by performing a formal procedure of the  $L^p$ estimate and finally letting  $p$  go to $\infty$. In fact, the required regularity here is already assured when used in the context of the fixed-point argument.
 
 \begin{lemma} [Trace estimates] \label{Lem:trace-est}
 	For each $n\in \mathbb{N}$, the solution  $f^{\varepsilon,a,n}$ to the approximate problem (\ref{app-eq}) \-- (\ref{modified-BC}) has trace values satisfying the estimate
 	\begin{equation}\label{trace-est-L^p}||\gamma_\pm f^{\varepsilon,a,n}||_{L^p(\Sigma_+)}
 \leq e^{\left(\frac{\nabla_{\!v}\cdot\mathbf{B}}{p}  +  \frac{4}{\varepsilon^2}\right) T}  \|f_0\|_{L^p (\Omega \times \mathbb{R}^3)}\left(\frac{1-(1-a)^{pn}}{1-(1-a)^p}\right)^{\frac{1}{p}},
 \end{equation} and
 	\begin{equation} \label{trace-est-L^infty}||\gamma_\pm f^{\varepsilon,a,n}||_{L^\infty(\Sigma_+)}
 	\leq e^{\left( \frac{4}{\varepsilon^2}\right) T}  \|f_0\|_{L^\infty (\Omega \times \mathbb{R}^3)}.
 	\end{equation} 
 \end{lemma}
 
 \begin{proof}  
 	Multiply the approximate equation (\ref{app-eq}) by  $p (f^{\varepsilon,a,n})^{ p-1}$, integrate over  $\Omega \times \mathbb{R}^3$, and then use the divergence theorem (i.e.,integration by parts), we get
 	\begin{align*} 
 	& \frac{d}{dt} \|f^{\varepsilon,a,n}(t)\|_{L^p (\Omega \times \mathbb{R}^3)}^{ p} + \int_{\gamma_+} \big(\gamma_+ f^{\varepsilon,a,n}\big)^p |\beta_\varepsilon(v)\!\cdot\! n_x|dS_xdv \\& \ \ \ \  \ \ \ \ \ \ \ \ \ -\! \int_{\gamma_-} \big(\gamma_- f^{\varepsilon,a,n}\big)^p |\beta_\varepsilon(v)\!\cdot\! n_x|dS_xdv\\
 	& \leq\;  (\nabla_{\!v}\cdot\mathbf{B}) \|f^{\varepsilon,a,n}(t)\|_{L^p (\Omega \times \mathbb{R}^3)}^{ p}  +  \frac{4p}{\varepsilon^2}  \|f^{\varepsilon,a,n}(t)\|_{L^p (\Omega \times \mathbb{R}^3)}^{ p} 	.
 	\end{align*}Then the Gr\"onwall inequality yields that we have
 	\begin{multline*}
 	||f^{\varepsilon,a,n}(t)||^{p}_{L^p(\Omega\times \rth)} +||\gamma_+f^{\varepsilon,a,n}||^p_{L^p(\Sigma_+)}\\
 	\leq ||\gamma_-f^{\varepsilon,a,n}||^p_{L^p(\Sigma_-)}+e^{\left(\nabla_{\!v}\cdot\mathbf{B} +  \frac{4p}{\varepsilon^2}\right) T} ||f^{\varepsilon,a,n}(0)||^{p}_{L^p(\Omega\times \rth)},\end{multline*} for each $t\in [0,T]$. 
We note that $||\gamma_- f^{\varepsilon,a,n}||^p_{L^p(\Sigma_-)}=(1-a)^p||\gamma_+ f^{\varepsilon,a,n-1}||^p_{L^p(\Sigma_+)}$. Then, by induction, we further have that
 	$$||\gamma_+f^{\varepsilon,a,n}||^p_{L^p(\Sigma_+)}\leq e^{\left(\nabla_{\!v}\cdot\mathbf{B}  +  \frac{4p}{\varepsilon^2}\right) T}  \|f_0\|^p_{L^p (\Omega \times \mathbb{R}^3)}\left(\frac{1-(1-a)^{pn}}{1-(1-a)^p}\right).
 	$$ Thus, we have
 	$$||\gamma_+f^{\varepsilon,a,n}||_{L^p(\Sigma_+)}
 		\leq e^{\left(\frac{\nabla_{\!v}\cdot\mathbf{B}}{p}  +  \frac{4}{\varepsilon^2}\right) T}  \|f_0\|_{L^p (\Omega \times \mathbb{R}^3)}\left(\frac{1-(1-a)^{pn}}{1-(1-a)^p}\right)^{\frac{1}{p}}.
$$ This proves the $L^p$ bounds of the traces for $1\leq p<\infty$. Since $0<1-a<1$, we can further pass to the limit $p\rightarrow \infty$ and obtain
 		$$||\gamma_+f^{\varepsilon,a,n}||_{L^\infty(\Sigma_+)}
 	\leq e^{\left( \frac{4}{\varepsilon^2}\right) T}  \|f_0\|_{L^\infty (\Omega \times \mathbb{R}^3)}.
 $$ This proves (\ref{trace-est-L^infty}) as $||\gamma_- f^{\varepsilon,a,n}||_{L^p(\Sigma_-)}<||\gamma_+ f^{\varepsilon,a,n-1}||_{L^p(\Sigma_+)}$ for any $1\leq p\leq \infty$ and $n\geq 2$.
 \end{proof}
 As a consequence, we can see that the $L^\infty$ norm of the traces for $f^{\varepsilon, a, n}$ is bounded independently of $n$.
As we can see from the proof above, only when  $p=\infty$  is the bounding coefficient controlled by  $T$, otherwise it could blow up as  $n\rightarrow \infty$  and $a\rightarrow 0$ in any short period of time. This observation suggests that we can work in a space with the $L^\infty$ norm for the specular reflection boundary, which requires the limit $n\rightarrow \infty$ and $a\rightarrow 0$. 
 
 In the following lemma, we will show the existence of a mild solution to the approximate problem (\ref{app-eq}) \-- (\ref{modified-BC}) by the fixed-point theorem and a standard continuation argument.

 \begin{lemma} [Existence of mild solutions to approximate problem] \label{existence:mild-sol}
 	For any given\newline constants $\varepsilon,$ $a>0$, $n\in \mathbb{N}$, a given $T\!>\!0$ independent of $\varepsilon$, $a$ and $n$, and a given initial distribution $f_0\in L^1\cap L^\infty (\Omega \times\mathbb{R}^3)$, there exists a unique mild solution  $F \!\eqdef  f^{\varepsilon,a,n} \in C \!\left([0,T]  ; L^1 (\Omega \times \mathbb{R}^3) \right)  \cap  L^\infty \!\left([0,T]  ; L^\infty (\Omega \times \mathbb{R}^3) \right)$ of the approximate problem (\ref{app-eq}) \-- (\ref{modified-BC}) in  $[0,T]$.
 \end{lemma}
 
 \begin{proof}  We first construct a solution local in time by a fixed-point argument. Let 
 		\begin{align*}
 		\mathcal{X} \eqdef &\; \bigg\{F\in C \!\left([0,T_1]  ; L^1 (\Omega \times \mathbb{R}^3) \right)  \cap  L^\infty \!\left([0,T_1]  ; L^\infty (\Omega \times \mathbb{R}^3) \right) : \\
 		&\; \sup_{t\in [0,T_1]}\|F(t)\|_{ L^\infty}\leq 2\|f_0\|_{L^\infty},\; 
 		\sup_{t\in [0,T_1]}\|F(t)\|_{ L^1}\leq 2\|f_0\|_{L^1}
 		\bigg\} 
 		\end{align*}
 		be our work space, and it is obviously a complete metric space.
 		See also the right-hand side of (\ref{mild-sol-integral-formula}) for the definition of  $\mathcal{T}[F] (t,x,v)$.
 		Now we aim to show that  $\mathcal{T}$  maps  $\mathcal{X}$  to itself and is a contraction, if  $T_1=T_1(\varepsilon,a,n) >0$  is sufficiently small. 
 		
 		 First, for all  $F\in\mathcal{X}$, 
 		\begin{equation*}
 		\big|Q^\varepsilon[F]\big(s,X(s),V(s)\big)\big|  \leq  
 		\frac{4}{\varepsilon^2}  \|F(s)\|_{ L^\infty}\! \int_{\mathbb{R}^3}\!\xi(u) du  \leq 
 		\frac{8}{\varepsilon^2} \|f_0\|_{L^\infty},
 		\end{equation*}
 		recalling the expression of  $Q^\varepsilon[F]$.
 		Together with the bound for the term of  $\bar{f_0}$  by using Lemma~\ref{Lem:trace-est}\footnote{Note that the trace estimate (\ref{trace-est-L^infty}) is actually uniform in the iteration process generated by the contraction mapping $\mathcal{T}$, and therefore the applicability can be justified.}, we get for all  $t\in[0,T_1]$,
 		\begin{align*}
 		\|\mathcal{T}[F](t)\|_{L^\infty} 
 		&  \leq  \|\bar{f_0}\|_{L^\infty} + T_1\big|Q^\varepsilon[F]\big| \\[2pt]
 		&  \leq  e^{C(\varepsilon) T_1}\|f_0\|_{L^\infty} + \frac{8T_1}{\varepsilon^2}\|f_0\|_{L^\infty} \\
 		&  \leq  2\|f_0\|_{L^\infty},
 		\end{align*}
 		provided  $T_1$  is chosen in such a way that  $e^{C(\varepsilon) T_1} \!+\! \frac{8T_1}{\varepsilon^2} \leq 2$, where  $C(\varepsilon)=\frac{4}{\varepsilon^2}\sim O(\frac{1}{\varepsilon^2})$  is given in (\ref{trace-est-L^infty}). 
 		This means 
 		\begin{equation*}
 		\sup_{t\in [0,T_1]}\|\mathcal{T}[F](t)\|_{ L^\infty} \leq  2\|f_0\|_{L^\infty} .
 		\end{equation*}
 		We go on to estimate the  $L^1$ norm of  $\mathcal{T}[F](t)$  for any  $t\in[0,T_1]$:
 		\begin{align*}
 		\|\mathcal{T}[F](t)\|_{L^1} 
 		&  \leq  \iint_{\Omega\times\mathbb{R}^3} \big|\bar{f_0}\big(X(t_0),V(t_0)\big)\big|dxdv \\
 		&  + \iint_{\Omega\times\mathbb{R}^3} \int_{t_0}^t \big|Q^\varepsilon[F]\big(s,X(s),V(s)\big)\big| ds dxdv \\[3pt]
 		&  =:  I + II ,
 		\end{align*}
 		where the first integral
 		\begin{align*}
 		I 
 		&  =  \iint_{\Omega\times\mathbb{R}^3} \big|\bar{f_0}\big(X(t_0),V(t_0)\big)\big|dxdv \\
 		&  \leq  e^{C T_1}\! \iint_{\Omega\times\mathbb{R}^3} \big|f_0(x,v)\big|dxdv ,
 		\end{align*}
 		by a change of variables  $(x,v)\mapsto\big(X(t_0),V(t_0)\big)$  and Lemma~\ref{Lem:jacobian-est} for the bound of its Jacobian, where the constant  $C>0$  can be found in (\ref{jacobian-est-1}).  
 		For a similar reason, the second integral can be estimated as follows.
 		\begin{align*}
 		II 
 		&  =  \iint_{\Omega\times\mathbb{R}^3} \int_{t_0}^t \big|Q^\varepsilon[F]\big(s,X(s),V(s)\big)\big| ds dxdv \\
 		&  \leq  \frac{2}{\varepsilon^2} \iint_{\Omega\times\mathbb{R}^3}  \int_{t_0}^t \int_{\mathbb{R}^3} \Big[\big|F\big(s,X(s),V(s)+\varepsilon u\big)\big| \\&\ \ + \big|F\big(s,X(s),V(s)\big)\big|\Big]\xi(u) du ds dxdv \\
 		&  \leq  \frac{4}{\varepsilon^2}  e^{C T_1}  \|\xi\|_{L^1}\! \int_0^t\iint_{\Omega\times\mathbb{R}^3} \big|F(s,x,v)\big| dxdv ds \\
 		&  \leq  \frac{4}{\varepsilon^2}  e^{C T_1}  T_1\! \sup_{s\in [0,T_1]}\!\|F(s)\|_{ L^1} \\
 		&  \leq  \frac{8T_1}{\varepsilon^2}  e^{C T_1}  \|f_0\|_{L^1} .
 		\end{align*}
 		So if  $T_1$  is further made small enough such that  $e^{C T_1}\! \left(1 + \frac{8T_1}{\varepsilon^2}\right) \leq 2$, we then have 
 		\begin{equation*}
 		\sup_{t\in [0,T_1]}\|\mathcal{T}[F](t)\|_{ L^1} 
 		 \leq  e^{C T_1}\! \left(1 + \frac{8T_1}{\varepsilon^2}\right) \|f_0\|_{L^1}
 		 \leq  2\|f_0\|_{L^1} .
 		\end{equation*}
 		For the continuity in time of  $\|\mathcal{T}[F](t)\|_{L^1}$, we argue by the absolute continuity of  $L^1$ norm and the dominated convergence theorem, observing that $X(s) \!=\! X(s;t,x,v)$, $V(s) \!=\! V(s;t,x,v)$, and  $t_0 \!=\! t_0(t,x,v)$  are continuous in $t$. 
 		So we see  $\mathcal{T}[F]\in C \!\left([0,T_1]  ; L^1 (\Omega \times \mathbb{R}^3) \right)$. 
 		All these above imply that  $\mathcal{T}[F]\in \mathcal{X}$  and hence  $\mathcal{T}$  maps  $\mathcal{X}$  into  $\mathcal{X}$.
 		
 		 Secondly, for all  $F_1,F_2 \in\mathcal{X}$  and  $t\in[0,T_1]$, similar arguments yield 
 		\begin{align*}
 		\|\mathcal{T}[F_1](t) - \mathcal{T}[F_2](t)\|_{L^\infty}
 		&  =  \|\mathcal{T}[F_1\!-\!F_2](t)\|_{L^\infty} \\[5pt]
 		&  \leq  0  +  T_1\big|Q^\varepsilon[F_1\!-\!F_2]\big| \\[1pt]
 		&  \leq  \frac{4T_1}{\varepsilon^2} \sup_{t\in [0,T_1]}\|F_1(t)-F_2(t)\|_{L^\infty},
 		\end{align*}
 		\begin{equation*}
 		\|\mathcal{T}[F_1](t) - \mathcal{T}[F_2](t)\|_{L^1}  \leq  
 		\frac{4T_1}{\varepsilon^2}  e^{C T_1}\!\! \sup_{t\in [0,T_1]}\|F_1(t)-F_2(t)\|_{L^1} .
 		\end{equation*}
 		Note that the first term of  $\mathcal{T}[F_1-F_2]$  vanishes, again due to Lemma~\ref{Lem:trace-est}  for the case when  $t_0>0$.
 		Since  $\frac{4T_1}{\varepsilon^2} < 1/2 <1$  and  $\frac{4T_1}{\varepsilon^2}  e^{C T_1} <1/2 <1$  with our choice of  $T_1$  above, we conclude that  $\mathcal{T}$  is a contraction.
 	 		
 		Therefore, by the Banach fixed-point theorem (i.e.,contraction mapping principle), there exists a unique mild solution in $\mathcal{X}$  on the time interval  $[0,T_1]$  for such  $T_1=T_1(\varepsilon)$  that both  $e^{C(\varepsilon) T_1} + \frac{8T_1}{\varepsilon^2}  \leq 2$  and  $e^{C T_1}\! \left(1 + \frac{8T_1}{\varepsilon^2}\right) \leq 2$  are satisfied.
 		 
 		  For the global existence, since  $T_1=T_1(\varepsilon)$  does not depend on the initial data $f_0$, by a continuation argument, we can extend the existence time interval to an arbitrary time  $T>0$  independent of  $(\varepsilon,a,n)$.
 \end{proof}
 \section{Uniform a priori estimates}
 In this section, we will obtain the uniform estimates for the approximate solutions, which is the prerequisite for the construction of weak solutions by passing to the limit the approximate solutions.
 
 The main ingredient of the proof is the maximum principle, a property that has been extensively studied in the analysis of elliptic, parabolic, and even ``hypo-elliptic'' problems. Here we exploit this property and adapt it to the corresponding set-up of our problem.

 \subsection{Weak solutions to the approximate problem}
 We first introduce the definition of weak solutions to the approximate problem and show the existence of the weak solutions.
 
 \begin{definition} [Weak solutions to approximate problem] \label{Def:weak-sol-app-pb}
 	$F \!\eqdef \! f^{\varepsilon,a,n} \in C \!\big([0,T]  ;$\newline$ L^1 (\Omega \times \mathbb{R}^3) \big) \cap L^\infty \!\left([0,T]  ; L^\infty (\Omega \times \mathbb{R}^3) \right)$ is a weak solution of the approximate problem (\ref{app-eq}) \-- (\ref{modified-BC}) if for any test function  $\psi \in C_{t,x,v}^{1,1,1}\left((0,T)\times\Omega \times \mathbb{R}^3\right)  \cap  C\!\left([0,T]\times\bar{\Omega} \times \mathbb{R}^3\right) $ such that $\psi(t)$ is compactly supported in $\bar{\Omega}\!\times\!\mathbb{R}^3$ for all $t\in [0,T]$ and with the dual modified specular reflection boundary condition:
 	\begin{equation} \label{dual-modified-specular-BC} 
 	\begin{array}{rcl} 
 	& \gamma_+ \psi = (1\!-\!a)  \mathcal{R^*}\!\left[\gamma_- \psi\right] & \\[5pt]
 	\text{i.e.,}\; & \gamma_+ \psi(t,x,v) = (1\!-\!a)  \gamma_- \psi(t,x,R_x v) & \quad \forall\; (t,x,v)\in\Sigma_{+}^{ T},
 	\end{array}
 	\end{equation}
 	it satisfies that for every $t\in [0,T]$,
 	\begin{multline} \label{Def:weak-formulation-app-pb}
 	 \iint_{\Omega\times\mathbb{R}^3} F(t,x,v) \psi(t,x,v) dxdv  - \iint_{\Omega\times\mathbb{R}^3} f_0(x,v) \psi(0,x,v) dxdv \\[3pt]
 	 = \int_0^t\! \iint_{\Omega\times\mathbb{R}^3} F(\tau,x,v) \bigg[ \partial_t \psi + \nabla_{\!x}\cdot \Big(\big\{\beta_\varepsilon(v) + [v\!-\!\beta_\varepsilon(v)]  \eta_\varepsilon(x) \big\} \psi \Big) \\[3pt]
 - \nabla_v\cdot\big(\mathbf{B}  \psi \big)  + \frac{2}{\varepsilon^2} \int_{\mathbb{R}^3}\! \big[\psi(\tau,x,v\!-\!\varepsilon u)-\psi(\tau,x,v)\big] \xi(u) du \bigg] dxdvd\tau\\
 	 -\int_{\Sigma^t}(\beta_\varepsilon(v)\cdot n_x)\gamma F\psi dS_xdvd\tau.
 	\end{multline}
 	Here we introduce a new notation  $\bar{Q}^\varepsilon[\psi]$  for the adjoint operator
 	\begin{equation*}
 	\bar{Q}^\varepsilon[\psi](t,x,v) \eqdef   \frac{2}{\varepsilon^2} \int_{\mathbb{R}^3}\! \big[\psi(t,x,v\!-\!\varepsilon u)-\psi(t,x,v)\big] \xi(u) du .
 	\end{equation*}
 \end{definition}
 \begin{remark}
 	The reason that we adapt the modified specular reflection boundary condition \eqref{dual-modified-specular-BC} for the test function $\psi$ is that we can have 
 	\begin{multline*}\int_{\Sigma^t}(\beta_\varepsilon(v)\cdot n_x)\gamma F\psi dS_xdvd\tau\\=(1-a)\int_{\Sigma_+^t}(\beta_\varepsilon(v)\cdot n_x)\gamma_+ (f^{\varepsilon,a,n}-f^{\varepsilon,a,n-1})\psi dS_xdvd\tau
 	\rightarrow 0,\end{multline*} as $n\rightarrow \infty$ when we go back to the original specular reflection boundary problem.
 \end{remark}
  Then we prove the existence of weak solutions to the approximate problem via showing that the mild solution from Lemma \ref{existence:mild-sol} is indeed a weak solution in the following lemma. The proof is a multi-dimensional generalization of the proof in \cite{MR3614499} for one-dimensional Fokker-Planck equation with the inflow boundary condition.
 \begin{lemma} [Existence of weak solutions to the approximate problem]
 	Let  $T>0$ and  $f_0\in L^1\cap L^\infty (\Omega \times\mathbb{R}^3)$. Then there exists a weak solution  $F \!\eqdef \! f^{\varepsilon,a,n} \in C \!\left([0,T]  ; L^1 (\Omega \times \mathbb{R}^3) \right)  \cap  L^\infty \!\left([0,T]  ; L^\infty (\Omega \times \mathbb{R}^3) \right)$ of the approximate problem\newline (\ref{app-eq}) \-- (\ref{modified-BC}).
 \end{lemma} 
 
 \begin{proof} 
 	We will show that the mild solution $F \!\eqdef \! f^{\varepsilon,a,n}$ (Definition~\ref{Def:mild-sol-app-pb}) obtained from Lemma \ref{existence:mild-sol} is indeed a weak solution (Definition \ref{Def:weak-sol-app-pb}) of the approximate problem (\ref{app-eq}) \-- (\ref{modified-BC}) by deriving the weak formulation (\ref{Def:weak-formulation-app-pb}) from formula~(\ref{mild-sol-integral-formula}).
 	
 	To be specific, we start by choosing a test function $\psi\!\in\! C^1$ as in Definition \ref{Def:weak-sol-app-pb}  and defining the forward stopping-time  $t_1\!=\!t_1(t,x,v)$  to be the minimum value of  $\tau \geq t$  such that  $X(\tau ;t,x,v)\in \partial\Omega$  for each given $(t,x,v)\in (0,T]\times\Omega\times\mathbb{R}^3$. Also, recall that $t_0$ is defined to be the backward stopping time of the trajectory. Then we proceed as follows.
 	
 	First, we compute
 	\begin{multline*}
 	I \eqdef  \iint_{\Omega\times\mathbb{R}^3}\int_{t_0}^{t_1} F\big(s,X(s),V(s)\big)  \frac{\partial}{\partial s} \left\{\psi\big(s,X(s),V(s)\big)  \frac{\partial\big(X(s),V(s)\big)}{\partial (x,v)}\right\} ds dxdv \\[3pt]
 	=\iint_{\Omega\times\mathbb{R}^3}\int_{t_0}^{t_1} F\big(s,X(s),V(s)\big)  \frac{\partial}{\partial s} \left\{\psi\big(s,X(s),V(s)\big) \right\} \frac{\partial\big(X(s),V(s)\big)}{\partial (x,v)} ds dxdv \\
 	+\iint_{\Omega\times\mathbb{R}^3}\int_{t_0}^{t_1} F\big(s,X(s),V(s)\big)   \left\{\psi\big(s,X(s),V(s)\big)  \frac{\partial\big(\frac{\partial}{\partial s}X(s),\frac{\partial}{\partial s}V(s)\big)}{\partial (x,v)}\right\} ds dxdv \\
 	 = \iint_{\Omega\times\mathbb{R}^3}\int_{t_0}^{t_1} F(s,x,v) \Big[ \partial_t \psi +  \Big(\big\{\beta_\varepsilon(v) + [v\!-\!\beta_\varepsilon(v)]  \eta_\varepsilon(x) \big\}  \Big)\cdot\nabla_x \psi\\ - \mathbf{B} \cdot \nabla_v\psi\Big] ds dxdv \\
 	 +\iint_{\Omega\times\mathbb{R}^3}\int_{t_0}^{t_1} F\big(s,X(s),V(s)\big)   \psi\big(s,X(s),V(s)\big)  \\\cdot\left(
 	 \begin{array}{c|c}
 	 &  \\[-12pt]
 	 \nabla_{\!x}\mathbf{W_{\!\varepsilon}}  &  \nabla_{\!v}\mathbf{W_{\!\varepsilon}} \\[3pt]
 	 \hline \\[-10pt]
 	 -\nabla_{\!x}\mathbf{B}  &  -\nabla_{\!v}\mathbf{B} \\[-12pt]
 	 & 
 	 \end{array}
 	 \right) _{\left(X(s),V(s)\right)}\frac{\partial(X(s),V(s))}{\partial(x,v)} ds dxdv, \end{multline*} by \eqref{system-ODE}. Then note that $\{(s,x,v):(x,v)\in \Omega\times\mathbb{R}^3, t_0<s<t_1\}=\{(s,x,v):s\in (0,t),\ V(s)\in \rth, \ X(s)\in \Omega\}.$ Therefore, we take a change of variables  $(s,x,v)\mapsto$\newline$\big(s,X(s),V(s)\big)$ and obtain
 	 \begin{multline*}
 	I = \iint_{\Omega\times\mathbb{R}^3}\int_{0}^{t} F(s,x,v) \Big[ \partial_t \psi + \nabla_{\!x}\cdot \Big(\big\{\beta_\varepsilon(v) + [v\!-\!\beta_\varepsilon(v)]  \eta_\varepsilon(x) \big\} \psi \Big)\\ - \nabla_v\cdot\big(\mathbf{B} \psi\big)\Big] ds dxdv \\[3pt]
 	 = \int_{0}^{t} \!\iint_{\Omega\times\mathbb{R}^3} F(\tau,x,v)\Big[ \partial_t \psi + \nabla_{\!x}\cdot \Big(\big\{\beta_\varepsilon(v) + [v\!-\!\beta_\varepsilon(v)]  \eta_\varepsilon(x) \big\} \psi \Big)\\ - \nabla_v\cdot\big(\mathbf{B} \psi\big)\Big] dxdv d\tau ,
 	\end{multline*}
 	by the chain rule with (\ref{characteristics-X}) and (\ref{characteristics-V}). 
 	
 	On the other hand, we observe
 	\begin{align*}
 	II &\eqdef  \iint_{\Omega\times\mathbb{R}^3}\int_{t_0}^{t_1} \bar{f_0}\big(X(t_0),V(t_0)\big)  \frac{\partial}{\partial s} \left\{\psi\big(s,X(s),V(s)\big)  \frac{\partial\big(X(s),V(s)\big)}{\partial (x,v)}\right\} ds dxdv \\[3pt]
 	& = \iint_{\Omega\times\mathbb{R}^3} \bar{f_0}\big(X(t_0),V(t_0)\big)  \Bigg\{\psi\big(t_1,X(t_1),V(t_1)\big)  \frac{\partial\big(X(t_1),V(t_1)\big)}{\partial (x,v)} \\
 	&\qquad\qquad\qquad\qquad\qquad\qquad - \psi\big(t_0,X(t_0),V(t_0)\big)  \frac{\partial\big(X(t_0),V(t_0)\big)}{\partial (x,v)}\Bigg\}  dxdv \\[3pt]
 	& =\iint_{\Omega\times\mathbb{R}^3} \bar{f_0}\big(X(t_0),V(t_0)\big)\psi\big(t_1,X(t_1),V(t_1)\big)  \frac{\partial\big(X(t_1),V(t_1)\big)}{\partial (x,v)}dxdv \\
 	&\quad-\iint_{\{\Omega\times\mathbb{R}^3\}\cap \{t_0>0\}}(\gamma_-F\psi)\big(t_0,X(t_0),V(t_0)\big)  \frac{\partial\big(X(t_0),V(t_0)\big)}{\partial (x,v)} dxdv\\
 	&\quad-\iint_{\{\Omega\times\mathbb{R}^3\}\cap \{t_0=0\}} f_0\big(X(0),V(0)\big)\psi\big(0,X(0),V(0)\big)  \frac{\partial\big(X(0),V(0)\big)}{\partial (x,v)} dxdv\\
 	& = \iint_{\Omega\times\mathbb{R}^3} \bar{f_0}\big(X(t_0),V(t_0)\big)  \psi(t_1,x,v)  dxdv \\
 	&\quad-\int_{\Sigma^t_-}|\beta_\varepsilon(v)\cdot n_x|\gamma_- F\psi dS_xdvd\tau
 	 - \iint_{\Omega\times\mathbb{R}^3} \bar{f_0}\big(x,v\big)  \psi(0,x,v)  dxdv,
 	\end{align*}
 	where in the last equality we made the change of variables  $(x,v)\mapsto\big(X(t_1),V(t_1)\big)$  for the first term and  $(x,v)\mapsto\big(X(t_0),V(t_0)\big)$  for the third term with the definition of  $\bar{f_0}$. Furthermore, the second term in the last equality was obtained via the change of variables $(x,v)\mapsto(t_0,S_{X(t_0)},V(t_0))$ with \begin{multline*}\frac{\partial\big(X(t_0),V(t_0)\big)}{\partial (x,v)}=\frac{\partial\big(X(t_0),V(t_0)\big)}{\partial (t_0,S_{X(t_0)},V(t_0))}\frac{\partial\big(t_0,S_{X(t_0)},V(t_0)\big)}{\partial (x,v)}\\
 	=\beta_\varepsilon (V(t_0))\cdot n_{X(t_0)}\frac{\partial\big(t_0,S_{X(t_0)},V(t_0)\big)}{\partial (x,v)}\end{multline*} for the second term as $\eta_\varepsilon(X(t_0))=0$ at the boundary. 
 	
 	Lastly, we observe that
 	\begin{align*}
 	III &\eqdef  \iint_{\Omega\times\mathbb{R}^3}\int_{t_0}^{t_1}\! 
 	\left(\int_{t_0}^{s} Q^\varepsilon[F]\big(\tau,X(\tau),V(\tau)\big)d\tau \right) \\
 	&\qquad\qquad\qquad\qquad\qquad\qquad \times
 	\frac{\partial}{\partial s} \left\{\psi\big(s,X(s),V(s)\big)  \frac{\partial\big(X(s),V(s)\big)}{\partial (x,v)}\right\} ds dxdv \\[5pt]
 	& = \iint_{\Omega\times\mathbb{R}^3}\int_{t_0}^{t_1}
 	Q^\varepsilon[F]\big(\tau,X(\tau),V(\tau)\big)\\
 	 &\qquad\qquad\qquad\qquad\qquad \times
 	\int_{\tau}^{t_1}\! \frac{\partial}{\partial s} \left\{\psi\big(s,X(s),V(s)\big)  \frac{\partial\big(X(s),V(s)\big)}{\partial (x,v)}\right\} ds d\tau dxdv \\[7pt]
 	& = \iint_{\Omega\times\mathbb{R}^3}\int_{t_0}^{t_1}
 	Q^\varepsilon[F]\big(\tau,X(\tau),V(\tau)\big)
 	\Bigg\{\psi\big(t_1,X(t_1),V(t_1)\big)  \frac{\partial\big(X(t_1),V(t_1)\big)}{\partial (x,v)} \\
 	&\qquad\qquad\qquad\qquad\qquad\qquad\qquad - \psi\big(\tau,X(\tau),V(\tau)\big)  \frac{\partial\big(X(\tau),V(\tau)\big)}{\partial (x,v)}\Bigg\}  d\tau dxdv \\[3pt]
 	& = \iint_{\Omega\times\mathbb{R}^3} \Big\{F\big(t_1,X(t_1),V(t_1)\big) - \bar{f_0}\big(X(t_0),V(t_0)\big) \Big\} \\[-3pt]
 	&\qquad\qquad\qquad\qquad\qquad\qquad\qquad \cdot \psi\big(t_1,X(t_1),V(t_1)\big)  \frac{\partial\big(X(t_1),V(t_1)\big)}{\partial (x,v)}  dxdv \\
 	&  \ \ \ \  - \iint_{\Omega\times\mathbb{R}^3}\int_{t_0}^{t_1}
 	Q^\varepsilon[F]\big(\tau,X(\tau),V(\tau)\big)  \\
 	&\qquad\qquad\qquad\qquad\qquad\qquad\qquad \cdot\psi\big(\tau,X(\tau),V(\tau)\big)  \frac{\partial\big(X(\tau),V(\tau)\big)}{\partial (x,v)}  d\tau dxdv \\[9pt]
 	& = \iint_{\{\Omega\times\mathbb{R}^3\}\cap \{t_1=t\}} F\big(t,X(t),V(t)\big) \psi\big(t,X(t),V(t)\big)  \frac{\partial\big(X(t),V(t)\big)}{\partial (x,v)}  dxdv\\
 	&\quad+ \iint_{\{\Omega\times\mathbb{R}^3\}\cap \{t_1<t\}} \gamma_+F\big(t_1,X(t_1),V(t_1)\big)\\&\qquad\qquad\qquad\qquad\qquad\qquad\qquad\cdot\psi\big(t_1,X(t_1),V(t_1)\big)  \frac{\partial\big(X(t_1),V(t_1)\big)}{\partial (x,v)}  dxdv\\
 	 &\quad- \iint_{\Omega\times\mathbb{R}^3} \bar{f_0}\big(X(t_0),V(t_0)\big)  \psi(t_1,x,v)  dxdv \\
 	&  \quad- \iint_{\Omega\times\mathbb{R}^3}\int_{0}^{t}
 	Q^\varepsilon[F](\tau,x,v)  \psi(\tau,x,v)  d\tau dxdv \\[5pt]
 	& = \iint_{\Omega\times\mathbb{R}^3} F(t,x,v)\psi(t,x,v)  dxdv +\int_{\Sigma^t_+}|\beta_\varepsilon(v)\cdot n_x|\gamma_+ F\psi dSdvd\tau\\
 	&\quad - \iint_{\Omega\times\mathbb{R}^3} \bar{f_0}\big(X(t_0),V(t_0)\big)  \psi(t_1,x,v)  dxdv \\
&\qquad\qquad\qquad\qquad\qquad\qquad\qquad 	 - \int_{0}^{t}\! \iint_{\Omega\times\mathbb{R}^3}
 	F(\tau,x,v)  \bar{Q}^\varepsilon[\psi](\tau,x,v)  dxdv d\tau .
 	\end{align*}
 	In the first step we use the Fubini theorem to interchange two integrals. The fourth equality is due to a substitution using (\ref{mild-sol-integral-formula}). The next step is by the change of variables  $(x,v)\mapsto\big(X(t_1),V(t_1)\big)$  for the first term and  $(\tau,x,v)\mapsto\big(\tau,X(\tau),V(\tau)\big)$  for the third term. Also, the second term in the last equality was obtained via the change of variables $(x,v)\mapsto(t_1,S_{X(t_1)},V(t_1))$ with \begin{multline*}\frac{\partial\big(X(t_1),V(t_1)\big)}{\partial (x,v)}=\frac{\partial\big(X(t_1),V(t_1)\big)}{\partial (t_1,S_{X(t_1)},V(t_1))}\frac{\partial\big(t_1,S_{X(t_1)},V(t_1)\big)}{\partial (x,v)}\\
 	=\beta_\varepsilon (V(t_1))\cdot n_{X(t_1)}\frac{\partial\big(t_1,S_{X(t_1)},V(t_1)\big)}{\partial (x,v)}\end{multline*} for the second term as $\eta_\varepsilon(X(t_1))=0$ at the boundary. 
 	
 	From the representation formula (\ref{mild-sol-integral-formula}) of a mild solution, we may equate $I$ with $II + III$, which leads to the verification of the weak formulation (\ref{Def:weak-formulation-app-pb}) in the time interval $[0,t]$. This completes the proof.
 \end{proof}

 \subsection{Adjoint problem}
 We will use duality argument to obtain the \emph{uniform} $L^\infty$ and $L^1$ estimates for the approximate solutions $F \!\eqdef \! f^{\varepsilon,a,n}$. To achieve that, we choose the smooth solutions of the adjoint problem as test functions in the weak formulation~(\ref{Def:weak-formulation-app-pb}) of Definition \ref{Def:weak-sol-app-pb} (see the proof of Lemma~\ref{MP-L^infty-app-sol} and Lemma~\ref{L^1-app-sol}).
 \begin{remark}
 Although these smooth solutions of the adjoint problem may not be compactly supported, since they can be approximated by test functions with compact supports as in Definition \ref{Def:weak-sol-app-pb}, they can still satisfy formula~(\ref{Def:weak-formulation-app-pb}) of Definition \ref{Def:weak-sol-app-pb}.
 \end{remark}
 \begin{definition} [Adjoint (backward) problem of approximate problem] \label{Def:adj-pb-app-pb}
 	Let \newline $\psi_T \in C_c^\infty(\Omega\times\mathbb{R}^3)$ be a smooth function satisfying the compatibility condition:
 	\begin{equation} \label{compatibility-condition}
 	\begin{array}{rcl} 
 	& \psi_T(x,v)=0, & \\[5pt]
 	\text{on }\; & \big\{ (x,v)\in\Omega\times\mathbb{R}^3 \!:  x_{\!\perp}^{ 2} + |\beta_\varepsilon(v)|^2 < \delta \text{,  for some } \delta \text{ small}  \big\}. &
 	\end{array}
 	\end{equation}
 	The adjoint equation of approximate equation (\ref{app-eq}) with the terminal condition and dual modified specular reflection boundary condition is as follows:
 	\begin{align}
 	\bar{\mathcal{L}}^\varepsilon \psi \eqdef  &\;  \partial_t \psi + \nabla_{\!x}\cdot \Big(\big\{\beta_\varepsilon(v) + [v\!-\!\beta_\varepsilon(v)]  \eta_\varepsilon(x) \big\} \psi \Big) - \nabla_{\!v}\cdot\big(\mathbf{B}  \psi \big) + \bar{Q}^\varepsilon[\psi] = 0 \nonumber \\[3pt]
 	& \bar{Q}^\varepsilon[\psi](t,x,v) \eqdef  \frac{2}{\varepsilon^2} \int_{\mathbb{R}^3}\! \big[\psi(t,x,v\!-\!\varepsilon u)-\psi(t,x,v)\big] \xi(u) du \label{adj-eq-Q-bar} \\[5pt] 
 	& \psi(T,x,v) = \psi_T(x,v) \label{adj-terminal} \\[5pt]
 	& \gamma_+ \psi = (1\!-\!a)  \mathcal{R^*}\!\left[\gamma_- \psi\right] \label{adj-dual-modified-specular-BC}
 	\end{align}
 \end{definition}

 \begin{lemma} [Adjoint Problem] \label{Lem:adj-pb}
 	Let $\psi_T \in C_c^\infty(\Omega\times\mathbb{R}^3)$ be a smooth data at  $t\!=\!T$  satisfying the compatibility condition (\ref{compatibility-condition}).
 	\begin{enumerate}
 		\item (Existence \& Regularity) Then there exists a smooth solution $\psi\in C^\infty\!\big((0,T)$
 		\newline $\times\Omega\times\mathbb{R}^3 \big)  \cap  C \!\left([0,T]  ; L^1 (\Omega \times \mathbb{R}^3) \right)  \cap  L^\infty \!\left([0,T]  ; L^\infty (\Omega \times \mathbb{R}^3) \right)$ to the adjoint problem (\ref{adj-eq-Q-bar}) \-- (\ref{adj-dual-modified-specular-BC}) backward in time.
 		\item (Max/Min principle, Non-negativity \& $L^\infty$ estimate) Moreover, with non-negative terminal data  $\psi_T\geq 0$, then  $\psi\geq 0$  in  $\bar{Q}_T = [0,T]\times\bar{\Omega}\times\mathbb{R}^3$. Generally, we have the maximum principle:
 		\begin{equation*}
 		\max_{\bar{Q}_T} |\psi| = \max_{\bar{\Omega}\times\mathbb{R}^3} |\psi_T|,
 		\end{equation*}
 		and consequently the $L^\infty$ estimate:
 		\begin{equation*}
 		\|\psi\|_{L^\infty (Q_T)}  \leq  \|\psi_T\|_{L^\infty (\Omega \times \mathbb{R}^3)}.
 		\end{equation*}
 		\item ($L^1$ estimate)  When  $\psi_T\geq 0$, $\psi$ is also integrable in $(x,v)$  for each fixed  $t\!\in\! [0,T]$, and the $L^1$ norm   $\|\psi(t)\|_{L^1 (\Omega \times \mathbb{R}^3)} = \iint_{\Omega \times \mathbb{R}^3}\psi(t,\cdot,\cdot)dxdv$  does not increase backward in time, i.e.,$\frac{d}{dt}\iint_{\Omega \times \mathbb{R}^3}\psi(t,\cdot,\cdot)dxdv \geq 0$, for all  $t\in [0,T]$. In particular,
 		\begin{equation*}
 		\iint_{\Omega \times \mathbb{R}^3}\psi(0,\cdot,\cdot ) dxdv  \leq  \iint_{\Omega \times \mathbb{R}^3}\psi(T,\cdot,\cdot ) dxdv.
 		\end{equation*}
 	\end{enumerate}
 \end{lemma}
 
 \begin{proof}
 \noindent\textbf{ (1) Existence \& Regularity.}
 		  		\begin{enumerate}
 			\item[(I)] For the existence of a mild solution $$\psi\in C \!\left([0,T]  ; L^1 (\Omega \times \mathbb{R}^3) \right)  \cap  L^\infty \!\left([0,T]  ; L^\infty (\Omega \times \mathbb{R}^3) \right),$$ with an analog of formula~(\ref{mild-sol-integral-formula}), we apply a fixed-point argument similar to Lemma \ref{existence:mild-sol}.
 			 
 			\item[(I\!I)] To prove the regularity (smoothness) of the solution, since $\psi_T \in C^\infty$ and satisfies the compatibility condition (\ref{compatibility-condition}), we can get the integral equation corresponding to the derivatives of $\psi$ by differentiating the integral representation for $\psi$ itself. Then again follow a similar procedure as in the proof of Lemma~\ref{existence:mild-sol} to show that the solution is indeed smooth and thus is a classical solution.
 		\end{enumerate}
 		\noindent\textbf{(2) Max/Min principle, Non-negativity \& $L^\infty$ estimate.}
 		  For the case $\psi_T\geq 0$ (together with all previous assumptions on $\psi_T$), we will show the non-negativity of $\psi$ as follows:
 			  			\begin{enumerate}
 				\item[(i)] Assuming first ``$\bar{\mathcal{L}}^\varepsilon \psi < 0$'' with ``$\psi(T)=\psi_T+k$ ($k\!>\!0$, small)'' instead of (\ref{adj-eq-Q-bar}), (\ref{adj-terminal}) and (\ref{adj-dual-modified-specular-BC}), we will show that ``$\psi>0$  in  $\bar{Q}_T = [0,T]\times\bar{\Omega}\times\mathbb{R}^3$.''
 				Define $$ T_* \eqdef  \inf \left\{ T_1\in [0,T] :  \psi>0 \;\text{ in }  Q_{T_1,T}\eqdef [T_1,T]\times\Omega\times\mathbb{R}^3 \right\}.$$ Since $\psi|_{t=T} = \psi_T+k \geq k>0$, by continuity of $\psi$, we have $T_*<T$, so  $T_*\in [0,T)$. Also, since $\psi\!\geq\! 0$ in $Q_{T_*,T}$ and smooth, its continuous extension attains the minimum in $\overline{Q_{T_*,T}}$. We may assume (by contradiction) that this minimum value is $0$, since otherwise $\psi>0$  in  $\bar{Q}_T$ by the definition of $T_*$ and thus we are done. 
 				
 				  We claim that ``the minimum $0$  cannot be attained at  $\overline{Q_{T_*,T}} \backslash \{t\!=\!T\}\times\bar{\Omega}\times\mathbb{R}^3$'' by arguing that  $\bar{\mathcal{L}}^\varepsilon \psi \geq 0$  otherwise. More precisely, we break it down into the following cases:

 				\begin{itemize}
 					\item[$^{_\bullet}$] If ``Min $=\!0$'' is attained at some \emph{interior point} $(t_0,x_0,v_0)\in (T_*,T)\times\Omega\times\mathbb{R}^3$ or at some $(T_*,x_0,v_0)$ with $(x_0,v_0)\in \Omega\times\mathbb{R}^3$, then at this point we have  $\partial_t\psi \geq 0$,  $\psi = \nabla_{\!x}\psi = \nabla_{\!v}\psi =0$  so that 
 					\begin{align*}
 					& \nabla_{\!x}\cdot \Big(\big\{\beta_\varepsilon(v) + [v\!-\!\beta_\varepsilon(v)]  \eta_\varepsilon(x) \big\} \psi \Big) =0, \\[5pt]
 					& \nabla_{\!v}\cdot\big(\mathbf{B}  \psi \big) 
 					= \mathbf{B}\cdot\nabla_{\!v}\psi 
 					 + (\nabla_{\!v}\cdot\mathbf{B}) \psi 
 					 =  0,
 					\end{align*}
 					and  $\bar{Q}^\varepsilon[\psi] \geq 0$  by the definition of  $\bar{Q}^\varepsilon[\psi]$.
 					 
 					\item[$^{_\bullet}$] If ``Min $=\!0$'' is attained at some \emph{incoming boundary point} $(t_0,x_0,v_0)\in [T_*,T)\times\gamma_-$, then at that point  $\partial_t\psi \geq 0$,  $\psi = \nabla_{\!v}\psi =0$  so that  $\nabla_{\!v}\cdot\big(\mathbf{B}  \psi \big) = 0$,
 					and again  $\bar{Q}^\varepsilon[\psi] \geq 0$.
 					Also, we have instead
 					\begin{equation*}
 					\nabla_{\!x}\cdot \Big(\big\{\beta_\varepsilon(v) + [v\!-\!\beta_\varepsilon(v)]  \eta_\varepsilon(x) \big\} \psi \Big) =  \beta_\varepsilon(v)\cdot\nabla_{\!x}\psi  \geq 0
 					\end{equation*}
 					by observing that  $\beta_\varepsilon(v)\!\cdot\!\nabla_{\!x}\psi = \big|\beta_\varepsilon(v)\big| D_{\vec{\hat{v}}}\psi$, where the directional derivative (with respect to $x$)  $D_{\vec{\hat{v}}}\psi \geq 0$  in the direction of $\vec{v}$ at the minimum point with $(x_0,v_0)$ being in the incoming set.

 					\item[$^{_\bullet}$] If ``Min $=\!0$'' is attained at some \emph{outcoming boundary point} $(t_0,x_0,v_0)\in [T_*,T)\times\gamma_+$, then the minimum is also attained at the point\newline  $(t_0,x_0,R_{x_0}v_0)\in [T_*,T)\times\gamma_-$ on the incoming boundary by the boundary condition (\ref{adj-dual-modified-specular-BC}), which reduces to the previous case.
 					 
 					\item[$^{_\bullet}$] If ``Min $=\!0$'' is attained at some \emph{``grazing'' boundary point} $(t_0,x_0,v_0)\in [T_*,T)\times\gamma_0$, then again at that point  $\partial_t\psi \geq 0$,  $\psi = \nabla_{\!v}\psi =0$\; so  $\nabla_{\!v}\cdot\big(\mathbf{B}  \psi \big) = 0$,
 					and  $\bar{Q}^\varepsilon[\psi] \geq 0$. Moreover, since  $\beta_\varepsilon(v) + [v\!-\!\beta_\varepsilon(v)]  \eta_\varepsilon(x) \equiv \vec{0}$  in a neighborhood of the grazing set $\big\{(x,v): |x_{\perp}|, |v_{\perp}| \leq\varepsilon \big\}$, we also have
 					\begin{equation*}
 					\nabla_{\!x}\cdot \Big(\big\{\beta_\varepsilon(v) + [v\!-\!\beta_\varepsilon(v)]  \eta_\varepsilon(x) \big\} \psi \Big)  =  0 .
 					\end{equation*}
 				\end{itemize}
 				In all the cases above, we would have  $\bar{\mathcal{L}}^\varepsilon \psi \geq 0$  at some point, a contradiction. So the claim holds.
 				
 				  Combined with the fact that $\psi \!=\! \psi_T+k >0$ at $t=T$, we have  $\psi>0$  in  $\overline{Q_{T_*,T}}$. Therefore, $T_*=0$  by the definition of  $T_*$  and the continuity of  $\psi$  (because if  $T_*>0$, then by the continuity of  $\psi$,  $\psi\!>\!0$  in $Q_{T_*\!-\delta, T}$ for some $\delta\!>\!0$, which contradicts the definition of  $T_*$ ), and hence we obtain  $\psi>0$  in  $\bar{Q}_T$.
 				 
 				\item[(ii)] Now back to the general case: ``$\bar{\mathcal{L}}^\varepsilon \psi \leq 0$'' with (\ref{adj-terminal}) and (\ref{adj-dual-modified-specular-BC}), we then show that ``$\psi \geq 0$  in  $\bar{Q}_T = [0,T]\times\bar{\Omega}\times\mathbb{R}^3$'' (by reducing to the model-case of (i) with the aid of auxiliary functions), and subsequently ``$\min_{\bar{Q}_T} \psi = \min_{\bar{\Omega}\times\mathbb{R}^3} \psi_T = 0$'' as follows.				
 				  Choose  $L>0$  sufficiently large such that
 				$$L  \geq  \big(v\!-\!\beta_\varepsilon(v)\big)\cdot\nabla_{\!x}\eta_\varepsilon(x)  +  \nabla_{\!v}\cdot\mathbf{B}$$
 				for all  $(t,x,v) \in \bar{Q}_T$. 
 				Let $$\psi^k(t,x,v) \eqdef  \psi(t,x,v) + \big[k+k(T\!-\!t)\big] e^{L(T-t)}$$ for  $k>0$  small. 
 				Then 
 				\begin{align*}
 				\bar{\mathcal{L}}^\varepsilon \psi^k  
 				&= \;  \bar{\mathcal{L}}^\varepsilon \psi -k  e^{L(T-t)} \\
 				& \; +\big\{\!-L +\big(v\!-\!\beta_\varepsilon(v)\big)\cdot\nabla_{\!x}\eta_\varepsilon(x)  + \nabla_{\!v}\cdot\mathbf{B}\big\} \big[k+k(T\!-\!t)\big] e^{L(T-t)} \\[2pt] 
 				&\leq \;  \bar{\mathcal{L}}^\varepsilon \psi -k  e^{L(T-t)} \leq \; -k  e^{L(T-t)}  <  0 ,
 				\end{align*} where$$
 				\psi^k|_{t=T}  =  \psi|_{t=T}+k  =  \psi_T+k  \geq  k > 0 .$$
 				Also, with $\psi$ satisfying the boundary condition (\ref{adj-dual-modified-specular-BC}), we can repeat the arguments in the model-case (i) for $\psi^k$ with some modifications of the outcoming-set case, which leads to a contradiction as well. 
 				
 				  Therefore, applying the result of (i) to $\psi^k$, we get  $\psi^k > 0$  in  $\bar{Q}_T$. 
 				Taking the limit $k\rightarrow 0$, since $\psi^k \rightarrow \psi$, we obtain  $\psi \geq 0$  in  $\bar{Q}_T$, and thus  $\min_{\bar{Q}_T} \!\psi \geq 0$. On the other hand, notice that  $\min_{\bar{Q}_T} \!\psi \leq \min_{\bar{\Omega}\times\mathbb{R}^3} \psi_T = 0$, so  $\min_{\bar{Q}_T} \!\psi = \min_{\bar{\Omega}\times\mathbb{R}^3} \psi_T = 0$.
 			\end{enumerate}

 	  From the result of (ii) applied to ``$M - \psi$'' and ``$\psi - m$'', respectively, where $M \!\eqdef \! \max_{\bar{\Omega}\times\mathbb{R}^3} \psi_T$, $m \!\eqdef \! \min_{\bar{\Omega}\times\mathbb{R}^3} \psi_T$, and that $\bar{\mathcal{L}}^\varepsilon \psi = 0$ implies  $\bar{\mathcal{L}}^\varepsilon (M\!-\!\psi) \leq 0$  and  $\bar{\mathcal{L}}^\varepsilon (\psi\!-\!m) \leq 0$, we can also get as a corollary (without the assumption ``$\psi_T\geq 0$'') that  $\max_{\bar{Q}_T} |\psi|  =  \max_{\bar{\Omega}\times\mathbb{R}^3} |\psi_T|$.  Then the $L^\infty$ estimate  $\|\psi\|_{L^\infty (Q_T)}  \leq  \|\psi_T\|_{L^\infty (\Omega \times \mathbb{R}^3)}$  follows.

 	\noindent\textbf{(3) $L^1$ estimate.}
 		If  $\psi_T\geq 0$, then we know $\psi \geq 0$  in $\bar{Q}_T$ by the above.
 		For each $t\in [0,T]$, we have
 		\begin{align*}
 		& \frac{d}{dt} \iint_{\Omega\times\mathbb{R}^3} \psi(t,x,v) dxdv 
 		= \;  \iint_{\Omega\times\mathbb{R}^3} \partial_{t}\psi dxdv \\ 
 		= &\;  -\iint_{\Omega\times\mathbb{R}^3} \nabla_{\!x}\cdot \Big(\big\{\beta_\varepsilon(v) + [v\!-\!\beta_\varepsilon(v)]  \eta_\varepsilon(x) \big\} \psi \Big) dxdv \\
 		& \;  +\iint_{\Omega\times\mathbb{R}^3} \nabla_{\!v}\cdot\big(\mathbf{B}  \psi \big) dxdv 
 		 -\iint_{\Omega\times\mathbb{R}^3} \bar{Q}^\varepsilon[\psi] dxdv & (\text{by  (\ref{adj-eq-Q-bar})}) \\
 		= &\;  -\iint_{\partial\Omega\times\mathbb{R}^3} \psi(t,x,v) \big(\beta_\varepsilon(v)\cdot n_x\big)  dS_xdv \\
 		= &\;  -\left(  \iint_{\gamma_-}\!+\iint_{\gamma_+}  \right) \psi\big(\beta_\varepsilon(v)\cdot n_x\big)  dS_xdv \\
 		= &\;  -\big[ 1-(1-a) \big] \iint_{\gamma_-} \psi\big(\beta_\varepsilon(v)\cdot n_x\big)  dS_xdv & (\text{by (\ref{adj-dual-modified-specular-BC})}) \\
 		= &\;  -a \iint_{\gamma_-} \psi\big(\beta_\varepsilon(v)\cdot n_x\big)  dS_xdv \;\geq  0 .
 		\end{align*}
 		The third equality is due to integration by parts in $x$ for the first term and in $v$ for the second term. Also, note that $\iint_{\Omega\times\mathbb{R}^3} \bar{Q}^\varepsilon[\psi] dxdv =~\!0$. The last inequality is due to the non-negativity of $\psi$ and that  $\beta_\varepsilon(v)\cdot n_x <0$  on $\gamma_-$. 
 		
 		  From the result above we can tell that the $L^1$ norm  $$\|\psi(t)\|_{L^1 (\Omega \times \mathbb{R}^3)} = \iint_{\Omega \times \mathbb{R}^3}\psi(t,\cdot,\cdot)dxdv$$  does not increase backward in time. Therefore, $$\iint_{\Omega \times \mathbb{R}^3}\psi(0,\cdot,\cdot )dxdv  \leq  \iint_{\Omega \times \mathbb{R}^3}\psi(T,\cdot,\cdot )dxdv.$$ 		 
 \end{proof}

 \subsection{Maximum principle \& $L^\infty$ estimate}
 We now go back to our original approximate Landau problem and establish the following maximum principle for weak solutions, which provides the result of uniform $L^\infty$ estimate for approximate solutions.
 
 \begin{lemma} [Maximum principle \& $L^\infty$ estimate for weak solutions of approximate problem] \label{MP-L^infty-app-sol}
 	If $f_0\in L^\infty (\Omega \times\mathbb{R}^3)$, then the weak solution $F \!\eqdef \! f^{\varepsilon,a,n}$ of the approximate problem (\ref{app-eq}) \-- (\ref{modified-BC}) (Definition~\ref{Def:weak-sol-app-pb}) satisfies that for all $t\in [0,T]$, $$\big|F(t,x,v)\big|  \leq  \|f_0\|_{L^\infty (\Omega \times \mathbb{R}^3)}$$ up to a zero-measure set on $\Omega \times \mathbb{R}^3$, which means that
 	\begin{equation*}
 	\|F\|_{L^\infty \left([0,T] ;  L^\infty (\Omega \times \mathbb{R}^3) \right)}  \leq  \|f_0\|_{L^\infty (\Omega \times \mathbb{R}^3)}.
 	\end{equation*}
 \end{lemma}
 
 \begin{proof}
 	We will only prove (by contradiction) that ``$F(t,x,v)  \leq  \|f_0\|_{L^\infty (\Omega \times \mathbb{R}^3)}$ up to a zero-measure set on  $\Omega \times \mathbb{R}^3$,  $\forall  t\in [0,T]$.'' The other side ``$-F(t,x,v)  \leq  \|f_0\|_{L^\infty (\Omega \times \mathbb{R}^3)}$  i.e.,$F(t,x,v)  \geq  -\|f_0\|_{L^\infty (\Omega \times \mathbb{R}^3)}$ up to a zero-measure set on  $\Omega \times \mathbb{R}^3$,  $\forall  t\in [0,T]$'' can be proved analogously.
 
  Suppose that there are  $\kappa >0$  and  $t_*\in (0,T]$  such that 
 		\begin{equation} \label{M.P.-contradiction-assumption}
 		F(t_*,x,v)  >  \|f_0\|_{L^\infty} + \kappa  
 		\end{equation}
 		on a set with positive measure, say, $A\subset \Omega\times\mathbb{R}^3$. Then for each given $\delta>0$  small, we can choose a ball $B\subset \Omega\times\mathbb{R}^3$  such that 
 		\begin{equation} \label{measure-lem}
 		m(B\cap A) > m(B)\cdot (1-\delta)  
 		\end{equation}
 		and $m(B)$ is independent of $\delta$ (cf. Lemma 8 of \cite{hwang2014fokker}).

 		  Let  $\varphi(x,v)\in C_c^\infty(\Omega\times\mathbb{R}^3)$  be a function satisfying  $\varphi\!\geq\! 0$,  $\|\varphi\|_{L^\infty} <~\!\infty$, and condition~(\ref{compatibility-condition}) such that
 		\begin{equation} \label{phi-assumption}
 		{\rm supp} \varphi \subset \bar{B}\  \text{and}   \iint_{\Omega\times\mathbb{R}^3}\varphi  dxdv = 1. 
 		\end{equation}
 		By Lemma~\ref{Lem:adj-pb}, there exists a smooth solution $\psi\in C^\infty\!\left((0,t_*)\times\Omega\times\mathbb{R}^3 \right)$ to the adjoint problem (\ref{adj-eq-Q-bar}) \-- (\ref{adj-dual-modified-specular-BC}) with the terminal condition $ \psi_{t_*}(x,v)\eqdef \psi(t_*,x,v)  = \varphi(x,v)$ such that $\psi \geq 0$  and 
 		\begin{equation} \label{psi-L^1}
 		\iint_{\Omega\times\mathbb{R}^3}\psi(0,x,v)  dxdv \leq \iint_{\Omega\times\mathbb{R}^3}\psi(t_*,x,v)  dxdv = 1. 
 		\end{equation}
 		Let $F$ be a weak solution to the approximate problem in the sense of Definition~\ref{Def:weak-sol-app-pb}, then from the weak formulation~(\ref{Def:weak-formulation-app-pb}) with test function $\psi$ chosen to be the solution of the adjoint problem (\ref{adj-eq-Q-bar}) \-- (\ref{adj-dual-modified-specular-BC}) (Definition~\ref{Def:adj-pb-app-pb}) obtained above. Then we have
 		\begin{equation} \label{weak-formulation-identity}
 		\iint_{\Omega\times\mathbb{R}^3} F(t,x,v)\psi(t,x,v)  dxdv = \iint_{\Omega\times\mathbb{R}^3} f_0(x,v)\psi(0,x,v)  dxdv, 
 		\end{equation}
 		for all $t\in [0,T]$.
 		  		  Now we estimate  $\iint_{\Omega\times\mathbb{R}^3} f_0(x,v)\psi(0,x,v)  dxdv$  and\\ $\iint_{\Omega\times\mathbb{R}^3} F(t_*,x,v)\psi(t_*,x,v)  dxdv$, respectively, and reach a contradiction. 
 	By (\ref{phi-assumption}), (\ref{psi-L^1}) and the non-negativity of $\psi$, we have
 		\begin{align} \label{MP-pf-ineq-1}
 		\begin{split}
 		& \iint_{\Omega\times\mathbb{R}^3} f_0(x,v)\psi(0,x,v)  dxdv \\ 
 		\leq &\;\; \|f_0\|_{L^\infty} \iint_{\Omega\times\mathbb{R}^3}\big|\psi(0,x,v)\big|  dxdv
 		\;\leq   \|f_0\|_{L^\infty} . 
 		\end{split}
 		\end{align}
 		Moreover, by the assumption~(\ref{M.P.-contradiction-assumption}) on $F$ at $t\!=\!t_*$, our choice of $B$ with (\ref{measure-lem}), and $\psi_{t_*}\!=\varphi$ satisfying (\ref{phi-assumption}),  we obtain
 		\begin{align} \label{MP-pf-ineq-2}
 		\begin{split}
 		& \iint_{\Omega\times\mathbb{R}^3} F(t_*,x,v)\psi(t_*,x,v)  dxdv \\ 
 		= &\;  \left( \iint_{B\cap A}+\iint_{B\backslash A} \right) F(t_*,x,v)\varphi(x,v)  dxdv \\[2pt] 
 		\geq &\;  \Big(\|f_0\|_{L^\infty}+\kappa\Big)\cdot\! \iint_{B\cap A}\varphi dxdv \\
 		& -  \|F\|_{L^\infty([0,T]\times\Omega\times\mathbb{R}^3)}\cdot\|\varphi\|_{L^\infty(\Omega\times\mathbb{R}^3)}\cdot m(B) \delta \\[3pt]
 		\geq &\;  \big(\|f_0\|_{L^\infty}+\kappa\big)\!\cdot\! \big(1 - \|\varphi\|_{L^\infty}\!\cdot m(B) \delta \big) \\
 		& -  \|F\|_{L^\infty}\!\cdot\!\|\varphi\|_{L^\infty}\!\cdot m(B) \delta \\[3pt]
 		= &\;  \|f_0\|_{L^\infty}+\kappa - C\delta , 
 		\end{split}
 		\end{align}
 		where  $C$  depends on  $\|f_0\|_{L^\infty (\Omega\times\mathbb{R}^3)}$, $\|F\|_{L^\infty ([0,T]\times\Omega\times\mathbb{R}^3)}$, $\|\varphi\|_{L^\infty (\Omega\times\mathbb{R}^3)}$, and $m(B)$. 
 	 		  On the other hand, combining (\ref{weak-formulation-identity}) with (\ref{MP-pf-ineq-1}) and (\ref{MP-pf-ineq-2}), we have
 		\begin{align} \label{MP-pf-ineq-3}
 		\begin{split}
 		\|f_0\|_{L^\infty}+\kappa - C\delta  
 		& \leq  \iint_{\Omega\times\mathbb{R}^3} F(t_*,x,v)\psi(t_*,x,v)  dxdv \\
 		& =  \iint_{\Omega\times\mathbb{R}^3} f_0(x,v)\psi(0,x,v)  dxdv \\[3pt]
 		& \leq  \|f_0\|_{L^\infty} . 
 		\end{split}
 		\end{align}
 		Thus if  $\delta$  is chosen sufficiently small in such a way that  $C\delta < \kappa/2$, we then get a contradiction. Therefore, the original claim holds.  
 \end{proof}
 
 As a direct consequence of the $L^\infty$ estimate (Lemma~\ref{MP-L^infty-app-sol}), we also deduce the uniqueness of the approximate solutions.
 
 \begin{corollary} [Uniqueness for weak solutions of approximate problem] \label{Cor:uniqueness} 
 	Let  $F_1\!\eqdef \!$\newline$  f_1^{\varepsilon,a}$, $F_2 \!\eqdef \! f_2^{\varepsilon,a}$ be two weak solutions of the approximate problem (\ref{app-eq}) \-- (\ref{modified-BC}) with the same initial and boundary conditions. Then $F_1 = F_2$  in  $L^\infty \!\left([0,T]  ; L^\infty (\Omega \times \mathbb{R}^3) \right)$. 
 \end{corollary}
 
 \begin{proof}
 	Since the equation (\ref{app-eq}) and boundary condition (\ref{modified-BC}) is linear, $F_0 \eqdef  F_1-F_2$  is a solution of   (\ref{app-eq}) with the initial condition  $F_0|_{t=0} = 0$  and the same boundary condition. Then applying the $L^\infty$ estimate (Lemma~\ref{MP-L^infty-app-sol}) to  $F_0$  yields  $\|F_1-F_2\|_{L^\infty} = 0$.
 \end{proof}

 \subsection{$L^1$ estimate}
 Next, we also present the $L^1$ estimate for the approximate solutions, as a dual result of Lemma~\ref{Lem:adj-pb}  as well.
 Let us remark that although it can be formally derived via integration by parts, here we are only allowed to work from the weak formulation because the regularity has yet to be shown.
 
 \begin{lemma} [$L^1$ estimate for weak solutions of approximate problem] \label{L^1-app-sol}
 	Let  $f_0\in$\newline$  L^1\cap L^\infty (\Omega \times\mathbb{R}^3)$ be given as an initial data, and  $F \!\eqdef \! f^{\varepsilon,a,n}$ a weak solution of the approximate problem \eqref{app-eq} \-- \eqref{modified-BC} (Definition~\ref{Def:weak-sol-app-pb}). Then its $L^1$ norm is non-increasing in time, i.e.,for each  $t\in [0,T]$,
 	\begin{equation*}
 	\|F(t)\|_{L^1 (\Omega \times \mathbb{R}^3)}  \leq  \|f_0\|_{L^1 (\Omega \times \mathbb{R}^3)} .
 	\end{equation*}
 \end{lemma}
 
 \begin{proof}Let  $t_*\in [0,T]$ be given. Let $\varphi(x,v)\in C_c^\infty(\Omega\times\mathbb{R}^3)$ be a function satisfying  $\|\varphi\|_{L^\infty} \!\leq\!1$  and the compatibility condition~(\ref{compatibility-condition}).
 		By Lemma~\ref{Lem:adj-pb}, there exists a smooth solution $\psi\in C^\infty\!\left((0,t_*)\times\Omega\times\mathbb{R}^3 \right)$ to the adjoint problem (\ref{adj-eq-Q-bar}) \-- (\ref{adj-dual-modified-specular-BC}) with the terminal condition  $ \psi_{t_*}(x,v)\eqdef \psi(t_*,x,v)  = \varphi(x,v)$  such that 
 		\begin{equation} \label{psi-L^infty}
 		\|\psi\|_{L^\infty(Q_{t_*})}  \leq  \|\psi_{t_*}\|_{L^\infty(\Omega\times\mathbb{R}^3)} = \|\varphi\|_{L^\infty(\Omega\times\mathbb{R}^3)} \leq 1. 
 		\end{equation}
 		Since $F$ is a weak solution to the approximate problem in the sense of Definition~\ref{Def:weak-sol-app-pb}, from the weak formulation~(\ref{Def:weak-formulation-app-pb}) with test function $\psi$ replaced by the solution of the adjoint problem (\ref{adj-eq-Q-bar}) \-- (\ref{adj-dual-modified-specular-BC}) (Definition~\ref{Def:adj-pb-app-pb}) obtained above, we have
 		\begin{equation} \label{weak-formulation-identity-t_*}
 		\iint_{\Omega\times\mathbb{R}^3} F(t_*,x,v)\varphi(x,v)  dxdv = \iint_{\Omega\times\mathbb{R}^3} f_0(x,v)\psi(0,x,v)  dxdv. 
 		\end{equation}
 		It follows from (\ref{psi-L^infty}) and (\ref{weak-formulation-identity-t_*}) that
 		\begin{align} \label{L^1-pf-ineq-1}
 		\begin{split}
 		\iint_{\Omega\times\mathbb{R}^3} F(t_*,x,v)\varphi(x,v)  dxdv 
 		\leq &\;\; \|\psi\|_{L^\infty(Q_{t_*})} \iint_{\Omega\times\mathbb{R}^3}\big|f_0(x,v)\big|  dxdv \\
 		\leq &\;\; \|f_0\|_{L^1(\Omega\times\mathbb{R}^3)}. 
 		\end{split}
 		\end{align}
 		Since the terminal function $\varphi\in C_c^\infty(\Omega\times\mathbb{R}^3)$ with  $\|\varphi\|_{L^\infty} \!\leq\!1$  can be arbitrarily chosen (as long as it satisfies condition~(\ref{compatibility-condition})), we can take a sequence of such functions $\{\varphi_k\}$ such that $$\varphi_k(x,v)  \rightarrow   {\rm sgn}\big[F(t_*,x,v)\big]\!\cdot\chi_{\left\{(x,v)  :  x_{\!\perp}^{ 2}  +  \left|\beta_\varepsilon(v)\right|^2  \geq  \delta \right\}}$$ as  $k\!\rightarrow\!\infty$  (for some $\delta$ small). 
 		Then by the Lebesgue's dominated convergence theorem and (\ref{L^1-pf-ineq-1}),
 		\begin{align*} 
 		\begin{split}
 		& \iint_{\Omega\times\mathbb{R}^3} \big|F(t_*,x,v)\big|\cdot\chi_{\left\{(x,v)  :  x_{\!\perp}^{ 2}  +  \left|\beta_\varepsilon(v)\right|^2  \geq  \delta \right\}}  dxdv \\ 
 		= &\;\; \lim_{k\rightarrow\infty} \iint_{\Omega\times\mathbb{R}^3} F(t_*,x,v)\varphi_k(x,v)  dxdv 
 		\leq \;\; \|f_0\|_{L^1(\Omega\times\mathbb{R}^3)} . 
 		\end{split}
 		\end{align*}
 		Again, since the above inequality holds for any small $\delta\!>\!0$ and for any $t_*\in [0,T]$, we finally obtain  $\|F(t)\|_{L^1 (\Omega \times \mathbb{R}^3)}  \leq  \|f_0\|_{L^1 (\Omega \times \mathbb{R}^3)}$ for every $t\in [0,T]$.
 \end{proof}

 \section{Well-posedness for the linearized Landau Equation}\label{section: wellposedness lin}
 Finally, we give the proof of well-posedness for the original Landau initial-boundary value problem and will further discuss some additional results.

 \subsection{Proof of Theorem \ref{main-thm}: existence, uniqueness, and $L^1\cap L^\infty$ estimate}
 We are now ready to prove the main theorem:
 
 \begin{proof}[Proof of Theorem \ref{main-thm}]The proof consists of the following four steps.
 	
\noindent \textbf{Step 1: Passing to the limit:  $f^{\varepsilon,a,n}\overset{ w^* }{\rightharpoonup} f$.}
 	We first obtain the weak limit of the approximating sequence $\left\{f^{\varepsilon,a,n}\right\}$ as a candidate for a weak solution by the weak compactness (Banach-Alaoglu theorem), which is ensured by the \emph{uniform} estimates of the approximate solutions established in the previous section. From the uniform $L^\infty$ estimate (Lemma~\ref{MP-L^infty-app-sol}) and $L^1$ estimate (Lemma~\ref{L^1-app-sol}), and from taking the limit in the weak-* topology as  $\varepsilon,a \rightarrow 0$ and $n\rightarrow \infty$, we obtain that a sequence of  $\left\{f^{\varepsilon,a,n}\right\}$ converges weakly to $f$  in $L^\infty \!\left([0,T]  ; L^1\cap L^\infty (\Omega \times \mathbb{R}^3) \right)$. 
 	Again it follows from Lemma~\ref{MP-L^infty-app-sol} and Lemma~\ref{L^1-app-sol} that $f$ satisfies the $L^\infty$ bound
 	\begin{equation} \label{L^infty-bound}
 	\|f\|_{L^\infty \left([0,T]  ;  L^\infty (\Omega \times \mathbb{R}^3) \right)}  \leq  \|f_0\|_{L^\infty (\Omega \times \mathbb{R}^3)},
 	\end{equation}
 	and the $L^1$ bound
 	\begin{equation} \label{L^1-bound}
 	\|f\|_{L^\infty \left([0,T]  ;  L^1 (\Omega \times \mathbb{R}^3) \right)}  \leq  \|f_0\|_{L^1 (\Omega \times \mathbb{R}^3)} .
 	\end{equation}
 	Moreover, by the Lebesgue's dominated convergence theorem, we have for each $t\in [0,T]$,
 	\begin{equation} \label{main-thm-pf-1}
 	\iint_{\Omega\times\mathbb{R}^3} f^{\varepsilon,a,n}(t,x,v)\psi(t,x,v)  dxdv  \rightarrow  \iint_{\Omega\times\mathbb{R}^3} f(t,x,v)\psi(t,x,v)  dxdv 	 
 	\end{equation}
 	and 
 	\begin{multline} \label{main-thm-pf-2}
 	\int_0^t\!\!\iint_{\Omega\times\mathbb{R}^3} f^{\varepsilon,a,n}(\tau,x,v)\psi(\tau,x,v)  dxdvd\tau \\ \rightarrow  \int_0^t\!\!\iint_{\Omega\times\mathbb{R}^3} f(\tau,x,v)\psi(\tau,x,v)  dxdvd\tau 	 
 	\end{multline}
 	as  $\varepsilon,a \rightarrow 0$.
 	 
 \noindent \textbf{Step 2: Weak continuity of $f(t)$ : $t \mapsto\! \iint_{\Omega\times\mathbb{R}^3} f(t,x,v)\psi(t,x,v) dxdv$.  }
 	Let a test function $\psi(t,x,v)$ be given and $t_1, t_2 \in [0,T]$. 
 	Notice that for the sequence of approximate solutions $f^{\varepsilon,a,n}$, it holds that
 	\begin{align*} 
 	\begin{split}
 	& \iint_{\Omega\times\mathbb{R}^3} f^{\varepsilon,a,n}(t_1)\psi(t_1)  dxdv 
 	 - \iint_{\Omega\times\mathbb{R}^3} f^{\varepsilon,a,n}(t_2)\psi(t_2)  dxdv \\ 
 	= &\; \iint_{\Omega\times\mathbb{R}^3} f^{\varepsilon,a,n}(t_1)\big[\psi(t_1)-\psi(t_2)\big] dxdv 
 	 + \iint_{\Omega\times\mathbb{R}^3} \big[f^{\varepsilon,a,n}(t_1)-f^{\varepsilon,a,n}(t_2)\big]\psi(t_2)  dxdv . 
 	\end{split}
 	\end{align*}
 	Since $f^{\varepsilon,a,n} \in C \!\left([0,T]  ; L^1 (\Omega \times \mathbb{R}^3) \right)  \cap  L^\infty \!\left([0,T]  ; L^\infty (\Omega \times \mathbb{R}^3) \right)$ has the uniform $L^\infty$ estimate (Lemma~\ref{MP-L^infty-app-sol}) and $L^1$ estimate (Lemma~\ref{L^1-app-sol}), both terms on the right-hand side can be made small uniformly in $\varepsilon$ and $a$ if  $|t_1-t_2|$ is sufficiently small; i.e.,we have shown that $t \mapsto \iint_{\Omega\times\mathbb{R}^3} f^{\varepsilon,a,n}(t)\psi(t) dxdv$  is ``equi-continuous''.
 	Also, since $f^{\varepsilon,a,n}$ converges in weak-* topology to $f$  by taking  $\varepsilon,a \rightarrow 0$ and $n\rightarrow \infty$, we deduce the weak continuity of $f(t)$. In particular, $\iint_{\Omega\times\mathbb{R}^3} f(t)\psi(t) dxdv$  is well-defined for every $t\in [0,T]$.
 	 
 \noindent \textbf{Step 3: Weak formulation (\ref{Def:weak-formulation}).}
 	Note that for the test function $\psi$, as $\varepsilon \rightarrow 0$, we have 
 $$
 	 \nabla_{\!x}\cdot \Big(\big\{\beta_\varepsilon(v) + [v\!-\!\beta_\varepsilon(v)]  \eta_\varepsilon(x) \big\} \psi(\tau,x,v) \Big) 
 	 \rightarrow  \nabla_{\!x}\cdot \big(v \psi(\tau,x,v) \big)
 	= v\cdot\nabla_{\!x}\psi(\tau,x,v) ,
 	$$
 	 and
 	 $$
 	\bar{Q}^{\varepsilon}_{\!\Delta}[\psi](\tau,x,v) \eqdef  \frac{2}{\varepsilon^2} \int_{\mathbb{R}^3}\! \big[\psi(\tau,x,v\!-\!\varepsilon u)-\psi(\tau,x,v)\big]\xi(u) du
 	 \rightarrow  \Delta_v \psi(\tau,x,v) ,
 	 $$
 	in $L^1\!\left((0,t)\times\Omega\times\mathbb{R}^3 \right)$, for each $t\in (0,T]$.  In addition, as $a \rightarrow 0$, 
 	\begin{equation} \label{main-thm-pf-5}
 	\gamma_+ \psi = (1\!-\!a)  \mathcal{R^*}\!\left[\gamma_- \psi\right] 
 	 \rightarrow  \mathcal{R^*}\!\left[\gamma_- \psi\right] .
 	\end{equation}
 	Combined with (\ref{main-thm-pf-1}) and (\ref{main-thm-pf-2}), we see that the limit $f$ satisfies the weak formulation (\ref{Def:weak-formulation}) of  Definition~\ref{Def:weak-sol}  by taking  $\varepsilon,a \rightarrow 0$. 
 	 Therefore, summing up Step 1--3, we obtain that $f$  is indeed a weak solution of   (\ref{linearized-eq-homo}) with (\ref{initial}) and (\ref{specular}), this concludes the proof of the existence.
 	 
 	\noindent \textbf{Step 4: Uniqueness.} 
 	With the $L^1\cap L^\infty$ estimate (\ref{L^infty-bound}) and (\ref{L^1-bound}), we can easily get the uniqueness of the weak solution in a similar manner to the proof of Corollary~\ref{Cor:uniqueness}.  
 \end{proof}
 \subsection{Well-definedness of the trace}
 Additionally, a classical trace result by Ukai \cite{MR882376} also assures us that the trace functions  $\gamma_\pm f^{\varepsilon,a,n}$  are bounded in $L^\infty$ independently of $\varepsilon$, $n$ and 
 $a$ in light of the uniform $L^\infty$ bound for $f^{\varepsilon,a,n}$, so that they have enough regularity to pass to the limit. This allows us to take a subsequence such that $$\gamma_\pm f^{\varepsilon,a,n} \rightarrow  g_\pm  =: \gamma_\pm f$$ weakly-* in $L^\infty(\Sigma_\pm^{ T})$, in which sense the trace $\gamma_\pm f$  is well-defined.

 \subsection{Recovering the initial-boundary condition}
 Let us rewrite the linearized \textit{Landau equation} (\ref{linearized-eq-homo}) as  $\mathcal{L}_g f =0$, where  $\mathcal{L}_g$ denotes the linear Landau operator 
 $$\mathcal{L}_g f \eqdef   \partial_t f + v\cdot\nabla_{\!x} f - a_g\cdot\nabla_{\!v}f - \nabla_v\cdot\big(\sigma_{\!G}\nabla_{\!v} f\big) .$$
 Suppose that the following Green's identity is valid for the solution $f$ and any $\psi$ as the test function in Definition~\ref{Def:weak-sol} with  $\gamma_+ \psi = \mathcal{R^*}\!\left[\gamma_- \psi\right]$  on $\Sigma_+^{ T}$: 
 \begin{align*}
 & \langle \mathcal{L}_g f, \psi \rangle_{Q_t}  +  \langle \mathcal{L}_g^{ *} \psi, f \rangle_{Q_t} \\[2pt]
 =\; & \int_{\Omega\times\mathbb{R}^3}\! \big[(f\psi)(t,x,v) - (f\psi)(0,x,v)\big]dxdv  + 
 \int_{\Sigma^{t}}\!\gamma f \gamma\psi  (v\!\cdot\! n_x)  dS_xdvd\tau ,
 \end{align*}
 where $\langle \cdot, \cdot \rangle_{Q_t}$ stands for the natural duality pairing defined as the integration of the product over  $Q_t \!\eqdef  (0,t)\times\Omega\times\mathbb{R}^3$. The adjoint operator reads 
 $$\mathcal{L}_g^{ *} \psi  \eqdef   -\partial_t \psi - v\cdot\nabla_{\!x} \psi + \nabla_v\cdot\big(a_g \psi\big) - \nabla_v\cdot\big(\sigma_{\!G}\nabla_{\!v} \psi\big) ,$$
 and the boundary term is written out as 
 \begin{align*}
 & \int_{\Sigma^{t}}\!\gamma f \gamma\psi  (v\!\cdot\! n_x)  dS_xdvd\tau \\
 =\; & \int_{\Sigma_+^{ t}}\!\gamma_{\!+} f  \gamma_{\!+}\psi  |v\!\cdot\! n_x|  dS_xdvd\tau
  - \int_{\Sigma_-^{ t}}\!\gamma_{\!-} f  \gamma_{\!-}\psi  |v\!\cdot\! n_x|  dS_xdvd\tau \\
 =\; & \int_{\Sigma_+^{ t}}\!\gamma_{\!+} f  \mathcal{R^*}\!\left[\gamma_{\!-}\psi\right] |v\!\cdot\! n_x|  dS_xdvd\tau
  - \int_{\Sigma_-^{ t}}\!\gamma_{\!-} f  \gamma_{\!-}\psi  |v\!\cdot\! n_x|  dS_xdvd\tau \\
 =\; & \int_{\Sigma_-^{ t}}\! \big[ \mathcal{R}\!\left[\gamma_{\!+}f\right] - \gamma_{\!-} f  \big] \gamma_{\!-}\psi  |v\!\cdot\! n_x|  dS_xdvd\tau .
 \end{align*}
 Taking into account that $f$ is a weak solution to the problem (\ref{linearized-eq-homo}), (\ref{initial}), and (\ref{specular}) in the sense of distributions and satisfies the weak formulation (\ref{Def:weak-formulation}), it is straightforward to deduce that
 \begin{equation*}
 \int_{\Omega\times\mathbb{R}^3}\! \big[f_0 - f(0)\big]\psi(0)  dxdv   +
 \int_{\Sigma_-^{ t}}\! \big[ \mathcal{R}\!\left[\gamma_{\!+}f\right]  -  \gamma_{\!-} f  \big] \gamma_{\!-}\psi  |v\cdot n_x|  dS_xdvd\tau  = 0 ,
 \end{equation*}
 for any $\psi$  as in Definition~\ref{Def:weak-sol}. Therefore, $f$ verifies the initial condition  $f(0,x,v) = f_0(x,v)$  and the specular reflection boundary condition  $\gamma_- f = \mathcal{R}\left[\gamma_+ f\right]$  as desired. 
 
 The above result implies that, for any test function of  Definition~\ref{Def:weak-sol} without the dual boundary condition (\ref{dual-specular-BC}), the weak formulation (\ref{Def:weak-formulation}) can be replaced by
 \begin{align*} \label{Def:weak-formulation-alt}
 \begin{split}
 & \iint_{\Omega\times\mathbb{R}^3} \big[f(t)\psi(t) - f_0 \psi(0)\big] dxdv  
  + \int_{\Sigma^{t}}\!\gamma f \gamma\psi  (v\!\cdot\! n_x)  dS_xdvd\tau \\[3pt]
 & = \int_0^t\! \iint_{\Omega\times\mathbb{R}^3} f(\tau,x,v) \Big[ \partial_t \psi + v\cdot\nabla_{\!x}\psi - \nabla_v\cdot\big(a_g \psi \big) + \nabla_v\cdot\big(\sigma_{\!G}\nabla_{\!v}\psi \big) \Big] dxdvd\tau 
 \end{split}
 \end{align*}
 for all  $t\in [0,T]$, where  $\gamma_\pm f$  satisfies the boundary condition (\ref{specular}).

 \section{$L^2$ decay estimate}
 \label{sec: L2decay}
 Thanks to the work in the previous section, we are now equipped with the wellposedness of \eqref{eq Agf only} with the \textit{specular reflection boundary condition} in the sense of distribution. So we may associate a continuous semigroup of linear and bounded operators $U(t)$ such that $f(t) = U(t)f_0$ is the unique weak solution of \eqref{eq Agf only}.
 Then by the Duhamel principle, the solution $\bar{f}$ of the whole linearized equation \eqref{linearized eq} can further be written as$$
 \bar{f}(t) = U(t)\bar{f}_{0} + \int_{0}^{t} U(t-s) \bar K_{g} \bar{f}(s) ds.$$
 
 The main question in this section is whether we can further obtain the $L^2$ decay estimates for the solution $\bar{f}$. For the notational simplicity we use $f$ to denote $\bar{f}$, solutions to the whole linearized equation \eqref{linearized eq}, throughout this section. 
 
 For the $L^2$ decay theory, we will work directly on the bounded domain $\Omega$ and obtain the estimates by following a constructive method in the same manner as Proposition 4.1 of \cite{MR3712934}, that is, by choosing test functions suitably, the $L^2$ norm of macro-components can be controlled by the micro-components. Particularly, to handle the boundary terms, we choose the Burnette functions as orthogonal bases for the micro-components, so that the boundary integral can be reformulated in a more delicate way.

 \begin{remark}
 The proof is actually a variation of the one in Section 4 of \cite{2016arXiv161005346K} modified for specular reflection boundary case. In particular, we need to instead consider the existence of solutions to the elliptic problem with certain boundary conditions, under additional assumption of conservation law of angular momentum \eqref{angcon}, in case the domain $\Omega$ has any rotational symmetry.
 Alternatively, we might also adapt the semigroup-compactness method/contradiction argument in \cite{MR2679358}, although the former method is more preferred.
\end{remark}
Throughout this section, we consider the following linearized \textit{Landau equation}
 \begin{equation}\label{linear}
 \partial_tf+v\cdot \nabla_x f+Lf=\Gamma(g,f).
 \end{equation}
 The initial-boundary condition of $f$ is given by
 \begin{equation}\label{ib}
 \begin{cases}
 f(0,x,v)=f_0(x,v), \text{ if }x\in\Omega \text{ and }v\in\rth,\\
 f(t,x,v)=f(t,x,v-2(v\cdot n_x)n_x),\text{ if }x\in\partial\Omega\text{ and }v\cdot n_x<0.
 \end{cases}
 \end{equation}
 We note that the linear Landau operator $L$, given by \eqref{L}, is a self-adjoint nonnegative operator in $L^2$. The null space $N$ of $L$ is spanned by $$ \mu^{1/2},\  v^j \mu^{1/2}\ (j=1,2,3),\ \frac{|v|^2}{2} \mu^{1/2},$$ which are known as the collision invariants. We normalize these invariants and define \begin{equation}\label{generators}
 \begin{split}
 \chi_0&= \mu^{1/2},\\
 \chi_j&=v^j \mu^{1/2},\ j=1,2,3\\
 \chi_4&=\frac{|v|^2-3}{\sqrt{6}} \mu^{1/2}.
 \end{split}
 \end{equation}
 Then we define the projection to the null space $N$ by $P$ as follows:
 \begin{equation}
 Pf=\sum_{k=0}^{4}\langle f,\chi_k\rangle \chi_k,
 \end{equation} where $\langle g,h\rangle \eqdef  \int_\rth ghdv.$
 We will also use the following Burnette functions of the space $N^{\perp}:$
 \begin{equation}\label{Aj}A_j(v)=v^j\frac{|v|^2-5}{\sqrt{10}}\sqrt{\mu},\end{equation}
 and
 \begin{equation}\label{Bij}
 B_{kl}(v)=\left(v^kv^l-\frac{\delta_{kl}}{3}|v|^2\right)\sqrt{\mu},\end{equation}
 where $k,j=1,2,3,$ and $\delta_{kl}=1$ if $ k=l$ and $=0$ otherwise.
 
 \subsection{Technical lemmas}
 For the nonlinear collision operator $\Gamma(f,h)$, we have the following known estimates.
 \begin{lemma}[Theorem 2.8 of \cite{2016arXiv161005346K}]
 	\label{2.8}
 	Let $\Gamma$ be defined as in \eqref{Gamma}.
 	For every $\vartheta\in\mathbb{R}$, there exists $C_{\vartheta}$ such
 	that
 	\begin{equation}
 	\label{Gammav}
 	|\langle w^{2 \vartheta} \Gamma[g_{1},g_{2}], g_{3} \rangle|\le C_{\vartheta}|g_{1}|_{\infty}|g_{2}|_{\sigma, \vartheta}
 	|g_{3}|_{\sigma, \vartheta},
 	\end{equation}
 	and
 	\begin{equation}
 	\label{Gammax}%
 	\left|  \left(  w^{2 \vartheta} \Gamma[g_{1},g_{2}], g_{3} \right)  \right| \le C_{\vartheta}\|g_{1}\|_{\infty}\|g_{2}\|_{\sigma, \vartheta}\|g_{3}\|_{\sigma, \vartheta}.
 	\end{equation}
 \end{lemma}
 \begin{proof}
 	The proof for \eqref{Gammav} is the same as the one for (2.11) in Theorem 2.8 of \cite{2016arXiv161005346K}. Then we use \eqref{Gammav} and H\"older's inequality and obtain
 	\begin{equation*}
 	\begin{split}
 	\left|  \left(  w^{2 \vartheta} \Gamma[g_{1},g_{2}], g_{3} \right)  \right|& = \int_\Omega|\langle w^{2 \vartheta} \Gamma[g_{1},g_{2}], g_{3} \rangle|dx\\
 	&  \le\int_\Omega C_{\vartheta}|g_{1}|_{\infty}|g_{2}|_{\sigma, \vartheta}|g_{3}|_{\sigma, \vartheta} dx\\
 	& \le C_{\vartheta}\|g_{1}\|_{\infty}\|g_{2}\|_{\sigma, \vartheta}\|g_{3}\|_{\sigma, \vartheta}.
 	\end{split}
 	\end{equation*}
 	Thus we have \eqref{Gammax}.
 \end{proof}
 
 Also, we have the coercive estimate on $L$:
 
 \begin{lemma}
 	[Lemma 5 in \cite{guo2002landau}]\label{Lemma : Guo2002 lemma5} 
 	Let $L$ be defined as in \eqref{L}. Then there is $\delta>0$, such that
 	\[
 	\langle Lg,g \rangle_v\ge\delta|(I-P) g|_{\sigma}^{2}.
 	\]
 	
 \end{lemma}

\subsection{Macro-micro decomposition}
 Then we can obtain the $L^2$ decay estimates for the solutions $f$ to \eqref{linear}:
 \begin{theorem}\label{linear l2}
 	Let $f$ be the weak solution of \eqref{linear} with initial-boundary value conditions \eqref{ib}, which satisfies the conservation laws \eqref{conservation laws}, and \eqref{angcon} if $\Omega$ has a rotational symmetry. 
 	Suppose that $\|g\|_{\infty} < \epsilon$ for some $\epsilon>0$.
 	For any $\vartheta\in 2^{-1}\mathbb{N} \cup\{0\}$, there exist $C$ and $\epsilon=\epsilon(\vartheta)>0$ such that 
 	\begin{equation}	\label{Eq : energy estimate linear}
 	\sup_{0 \le s< \infty}\mathcal{E}_{\vartheta}(f(s)) \le C 2^{2\vartheta} \mathcal{E}_{\vartheta}(0),
 	\end{equation}
 	and
 	\begin{equation}	\label{Eq : decay estimate linear}
 	\|f(t)\|_{2,\vartheta} \le C_{\vartheta,k} \left(\mathcal{E}_{\vartheta+\frac{k}{2}}(0) \right)^{1/2}\left(  1+ \frac{t}{k}\right)^{-k/2}%
 	\end{equation}
 	for any $t>0$ and $k\in\mathbb{N}$, where $\mathcal{E}_\vartheta(f(t))$ is defined in \eqref{E}.
 	
 \end{theorem}
 
 In order to prove this theorem, the key estimate that we need to establish is the following proposition. Here we estimate $Pf$ in terms of $(I-P)f$:
 \begin{proposition}\label{Macromicro}Assume $\|g\|_{\infty} < \epsilon$ for some $\epsilon>0$.
 	Let $f$ be a weak solution of \eqref{linear} and \eqref{initial} with \eqref{conservation laws}, and \eqref{angcon} if $\Omega$ has a rotational symmetry. Then there exist $C$ and a function $|\eta(t)| \le C\|f(t)\|_{2}^{2}$, such
 	that 
 	\[
 	\int_{t_0}^{t} \|P f(\tau)\|_{\sigma}^{2}ds \le\eta(t)-\eta(t_0) + C \int
 	_{t_0}^{t} \|(I-P)f(\tau)\|_{\sigma}^{2}.
 	\]
 \end{proposition}Using this proposition, we will obtain the coercivity of the linearized Landau operator $L$ later.
 
 \begin{proof}
 	The proof will be given in 7 steps as follows.
 	The definition of a weak solution to \eqref{linear} is defined in \eqref{L2weak}.
 	The rest of this section is devoted to the proof of Theorem \ref{linear l2}. In terms of the macro-micro decomposition, we define \begin{equation}
 	\begin{split}
 	a(t,x)&=\langle f(t,x,v),\chi_0\rangle \\
 	b^j(t,x)&=\langle f(t,x,v),\chi_j\rangle, \ j=1,2,3,\\
 	c(t,x)&=\langle f(t,x,v),\chi_4\rangle,\\
 	d(t,x,v)&=(I-P)f(t,x,v),
 	\end{split}
 	\end{equation}
 	where the generators $\{\chi_j\}_{j=0}^4 $ are defined in \eqref{generators}.
 	Then we have 
 	\begin{equation}f=\left(a\chi_0+\sum_{j=1}^3 b^j\chi_j+c\chi_4\right)+d.
 	\end{equation}
 	The conservation laws of mass and energy \eqref{conservation laws} implies that 
 	$$\int_\Omega adx=0$$ and $$\int_\Omega cdy=0.$$ 
 	
 	We will derive the estimates of $a,b,c$ in terms of $d$. 
 	We write the linear \textit{Landau equation} \eqref{linear} in the following weak formulation:
 	\begin{multline}\label{L2weak}
 	\int_{\Omega\times \rth}\{\psi f(t)-\psi f(s)\}dxdv-\int_s^t \int_{\Omega\times \rth}f\partial_\tau\psi dxdvd\tau\\
 	=\int_s^t  \int_{\Omega\times \rth}(v\cdot\nabla_x\psi )fdxdvd\tau-\int_s^t \int_\gamma f\psi d\gamma d\tau-\int_s^t  \int_{\Omega\times \rth}\psi (Lf)dxdvd\tau\\
 	+\int_{s}^{t}\int_{\Omega\times \rth}\psi\Gamma(g,f)dxdvd\tau\eqdef  I_1+I_2+I_3+I_4,
 	\end{multline}
 	where $d\gamma\eqdef  (v\cdot n_x)dS_xdv,$ and $\psi\in C^\infty \cap L^2((s,t)\times \Omega\times \rth)$ is a test function.
 	
 	In Step 1 through Step 3 below, we consider a $t$-mollification of the functions $a,b,$ and $c$ so that they are smooth in $t$. For notational simplicity we omit the explicit parameter of the regularization.
 	
 	\textbf{Step 1. Choosing a test function $\psi= \phi(x)\sqrt
 		\mu$ in \eqref{L2weak}.}
 	In this case, we have $\int\sqrt{\mu}Lf=\int\sqrt{\mu}\ \Gamma(g,f)=0.$ Thus we have 
 	\[
 	\int_{\Omega} [a(t+ \epsilon) - a(t)]\phi(x) = \int_{t}^{t+ \epsilon}\int_{\Omega}(b\cdot
 	\nabla_{x})\phi(x).
 	\]
 	Therefore, by letting $\epsilon\rightarrow 0$, we have
 	\[
 	\int_{\Omega} \phi\partial_{t} a dx=  \int_{\Omega%
 	}(b \cdot\nabla_{x})\phi dx.
 	\]
 	By taking $\phi=1$, we have $$\int_{\Omega} \partial_{t} a(t)
 	dx =0$$ for all $t>0$. On the other hand, for all $\phi(x)\in H^{1}%
 	(\Omega)$, we have
 	\[
 	\left|  \int_{\Omega} \phi(x) \partial_{t} a dx \right|
 	\lesssim\|b\|_{2} \|\phi\|_{H^{1}}.
 	\]
 	Therefore, for all $t>0$, $$\|\partial_{t} a(t)\|_{H^{-1}_0} \lesssim
 	\|b(t)\|_{2}.$$ Since $\int_{\Omega} \partial_{t} a dx =0$ for all
 	$t>0$, by standard elliptic theory, there exists a unique weak solution $\Phi_a$ to the following Poisson equation
 	\begin{equation}
 	\begin{split}
 	-\Delta\Phi_a=&\partial_ta(t),\ \text{in}\ \Omega,\\
 	\frac{\partial \Phi_a}{\partial n}=&0, \ \text{on}\ \partial \Omega\\
 	\int_\Omega\Phi_a dx=&0.
 	\end{split}
 	\end{equation} Moreover,
 	we have
 	\begin{equation}
 	\label{dtphia}\|\nabla_{x} \Phi_a \|_{2} =
 	\|\Phi_{a}\|_{H^{1}} \lesssim\|\partial_{t} a(t)\|_{H^{-1}_0}%
 	\lesssim\|b(t)\|_{2}.
 	\end{equation}

 	\textbf{Step 2. Choosing a test function $\psi=
 		\phi(x)\chi_i$,  in \eqref{L2weak}.}  The left-hand side of \eqref{L2weak} is equal to
 	$$(LHS)=\int_{\Omega} [b^{i}(t+ \epsilon
 	)-b^{i}(t)]\phi dx.$$
 	On the other hand, note that $$v\cdot\nabla_x\psi=\partial_i\phi (\chi_0+\frac{\sqrt{6}}{3}\chi_4)+\sum_{k=1}^3 \partial_k\phi B_{ki},$$ where $B_{ki}$ is defined in \eqref{Bij}. Thus,
 	$$I_1=\int_{t}^{t+\epsilon}\int_{\Omega}
 	\partial_{i} \phi[a+\frac{\sqrt{6}}{3}c]dxd\tau + \int_{t}^{t+\epsilon}\iint_{\Omega%
 		\times\mathbb{R}^{3}}\sum_{k=1}^{3} \langle B_{ki},d\rangle\partial_jdxdvd\tau,$$ because $PB_{kl}=0.$
 	Additionally, from the fact that $\langle Lf,\chi_i\rangle=0$, we have $I_3=0$. Therefore, we have
 	\begin{multline*}\int_{\Omega} [b^{i}(t+ \epsilon
 	)-b^{i}(t)]\phi dx\\=\int_{t}^{t+\epsilon}\int_{\Omega}
 	\partial_{i} \phi[a+\frac{\sqrt{6}}{3}c]dx d\tau+ \int_{t}^{t+\epsilon}\iint_{\Omega%
 		\times\mathbb{R}^{3}}\sum_{k=1}^{3} \langle B_{ki},d\rangle\partial_jdxdvd\tau\\
 	+\int_{t}^{t+\epsilon}\int_\gamma f\phi\chi_i d\gamma d\tau+ \iint_{\Omega\times\rth} \phi \chi_i \Gamma(g,f)dxdvd\tau.
 	\end{multline*}
 	Therefore, by taking $\epsilon\rightarrow 0$, we have
 	\begin{multline*}\int_{\Omega} \partial_tb^i\phi dx=\int_{\Omega}
 	\partial_{i} \phi[a+\frac{\sqrt{6}}{3}c]dx +\iint_{\Omega%
 		\times\mathbb{R}^{3}}\sum_{k=1}^{3} \langle B_{ki},d\rangle\partial_jdxdv\\
 	+\int_\gamma f\phi\chi_i d\gamma + \int_{\Omega\times\rth} \phi \chi_i \Gamma(g,f)dxdv.
 	\end{multline*} Now for fixed $t>0$, let $\phi=\Phi^i_b$ where $\Phi^i_b$ is a solution of the following Poisson equation:
 	\begin{equation}
 	\begin{split}
 	-\Delta\Phi_b=&\partial_tb(t),\ \text{in}\ \Omega,\\
 	\Phi_b\cdot n=&0, \ \text{on}\ \partial \Omega\\
 	\partial_n\Phi_b =&(\partial_n\Phi_b\cdot n)n  \ \text{on}\ \partial \Omega.
 	\end{split}
 	\end{equation} The existence of such solutions is given in Appendix A of [YZ]. Then after summation on $i=1,2,3$, we have that the boundary integral $I_2$ is equal to 
 	$$\sum_{i=1}^3 \int_\gamma f\chi_i\Phi_b^i d\gamma=2\int_{\partial \Omega}\int_{v\cdot n>0}(\Phi_b\cdot n)(v\cdot n)^2M^{1/2}fdvdS_x=0.$$
 	By Lemma \ref{2.8}, we have
 	\[
 	\left\|  \iint_{\Omega \times\mathbb{R}^{3}}\phi\chi_{i} \Gamma(g,f)(t)
 	dxdv \right\|  \le C \|g\|_{\infty}\|f\|_{\sigma}\|\phi \chi_{i}
 	\|_{\sigma} \le\|g\|_{\infty}\|f\|_{\sigma}\|\phi\|_{2}.
 	\]
 	Thus, \begin{multline*}
 	\sum_{i=1}^3 \int_\Omega |\nabla_x\Phi^i_b|^2dx = \sum_{i=1}^3 \int_\Omega-\Delta \Phi_b^i\Phi_b^idx=\sum_{i=1}^3 \int_\Omega\Phi_b^i\partial_t b^idx\\
 	\lesssim ||\nabla_x\Phi_b||_2(||a||_2+||c||_2+||d||_2+||g||_\infty ||f||_\sigma).
 	\end{multline*}
 	Together with $||\Phi_b||_2\lesssim ||\nabla_x\Phi_b||_2$, we have 
 	\begin{equation}
 	\label{dtphib}||\Phi_b||_{H^1}\lesssim(||a||_2+||c||_2+||d||_2+||g||_\infty||f||_\sigma).
 	\end{equation}

 	\textbf{Step 3. Choosing a test function $\psi= \phi(x) \chi_4$ in \eqref{L2weak}.} 
 	From \eqref{L2weak}, we have
 	$$(LHS)=\int_\Omega (c(t+\epsilon)-c(t))\phi dx.$$
 	In this case, $$v\cdot \nabla_x \psi = \sum_{k=1}^3 \frac{\sqrt{6}}{3}(\partial_k \phi)\chi_k +\frac{\sqrt{15}}{3}\partial_k\phi A_k,$$ where $A_k$ is defined in \eqref{Aj}.  Then we have
 	$$I_1=\int_{t}^{t+\epsilon}\int_{\Omega}\left(
 	\frac{\sqrt{6}}{3}b\cdot \nabla_x\phi dxd\tau + \sum_{k=1}^{3}\frac{\sqrt{15}}{3}\partial_k\phi \langle A_k,d\rangle \right)dxd\tau,$$ and $I_2=0$. Additionally, from the fact that $\langle Lf,\chi_4\rangle=0$, we have $I_3=0$. Consequently, we have
 	\begin{multline*}
 	\int_{\Omega} \phi(x) \partial_{t} c(t)
 	= \frac{\sqrt{6}}{3}\int_{\Omega} b(t)\cdot\nabla_{x} \phi dx+\int_\Omega\sum_{k=1}^{3}\frac{\sqrt{15}}{3}\partial_k\phi \langle A_k,d\rangle dx \\+\iint_{\Omega \times
 		\mathbb{R}^{3}}\phi\Gamma(g,f)(t) \chi_4dxdv.
 	\end{multline*}
 	Note that 
 	$$\left|\int_\Omega\sum_{k=1}^{3}\partial_k\phi \langle A_k,d\rangle dx\right|\lesssim ||d||_2||\nabla_x\phi||_2.$$
 	Also, by Lemma \ref{2.8}, we have
 	\[
 	\left\|  \iint_{\Omega \times\mathbb{R}^{3}}\phi\Gamma(g,f)(t) \chi_4 dxdv\right\|  \le\|g\|_{\infty
 	}\|f\|_{\sigma}\|\phi\|_{2}.
 	\]
 	Now for fixed $t>0$, let $\phi=\Phi_c$ where $\Phi_c$ is a solution of the following Poisson equation:
 	\begin{equation}
 	\begin{split}
 	-\Delta\Phi_c=&\partial_tc(t)\ \text{in}\ \Omega,\\
 	\frac{\partial\Phi_c}{\partial n}=&0 \ \text{on}\ \partial \Omega.
 	\end{split}
 	\end{equation} Then 
 	\begin{multline*} \int_\Omega|\nabla_x\Phi_c|^2dx=\int_\Omega -\Delta\Phi_c \Phi_cdx=\int_\Omega \Phi_c \partial_tcdx\\
 	\lesssim ||\nabla_x\Phi_c||_2(||b||_2+||d||_2+||g||_\infty||f||_\sigma).\end{multline*}
 	Therefore, for all $t>0$,
 	\begin{equation}
 	\label{dtphic}||\Phi_c||_{H^1}\lesssim ||b||_2+||d||_2+||g||_\infty||f||_\sigma.
 	\end{equation}

 	\textbf{Step 4. Estimate of $c$.} We choose a test function $$\psi=\psi_c= \left( |v|^{2}-5\right)
 	\sqrt{\mu} v\cdot\nabla_{x} \phi_{c}=\sqrt{10}\sum_{j=1}^3A_j \partial_j \phi_c,$$ where $A_j$ is defined in \eqref{Aj} and $\phi_c$ is a solution of the following Poisson system: \begin{equation}
 	\begin{split}
 	-\Delta\phi_c=&c(t)\ \text{in}\ \Omega,\\
 	\frac{\partial\phi_c}{\partial n}=&0 \ \text{on}\ \partial \Omega.
 	\end{split}
 	\end{equation} We have\begin{equation}
 	\label{weak2}\begin{split}
 	-\int_0^t v\cdot \nabla_x\psi fdxdvd\tau
 	&  = -\iint_{\Omega \times\mathbb{R}^{3}} (\psi f(t) -\psi f(0))\\
 	&  \quad- \int_0^t\iint_{\Omega\times \rth}\psi \partial_\tau f dxdvd\tau-\int_0^t\int_\gamma \psi f d\gamma d\tau\\
 	&\quad+ \int_{0}^{t} \iint_{\Omega \times\mathbb{R}^{3}}
 	\psi\Gamma(g,f)dxdvd\tau.
 	\end{split}
 	\end{equation} Note that we have
 	\begin{multline*}
 	v\cdot \nabla_x \psi=\sum_{j,k=1}^3 \partial_{jk}^2\phi_cv^kv^j(|v|^2-5)\sqrt{\mu}\\
 	=\frac{5\sqrt{6}}{3}\Delta\phi_c\chi_4+\sum_{j,k=1}^3 \partial_{jk}^2\phi_c(I-P)(v^kv^j(|v|^2-5)\sqrt{\mu}).
 	\end{multline*}
 	Therefore, the left-hand side of \eqref{weak2} is now equal to
 	\begin{multline*}
 	-\int_0^t v\cdot \nabla_x\psi fdxdvd\tau\\=\frac{5\sqrt{6}}{3}\int_0^t\int_\Omega -\Delta\phi_c cdxd\tau-\sum_{j,k=1}^3\int_0^t\int_\Omega \partial_{jk}^2\phi_c\langle d,v^kv^j(|v|^2-5)\sqrt{\mu}\rangle dxd\tau\\
 	=\frac{5\sqrt{6}}{3}\int_0^t\int_\Omega c^2 dxd\tau +K_1,
 	\end{multline*}
 	where, for any $\epsilon>0$,
 	$$|K_1|\leq \epsilon^2\int_0^t ||c||^2_2d\tau +\frac{1}{\epsilon^2}\int_0^t ||d||^2_2 d\tau.$$
 	We will now estimate each term on the right-hand side of \eqref{weak2}. By the estimate \eqref{dtphic}, we have
 	\begin{multline*}
 	\left|\int_0^t\iint_{\Omega\times \rth}\psi\partial_\tau f dxdvd\tau\right|\lesssim \sum_{j=1}^3 \int_0^t \int_\Omega |\partial_{\tau j}^2\phi_c\langle A_j,d\rangle |dxd\tau\\
 	\lesssim \int_0^t ||\partial_\tau \phi_c||_{H^1}||d||_2d\tau
 	\lesssim \int_0^t (||b||_2+||d||_2+||g||_\infty||f||_\sigma)||d||_2d\tau\\
 	\lesssim \epsilon^2\int_0^t ||b||^2_2d\tau+C_\epsilon\int_0^t( ||d||^2_2+||g||^2_\infty||f||^2_\sigma)d\tau.
 	\end{multline*}
 	By the boundary conditions of $\phi_c$ and $f$, we have 
 	\begin{multline*}
 	\int_0^t\int_\gamma \psi f d\gamma d\tau\\
 	=\sum_{j=1}^3 \int_{\partial \Omega}\partial_j\phi_c \left(\int_{v\cdot n>0}+\int_{v\cdot n<0}\right)(v\cdot n) v^j(|v|^2-5)\sqrt{\mu}fdvdS_x\\
 	=\sum_{j=1}^3 \int_{\partial \Omega}\partial_j\phi_c \int_{v\cdot n>0}(v\cdot n) v^j(|v|^2-5)\sqrt{\mu}fdvdS_x\\
 	+\sum_{j=1}^3 \int_{\partial \Omega}\partial_j\phi_c \int_{v\cdot n<0}(v\cdot n) v^j(|v|^2-5)\sqrt{\mu}fdvdS_x\\
 	=\sum_{j=1}^3 \int_{\partial \Omega}\partial_j\phi_c \int_{v\cdot n>0}(v\cdot n) v^j(|v|^2-5)\sqrt{\mu}fdvdS_x\\
 	+\sum_{j=1}^3 \int_{\partial \Omega}\partial_j\phi_c \int_{v\cdot n>0}(-v\cdot n) (v^j-2(v\cdot n)n^j)(|v|^2-5)\sqrt{\mu}fdvdS_x\\
 	=2\sum_{j=1}^3 \int_{\partial \Omega}\frac{\partial\phi_c}{\partial n} \int_{v\cdot n>0}(v\cdot n)^2(|v|^2-5)\sqrt{\mu}fdvdS_x
 	=0.
 	\end{multline*}
 	Also,
 	\[
 	\iint_{\Omega \times\mathbb{R}^{3}} \psi\Gamma(g,f)dxdv \le C\|g\|_{\infty
 	}\|f\|_{\sigma}\|c\|_{2}\le\epsilon\|c\|_{2}^{2} + C_{ \epsilon
 	}\|g\|_{\infty}^{2}\|f\|_{\sigma}^{2}.
 	\]
 	For a small $\epsilon>0$, we can absorb $\epsilon^2\|c\|_{2}^{2}$ on the RHS to the LHS.
 	Altogether, we have
 	\begin{multline}
 	\label{Eq : estimate c}%
 	\int_{0}^{t} \|c(s)\|^{2} ds\le C \int_{\Omega\times \rth}(-\psi_cf(t)+\psi_cf(0))dxdv \\
 	+ \int_{0}^{t} C_{\epsilon}\left\{||d||_2^2 + \|g\|_{\infty}^{2}\|f\|_{\sigma}^{2}\right\}  +\epsilon\|b\|_{2}^{2} ds.
 	\end{multline}

 	\textbf{Step 5. Estimate of $b$.} 
 	Now for fixed $t>0$, let $\phi=\phi^i_b$ where $\phi^i_b$ is a solution of the following Poisson equation:
 	\begin{equation}
 	\begin{split}
 	-\Delta\phi_b=&b(t),\ \text{in}\ \Omega,\\
 	\phi_b\cdot n=&0, \ \text{on}\ \partial \Omega\\
 	\partial_n\phi_b =&(\partial_n\phi_b\cdot n)n  \ \text{on}\ \partial \Omega.
 	\end{split}
 	\end{equation} The existence of such solutions is given in Appendix A of [YZ].  Now set
 	\begin{multline}\label{phib}\psi=\psi_b= \sum_{i,j=1}^3 \partial_j \phi_b^iv^iv^j\sqrt{\mu}    -\sum_{i=1}^3 \partial_i \phi_b^i \frac{|v|^2-1}{2}\sqrt{\mu}\\
 	=\sum_{i,j=1}^3\partial_j \phi_b^iB_{ij}-\sum_{i=1}^3\frac{\sqrt{6}}{6} \partial_i \phi_b^i \chi_4,
 	\end{multline} where $B_{ij}$ is defined in \eqref{Bij}.
 	Then we have
 	$$
 	v\cdot \nabla_x \psi=\sum_{i=1}^3 \Delta\phi^i_B\chi_i+\sum_{i,j,k=1}^3\partial^2_{jk}\phi_b^i(I-P)(v^iv^jv^k\sqrt{\mu}).  
 	$$
 	Therefore, the left-hand side of \eqref{weak2} is now equal to
 	\begin{multline*}
 	-\int_0^t v\cdot \nabla_x\psi fdxdvd\tau\\=-\sum_{i=1}^3\int_0^t\int_\Omega b^i\Delta\phi_b^idxd\tau -\sum_{i,j,k=1}^3 \int_0^t\int_\Omega \partial^2_{jk}\phi_b^i \langle v^iv^jv^k\sqrt{\mu}, d\rangle dxd\tau\\
 	=\int_0^t\int_\Omega |b|^2 dxd\tau + K_2,
 	\end{multline*}
 	where, for any $\epsilon>0$,
 	$$|K_2|\leq \epsilon^2\int_0^t ||b||^2_2d\tau +\frac{1}{\epsilon^2}\int_0^t ||d||^2_2 d\tau.$$
 	We will now estimate each term on the right-hand side of \eqref{weak2}. Note that $\Phi_b=\partial_t\phi_b$. An integration by parts, \eqref{phib}, and \eqref{dtphib} yield that
 	\begin{multline*}
 	\left|\int_0^t\iint_{\Omega\times \rth}\psi\partial_\tau f dxdvd\tau\right|\lesssim  \int_0^t \int_\Omega (||c||_2+||d||_2) \|\partial_{\tau }\nabla_x\phi_b\|_2 dxd\tau\\
 	\lesssim \epsilon^2\int_0^t ||a||^2_2d\tau+C_\epsilon\int_0^t(||c||^2_2+ ||d||^2_2+||g||^2_\infty||f||^2_\sigma)d\tau.
 	\end{multline*}
 	Also, we have 
 	\begin{multline*}
 	\int_0^t\int_\gamma \psi f d\gamma d\tau\\
 	=\int_0^t\int_\gamma\left(\sum_{i,j=1}^3\partial_j\phi^i_bv^iv^j\sqrt{\mu}-\sum_{i=1}^3\partial_i\phi^i_b\frac{|v|^2-1}{2}\sqrt{\mu}\right)(v\cdot n)fdvdS_xd\tau\\
 	=\int_0^t\int_\gamma \sum_{i,j=1}^3\partial_j\phi^i_bv^iv^j\sqrt{\mu}(v\cdot n)fdvdS_xd\tau,
 	\end{multline*} because the integration on $v\cdot n>0$ cancels out the integration on $v\cdot n<0$ for the latter sum. 
 	Define $$K_{20}\eqdef   \int_\rth\sum_{i,j=1}^3\partial_j\phi^i_bv^iv^j\sqrt{\mu}(v\cdot n)fdv.$$ Using coordinate change, we let $n=(1,0,0)$ without loss of generality. Then we have
 	\begin{multline*}
 	K_{20}=\partial_1 \phi_b^1\int_\rth (v^1)^3\sqrt{\mu}fdv+\sum_{i,j=2}^3\partial_j\phi_b^i\int_\rth v^1v^iv^j\sqrt{\mu}fdv\\
 	+\sum_{j=2}^3\left(\partial_1\phi_b^j\int_\rth (v^1)^2v^j\sqrt{\mu}fdv+\partial_j\phi_b^1\int_\rth (v^1)^2v^j\sqrt{\mu}fdv\right).\end{multline*}
 	The first two integrals in $K_{20}$ is odd in $v^1$ and the specular reflection boundary condition of $f$ gives $f(v)=f(-v^1,v^2,v^3).$ Thus, the first two integrals are zero. The third integral in $K_{20}$ contains $\partial_1\phi^i_b$ for $i=2,3$ and this is zero by the boundary condition that $\partial_n\phi_b=(\partial_n\phi_b\cdot n)n=(\partial_1\phi_b^1,0,0).$ The last term in $K_{20}$ contains $\partial_j\phi_b^1$ for $j=2,3$ and this is zero because $\phi^1_b=0$ on $\partial\Omega$ by the boundary condtion $\phi_b\cdot n=0$. Therefore, $K_{20}=0$ and hence $$\int_0^t\int_\gamma \psi f d\gamma d\tau=0.$$
 	Also,
 	\[
 	\iint_{\Omega \times\mathbb{R}^{3}} \psi\Gamma(g,f)dxdv \le C\|g\|_{\infty
 	}\|f\|_{\sigma}\|b\|_{2}\le\epsilon\|b\|_{2}^{2} + C_{ \epsilon
 	}\|g\|_{\infty}^{2}\|f\|_{\sigma}^{2}.
 	\]
 	For a small $\epsilon>0$, we can absorb $\epsilon^2\|b\|_{2}^{2}$ on the RHS to the LHS.
 	Altogether, we have
 	\begin{multline}
 	\label{Eq : estimate b}%
 	\int_{0}^{t} \|b(s)\|^{2} ds\le C \int_{\Omega\times \rth}(-\psi_bf(t)+\psi_bf(0))dxdv \\
 	+ \int_{0}^{t} C_{\epsilon}\left\{||c||_2^2+||d||_2^2 + \|g\|_{\infty}^{2}\|f\|_{\sigma}^{2}\right\}  +\epsilon(||a||_2^2+\|b\|_{2}^{2} )ds.
 	\end{multline}

 	\textbf{Step 6. Estimate of $a$.} 
 	We choose a test function $$\psi= \psi_a= \left( |v|^{2}-10\right)
 	\sqrt{\mu}\sum_{j=1}^3\partial_j\phi_av^j =\sum_{j=1}^3(\sqrt{10}A_j-5\chi_j) \partial_j \phi_a,$$ where $A_j$ is defined in \eqref{Aj} and $\phi_a$ is a solution of the following Poisson system: \begin{equation}
 	\begin{split}
 	-\Delta\phi_a=&a(t)\ \text{in}\ \Omega,\\
 	\frac{\partial\phi_a}{\partial n}=&0 \ \text{on}\ \partial \Omega.
 	\end{split}
 	\end{equation} Note that we have
 	\begin{multline*}
 	v\cdot \nabla_x \psi=\sum_{j,k=1}^3 \partial_{jk}^2\phi_av^kv^j(|v|^2-10)\sqrt{\mu}\\
 	=5\Delta\phi_a\chi_0+\sum_{j,k=1}^3 \partial_{jk}^2\phi_a(I-P)(v^kv^j(|v|^2-10)\sqrt{\mu}).
 	\end{multline*}
 	Therefore, the left-hand side of \eqref{weak2} is now equal to
 	\begin{multline*}
 	-\int_0^t v\cdot \nabla_x\psi fdxdvd\tau\\=5\int_0^t ||a||^2_2 d\tau-\sum_{j,k=1}^3\int_0^t\int_\Omega \partial_{jk}^2\phi_a\langle d,v^kv^j(|v|^2-10)\sqrt{\mu}\rangle dxd\tau\\
 	=5\int_0^t ||a||^2_2 d\tau+K_3,
 	\end{multline*}
 	where, for any $\epsilon>0$,
 	$$|K_3|\leq \epsilon^2\int_0^t ||a||^2_2d\tau +\frac{1}{\epsilon^2}\int_0^t ||d||^2_2 d\tau.$$
 	We will now estimate each term on the right-hand side of \eqref{weak2}. By the estimate \eqref{dtphia}, we have
 	\begin{multline*}
 	\left|\int_0^t\iint_{\Omega\times \rth}\psi\partial_\tau f dxdvd\tau\right|\\
 	\lesssim \int_0^t ||\partial_\tau \phi_a||_{H^1}(||b||_2+||d||_2)d\tau
 	\lesssim \int_0^t (||b||_2+||d||_2)||b||_2d\tau\\
 	\lesssim \int_0^t( ||b||^2_2+||d||^2_2)d\tau.
 	\end{multline*}
 	By the boundary conditions of $\phi_a$ and $f$, we have 
 	\begin{multline*}
 	\int_0^t\int_\gamma \psi f d\gamma d\tau\\
 	=\sum_{j=1}^3 \int_{\partial \Omega}\partial_j\phi_c \left(\int_{v\cdot n>0}+\int_{v\cdot n<0}\right)(v\cdot n) v^j(|v|^2-10)\sqrt{\mu}fdvdS_x\\
 	=\sum_{j=1}^3 \int_{\partial \Omega}\partial_j\phi_c \int_{v\cdot n>0}(v\cdot n) v^j(|v|^2-10)\sqrt{\mu}fdvdS_x\\
 	+\sum_{j=1}^3 \int_{\partial \Omega}\partial_j\phi_c \int_{v\cdot n<0}(v\cdot n) v^j(|v|^2-10)\sqrt{\mu}fdvdS_x\\
 	=\sum_{j=1}^3 \int_{\partial \Omega}\partial_j\phi_c \int_{v\cdot n>0}(v\cdot n) v^j(|v|^2-10)\sqrt{\mu}fdvdS_x\\
 	+\sum_{j=1}^3 \int_{\partial \Omega}\partial_j\phi_c \int_{v\cdot n>0}(-v\cdot n) (v^j-2(v\cdot n)n^j)(|v|^2-10)\sqrt{\mu}fdvdS_x\\
 	=2\sum_{j=1}^3 \int_{\partial \Omega}\frac{\partial\phi_c}{\partial n} \int_{v\cdot n>0}(v\cdot n)^2(|v|^2-10)\sqrt{\mu}fdvdS_x
 	=0.
 	\end{multline*}
 	Also,
 	\[
 	\iint_{\Omega \times\mathbb{R}^{3}} \psi\Gamma(g,f)dxdv \le C\|g\|_{\infty
 	}\|f\|_{\sigma}\|a\|_{2}\le\epsilon\|a\|_{2}^{2} + C_{ \epsilon
 	}\|g\|_{\infty}^{2}\|f\|_{\sigma}^{2}.
 	\]
 	For a small $\epsilon>0$, we can absorb $\epsilon^2\|a\|_{2}^{2}$ on the RHS to the LHS.
 	Altogether, we have
 	\begin{multline}
 	\label{Eq : estimate a}%
 	\int_{0}^{t} \|a(s)\|^{2}_2 ds\le C \int_{\Omega\times \rth}(-\psi_af(t)+\psi_af(0))dxdv \\
 	+ \int_{0}^{t} C_{\epsilon}\left\{||d||_2^2 + \|g\|_{\infty}^{2}\|f\|_{\sigma}^{2}+\|b\|_{2}^{2}\right\} ds.
 	\end{multline}

 	\textbf{Step 7. Proof of Proposition \ref{Macromicro}.}
 	We will combine \eqref{Eq : estimate c}, \eqref{Eq : estimate b}, and
 	\eqref{Eq : estimate a} as follows: We choose sufficiently small $\delta_a>0$ such that the coefficient $C_\epsilon$ in \eqref{Eq : estimate a} multiplied by $\delta_a$ is less than $\frac{1}{2}.$ Here $\epsilon>0$ in \eqref{Eq : estimate a} is chosen such that $\epsilon<\frac{\delta_a}{2}$. Then we obtain \begin{multline}\label{estimate ab}\min\left\{\frac{1}{2},\frac{\delta_a}{2}\right\} \int_0^t (||a||^2_2+||b||^2_2)ds \\\leq C'\int_{\Omega\times \rth}(-\psi_af(t)-\psi_bf(t)+\psi_af(0)+\psi_bf(0))dxdv \\
 	+ \int_{0}^{t} C''\left\{||d||_2^2 + \|g\|_{\infty}^{2}\|f\|_{\sigma}^{2}+\|c\|_{2}^{2}\right\} ds.\end{multline} Finally, we choose sufficiently small $\delta_b>0$ such that $\delta_bC''<\frac{1}{2}.$ Then, we choose sufficiently small $\epsilon>0$ in \eqref{Eq : estimate c} such that $\epsilon<\frac{1}{2}\delta_b\times \min\left\{\frac{1}{2},\frac{\delta_a}{2}\right\}.$ Then we take \eqref{Eq : estimate c}+$\delta_b\times$\eqref{estimate ab} and obtain 
 	\begin{multline}
 	\int_{0}^{t} \|P f\|_{\sigma}^{2}   ds \\\leq C_0\int_{\Omega\times \rth}(-\psi_af(t)-\psi_bf(t)-\psi_cf(t)+\psi_af(0)+\psi_bf(0)+\psi_cf(0))dxdv\\
 	+C'''\int
 	_{0}^{t} \left\{  \|(I-P)f(s)\|_{2}^{2} +
 	\|g(s)\|_{\infty}^{2} \|f(s)\|_{\sigma}^{2} \right\}  ds\\
 	\leq C_0\int_{\Omega\times \rth}(-\psi_af(t)-\psi_bf(t)-\psi_cf(t)+\psi_af(0)+\psi_bf(0)+\psi_cf(0))dxdv\\
 	+C'''\int
 	_{0}^{t} \left\{  (1+||g||^2_\infty)\|(I-P)f(s)\|_{\sigma}^{2} +
 	\|g(s)\|_{\infty}^{2} \|Pf(s)\|_{\sigma}^{2} \right\}  ds.
 	\end{multline}Now we choose sufficiently small $\epsilon'>0$ with $||g(s)||^2_\infty<\epsilon'$ for all $s$ such that $C'''\times ||g(s)||^2_\infty <\frac{1}{2},$ so the $||Pf||_\sigma^2$ integral on the RHS can be absorbed in the LHS. Now we define $$\eta(t)\eqdef -\int_{\Omega\times \rth}(\psi_af(t)+\psi_bf(t)+\psi_cf(t))dxdv.$$ Then $|\eta|\lesssim ||f||_2^2.$ This completes the proof for Proposition \ref{Macromicro}.
 \end{proof}
 
 Using this proposition, we can obtain the following coercivity of the linearized Landau operator $L$:\begin{corollary}
 	\label{Coro : coercivity}Assume $\|g\|_{\infty} < \epsilon$ for some $\epsilon>0$.
 	Let $f$ be a weak solution of \eqref{linear} and \eqref{initial} with \eqref{conservation laws} and \eqref{angcon}.  Then there exist a constant $0< \delta^{\prime
 	}\le1/4$ and a function $0\le\eta(t) \le C\|f(t)\|_{2}^{2}$, such that
 	\begin{equation}
 	\label{Eq : coercivity}\int_{s}^{t} (L[f(\tau)],f(\tau))d \tau\ge
 	\delta^{\prime}\left(  \int_{s}^{t} \|f(\tau)\|_{\sigma}^{2} d \tau-
 	\{\eta(t)-\eta(s)\} \right).
 	\end{equation}Here $C>0$ is the same constant as the one in Proposition \ref{Macromicro}.
 	
 \end{corollary}
 
 \begin{proof}
 	By Lemma \ref{Lemma : Guo2002 lemma5} and Proposition \ref{Macromicro}, we have
 	\begin{multline*}
 	\int_{s}^{t} (L[f(\tau)],f(\tau))d \tau   \ge\delta\int_{s}^{t}
 	\|(I-P)f(\tau)\|_{\sigma}^{2} d \tau\\
 	\ge\delta\frac{C}{1+C} \int_{s}^{t} \|(I-P)f(\tau)\|_{\sigma}^{2} d
 	\tau+ \delta\frac{1}{1+C} \int_{s}^{t} \|(I-P)f(\tau)\|_{\sigma}^{2} d
 	\tau\\  \ge\delta\frac{C}{1+C} \int_{s}^{t} \|(I-P)f(\tau)\|_{\sigma}^{2} d
 	\tau+ \delta\frac{1}{1+C} C \left(  \int_{s}^{t} \|P f(\tau)\|_{\sigma
 	}^{2} d \tau- \{\eta(t) - \eta(s)\}\right) \\
 	= \frac{C \delta}{1+C} \left(  \int_{0}^{t} \|f(\tau)\|_{\sigma}^{2} d
 	\tau- \{\eta(t) - \eta(s)\} \right).
 	\end{multline*}
 \end{proof}
 
 Now the main ingredient, the coercivity of $L$, is ready for the proof of Theorem \ref{linear l2}. Then the rest of the proof for Theorem \ref{linear l2} follows by standard weighted energy estimates with extra weight $w^\vartheta$, which is exactly the same as the proof for Theorem 1.4 of \cite{2016arXiv161005346K} except that we replace the spatial domain $\mathbb{T}^3$ by $\Omega.$

 \section{Proof of Theorem \ref{Thm : main result}: $L^2\!\rightarrow\!L^\infty\!\rightarrow $H\"{o}lder$ \rightarrow\! S^p$ estimates and global well-posedness}
 \label{boundary flattening and mirror extension}
 Since we would like to construct an $L^2$ and $L^\infty$ global weak solution to the nonlinear \textit{Landau equation}, the strategy is basically to follow the same $L^2\!\rightarrow\!L^\infty$ framework as in \cite{2016arXiv161005346K} with some additional treatment and modifications for the boundary. 
 The modifications involve a delicate (yet natural) change of coordinates flattening the boundary and then adapt the mirror extension, so that the various estimates can be applied to our extended solution in the whole space. For the sake of simplicity, we include only the ingredients specific to the specular reflection boundary problem, while refrain from repeating the same arguments in the original paper, although the applicability is checked and explained.

In this section, we will look at a reformulation of the linearized equation (\ref{linearized-eq}) to bootstrap $L^\infty$ and H\"{o}lder estimates from $L^2$ solutions using the machinery developed for a class of kinetic Fokker-Planck equations by \cite{golse2016harnack}, thereby deriving the $S^p$ estimate to get the uniqueness. Our goal is to carry out this kind of procedure in the bounded domain with the specular reflection boundary condition.
 \subsection{Extension of solutions to the whole space}
 In this subsection, we will show step by step the way of extending our equation satisfied on a bounded domain with specular-reflection BC to a whole space problem.

 \subsubsection{``Boundary-flattening'' transformation.}
 Let
 \begin{equation}
 \begin{array}{rcl}
 \vec{\Phi} :\;\; \overline{\Omega}\times \mathbb{R}^3 & \;\rightarrow\; & \overline{\mathbb{H}}_{-}\!\times \mathbb{R}^3 \\
 (x,v) \;& \;\mapsto\; & (y,w) \eqdef  \big(\vec{\phi}(x),A v \big)
 \end{array} \label{Phi}
 \end{equation}
 be the (local) transformation that flattens the boundary, where
 \begin{equation*}
 A\eqdef \Big[\frac{\partial y}{\partial x}\Big]=D\vec{\phi}
 \end{equation*}
 is a non-degenerate $3\!\times\! 3$ Jacobian matrix, and the explicit definition of $y\!=\!\vec{\phi}(x)$ will be given below.
 Let
 \begin{equation}
 \tilde{f}(t,y,w) \eqdef  f\big(t,\vec{\Phi}^{-1}(y,w)\big) = f\big(t,\vec{\phi}^{-1}(y),A^{-1}w\big) = f(t,x,v). \label{f-trans}
 \end{equation}
 denote the solution under the new coordinates.
 
 \begin{remark}
 It is crucial that we define our transformation $\vec{\Phi}$ for both $(x,v)$ variables in this certain form so that it preserves the characteristics and the transport operator as explained in the above section (see also the subsection of \ref{Transformed-Equation} below for more details).  \end{remark}
 Suppose the boundary $\partial\Omega$ is (locally) given by the graph $x_3 = \rho(x_1,x_2)$, and $\big\{(x_1,x_2,x_3)\!\in\!\mathbb{R}^3: x_3\!<\!\rho(x_1,x_2)\big\} \subseteq \Omega$.
 Inspired by Lemma 15 in \cite{MR3592757}\footnote{The authors used spherical-type coordinates to make the map almost globally defined; here we just prefer the standard coordinates for simplicity.}, we define $y\!=\!\vec{\phi}(x)$ explicitly as follows: 
 
 \begin{align*}
 \vec{\phi}^{-1} :\;\; 
 \begin{pmatrix} 
 y_1 \\ y_2 \\ y_3
 \end{pmatrix}
 & \mapsto\; \vec{\eta}(y_1,y_2) + y_3\cdot\vec{n}(y_1,y_2) \\
 & =\;
 \begin{pmatrix} 
 y_1 \\ y_2 \\  \rho(y_1,y_2)
 \end{pmatrix}
 + y_3\cdot 
 \begin{pmatrix} 
 -\partial_1\rho \\ -\partial_2\rho \\ 1
 \end{pmatrix} \\[5pt]
 & =\;
 \begin{pmatrix} 
 y_1 - y_3\cdot\rho_1 \\ y_2 - y_3\cdot\rho_2 \\ \rho + y_3
 \end{pmatrix}
 =:
 \begin{pmatrix} 
 x_1 \\ x_2 \\ x_3
 \end{pmatrix}
 \end{align*}
 where we denote by $\rho=\rho(y_1,y_2)$, $\rho_i=\partial_i\rho(y_1,y_2),  i\!=\!1,2$, and 
 \begin{align*}
 \vec{\eta}(y_1,y_2) &\eqdef  \big(y_1, y_2, \rho(y_1,y_2) \big)  \in \partial\Omega \\[2pt]
 \partial_1\vec{\eta} &\eqdef  \frac{_{\partial\vec{\eta}}}{^{\partial y_1}} = \langle 1,0,\rho_1 \rangle \\
 \partial_2\vec{\eta} &\eqdef  \frac{_{\partial\vec{\eta}}}{^{\partial y_2}} = \langle 0,1,\rho_2 \rangle,
 \end{align*}
 then the (outward) normal vector at the point $\vec{\eta}(y_1,y_2) \in\partial\Omega$ is chosen to be
 \begin{equation*}
 \vec{n}(y_1,y_2) \eqdef  \partial_1\vec{\eta} \times \partial_2\vec{\eta} = \langle -\rho_1,-\rho_2,1  \rangle.
 \end{equation*}
 
 From the above definition we can see that the transformation $\vec{\phi}$ is ``boundary-flattening'' because it maps the points on the boundary $\{x_3 \!=\! \rho(x_1,x_2)\}$ to the plane $\{y_3\!=\!0\}$.
 We also remark that the map is locally well-defined and is a smooth homeomorphism in a tubular neighborhood of the boundary (see Lemma 15 of \cite{MR3592757} for the rigorous proof). 
 
 Directly we compute the Jacobian matrix
 \begin{align*}
 A^{-1} = D\vec{\phi}^{-1} = \Big[\frac{\partial x}{\partial y}\Big] 
 =&\; \big[ \partial_1\vec{\eta}+y_3\!\cdot\!\partial_1\vec{n} ;  \partial_2\vec{\eta}+y_3\!\cdot\!\partial_2\vec{n} ;  \vec{n} \big] \\
 =&  
 \begin{pmatrix} 
  1\!-\!y_3\!\cdot\!\rho_{11} &\;\; -y_3\!\cdot\!\rho_{12} &\;\; -\rho_1  \\[2pt]
 -y_3\!\cdot\!\rho_{12} &\;\; 1\!-\!y_3\!\cdot\!\rho_{22} &\;\; -\rho_2 \\[2pt]
 \rho_1 & \rho_2 & 1 
 \end{pmatrix} \\
 \xrightarrow{\text{on}\;\partial\Omega :\;  y_3=0}&\;  
 \big[ \partial_1\vec{\eta} ;  \partial_2\vec{\eta} ;  \vec{n} \big] \\[2pt]
 =&  
 \begin{pmatrix}
 1 &\; 0 &\; -\rho_1 \\[2pt]
 0 &\; 1 &\; -\rho_2 \\[2pt]
  \rho_1 &\;  \rho_2 & 1 
 \end{pmatrix}.
 \end{align*}
 
 So we can write out $\vec{\Phi}^{-1}$ as
 \begin{equation}
 \vec{\Phi}^{-1} :\;\; ( y,w) \; \;\mapsto\;  ( x,v) \eqdef  \big( \vec{\phi}^{-1}(y),  A^{-1}w \big)
 \end{equation}
 \begin{align*}
 \begin{pmatrix} 
 w_1 \\ w_2 \\ w_3
 \end{pmatrix}
 \mapsto &\; 
 \begin{pmatrix}
  1\!-\!y_3\!\cdot\!\rho_{11} &\;\; -y_3\!\cdot\!\rho_{12} &\;\; -\rho_1  \\[2pt]
 -y_3\!\cdot\!\rho_{12} &\;\; 1\!-\!y_3\!\cdot\!\rho_{22} &\;\; -\rho_2 \\[2pt]
 \rho_1 & \rho_2 & 1 
 \end{pmatrix}
 \begin{pmatrix} 
 w_1 \\ w_2 \\ w_3
 \end{pmatrix} \\[4pt]
 = &\; 
 \begin{pmatrix} 
 (1\!-\!y_3 \rho_{11})\cdot w_1 - y_3 \rho_{12}\cdot w_2 - \rho_1\cdot w_3 \\[2pt] 
 -y_3 \rho_{12}\cdot w_1 + (1\!-\!y_3 \rho_{22})\cdot w_2 - \rho_2\cdot w_3 \\[2pt] 
 \rho_1\cdot w_1 + \rho_2\cdot w_2 + w_3
 \end{pmatrix}
 =:
 \begin{pmatrix} 
 v_1 \\ v_2 \\ v_3
 \end{pmatrix}
 \end{align*}
 Restricted on the boundary $\partial\Omega$  i.e.,$\{y_3\!=\!0\}$, the map becomes 
 \begin{align*}
 \begin{pmatrix} 
 v_1 \\ v_2 \\ v_3
 \end{pmatrix}
 &=  w_1\cdot\partial_1\vec{\eta} + w_2\cdot\partial_2\vec{\eta} + w_3\cdot\vec{n} \\
 &= 
 \begin{pmatrix}
 1 &\; 0 &\; -\rho_1 \\[2pt]
 0 &\; 1 &\; -\rho_2 \\[2pt]
  \rho_1 &\;  \rho_2 & 1 
 \end{pmatrix}
 \begin{pmatrix} 
 w_1 \\ w_2 \\ w_3
 \end{pmatrix}
 =
 \begin{pmatrix} 
 w_1 - \rho_1\!\cdot\! w_3 \\[2pt] 
 w_2 - \rho_2\!\cdot\! w_3 \\[2pt] 
  \rho_1\!\cdot\! w_1 + \rho_2\!\cdot\! w_2 + w_3 
 \end{pmatrix}
 \end{align*}
 
 Now we are ready to show the key feature of the transformation $\vec{\Phi}$ -- preserving the ``specular symmetry'' on the boundary: it sends any two points $(x,v), (x,R_x v)$ on the phase boundary $\gamma=\partial\Omega\times\mathbb{R}^3$ with specular-reflection relation to two points on $\{y_3\!=\!0\}\times\mathbb{R}^3$ which are also specular-symmetric to each other. 
 In other words, we have the following commutative diagram (when $x\!\in\!\partial\Omega$  i.e.,$y_3\!=\!0$ ):
 \begin{equation*}
 \xymatrix{
 	(y,w)\ar[r]^{\vec{\Phi}^{-1}}\ar[d]_{R_y} & (x,v)\ar[d]^{R_x} \\
 	(y,R_y w)\ar[r]^{\vec{\Phi}^{-1}} & (x,R_x v)
 }
 \end{equation*}
 from which we can equivalently write
 \begin{equation*}
 A^{-1}\big(R_y w\big) = R_x\big(A^{-1}w\big),\quad \text{if}\;  y_3\!=\!0.
 \end{equation*}
 This can be verified either geometrically by noticing that\\ $\vec{n} = \partial_1\vec{\eta} \times \partial_2\vec{\eta}$, so
 \begin{align*}
 A^{-1}\big(R_y w\big) &= A^{-1}\langle w_1,w_2,-w_3\rangle = w_1\cdot\partial_1\vec{\eta} + w_2\cdot\partial_2\vec{\eta} - w_3\cdot\vec{n} \\
 &= R_x\big(w_1\cdot\partial_1\vec{\eta} + w_2\cdot\partial_2\vec{\eta} + w_3\cdot\vec{n}\big)
 = R_x\big(A^{-1}w\big)  ;
 \end{align*}
 or arithmetically using the definition $R_x v = v-2(n_x\!\cdot\! v) n_x$.
 
 Having this property, the specular reflection boundary condition on the solutions is also preserved: 
 \begin{equation*}
 \tilde{f}(t,y,w) = \tilde{f}(t,y,Rw),\quad \text{on}\;  \{y_3\!=\!0\},
 \end{equation*}
 where  $R\!\eqdef \!{\rm diag}\{1,1,-1\}$,
 which allows us to construct the mirror extension (as in the next subsection) that is consistent with this restriction (and thus is automatically satisfied). 
 
 To conclude this part, we take down more computations for later use: 
 \begin{equation}
 D\vec{\Phi}^{-1} = \left[\frac{\partial(x,v)}{\partial(y,w)}\right] = \left(
 \begin{array}{c|c}
 \frac{\partial x}{\partial y}  &  \frac{\partial x}{\partial w} \\[3pt]
 \hline \\[-12pt]
 \frac{\partial v}{\partial y}  &  \frac{\partial v}{\partial w}
 \end{array}
 \right)
 = \left(
 \begin{array}{c|c}
 A^{-1} &  0_{3\times 3} \\[1pt]
 \hline \\[-12pt]
 B & A^{-1}
 \end{array}
 \right) ,\label{DPhi}  
 \end{equation} 
 \begin{multline*}
 B \eqdef  \Big[\frac{\partial v}{\partial y}\Big] 
 \\=  \begin{pmatrix} 
 - y_3 \rho_{1\!1\!1}\!\cdot\! w_1 \!-\! y_3 \rho_{1\!1\!2}\!\cdot\! w_2  
 &\;\; - y_3 \rho_{1\!1\!2}\!\cdot\! w_1 \!-\! y_3 \rho_{1\!2\!2}\!\cdot\! w_2  
 &\;\; - \rho_{11}\!\cdot\! w_1 \!-\! \rho_{12}\!\cdot\! w_2 \\[-2pt]
 - \rho_{11}\!\cdot\! w_3 
 & - \rho_{12}\!\cdot\! w_3 
 &  \\[3pt]
 - y_3 \rho_{1\!1\!2}\!\cdot\! w_1 \!-\! y_3 \rho_{1\!2\!2}\!\cdot\! w_2  
 &\;\; - y_3 \rho_{1\!2\!2}\!\cdot\! w_1 \!-\! y_3 \rho_{2\!2\!2}\!\cdot\! w_2  
 &\;\; - \rho_{12}\!\cdot\! w_1 \!-\! \rho_{22}\!\cdot\! w_2 \\[-2pt]
 - \rho_{12}\!\cdot\! w_3 
 & - \rho_{22}\!\cdot\! w_3 
 &  \\[3pt]
 \rho_{11}\!\cdot\! w_1 + \rho_{12}\!\cdot\! w_2 
 & \rho_{12}\!\cdot\! w_1 + \rho_{22}\!\cdot\! w_2 
 & 0 
 \end{pmatrix}  ,
 \end{multline*}
 \begin{align*}
 &A =  \frac{1}{\det\!\left(A^{-1}\right)}\cdot\left(A^{-1}\right)^{\ast} \\
 &= \left[ y_3^{ 2}\!\cdot\!\left(\rho_{11}\rho_{22}\!-\!\rho_{12}^{ 2}\right) + y_3\!\cdot\!\left(2\rho_1\rho_2\rho_{12}\!-\!\rho_2^{ 2}\rho_{11}\!-\!\rho_1^{ 2}\rho_{22}\!-\!\rho_{11}\!-\!\rho_{22}\right) + \left(\rho_1^{ 2}\!+\!\rho_2^{ 2}\!+\!1\right) \right]^{-1} \\[3pt]
 &\;\cdot 
 \begin{pmatrix} 
  (1\!+\!\rho_2^{ 2}) - y_3\!\cdot\!\rho_{22} 
 &\;\; -\rho_1\rho_2 + y_3\!\cdot\!\rho_{12} 
 &\;\; \rho_1 + y_3\!\cdot\!(\rho_2\rho_{12}\!-\!\rho_1\rho_{22})  \\[5pt]
 -\rho_1\rho_2 + y_3\!\cdot\!\rho_{12} 
 &\;\; (1\!+\!\rho_1^{ 2}) - y_3\!\cdot\!\rho_{11} 
 &\;\; \rho_2  + y_3\!\cdot\!(\rho_1\rho_{12}\!-\!\rho_2\rho_{11})  \\[5pt]
 -\rho_1 + y_3\!\cdot\!(\rho_1\rho_{22}\!-\!\rho_2\rho_{12}) 
 &\;\; -\rho_2 + y_3\!\cdot\!(\rho_2\rho_{11}\!-\!\rho_1\rho_{12}) 
 & 1 - y_3\!\cdot\!(\rho_{11}\!+\!\rho_{22})  \\[-1pt]
 & & + y_3^{ 2}\!\cdot\!(\rho_{11}\rho_{22}\!-\!\rho_{12}^{ 2})
 \end{pmatrix},
 \end{align*}
  
 \begin{align*}
 C &\eqdef  A^{-T}\!A^{-1}\\
 &=  \begin{pmatrix}
 y_3^{ 2}\!\cdot\!(\rho_{11}^{ 2}\!+\!\rho_{12}^{ 2}) 
 &\;\; y_3^{ 2}\!\cdot\!\rho_{12}(\rho_{11}\!+\!\rho_{22})  
 &\;\; y_3\!\cdot\!(\rho_1\rho_{11}\!+\!\rho_2\rho_{12})  \\[-.5pt]
 - y_3\!\cdot\!2\rho_{11} + (\rho_1^{ 2}\!+\!1) 
 & - y_3\!\cdot\!2\rho_{12} + \rho_1\rho_2 
 &  \\[5pt]
 y_3^{ 2}\!\cdot\!\rho_{12}(\rho_{11}\!+\!\rho_{22}) 
 &\;\; y_3^{ 2}\!\cdot\!(\rho_{12}^{ 2}\!+\!\rho_{22}^{ 2})  
 &\;\; y_3\!\cdot\!(\rho_1\rho_{12}\!+\!\rho_2\rho_{22})  \\[-.5pt]
 - y_3\!\cdot\!2\rho_{12} + \rho_1\rho_2 
 & - y_3\!\cdot\!2\rho_{22} + (\rho_2^{ 2}\!+\!1) 
 &  \\[5pt]
 y_3\!\cdot\!(\rho_1\rho_{11}\!+\!\rho_2\rho_{12}) 
 & y_3\!\cdot\!(\rho_1\rho_{12}\!+\!\rho_2\rho_{22}) 
 & \rho_1^{ 2}\!+\!\rho_2^{ 2}\!+\!1
 \end{pmatrix}, \\[10pt]
 & C^{-1} = AA^T =  \frac{1}{\det(C)}\cdot C^{\ast}. 
 \end{align*}\begin{remark}
 	Here we just assume $\rho$ is (locally) smooth enough and its derivatives remain uniformly bounded, so that all the coefficients of transformed equations where the above matrices appear will keep roughly the same size as the original ones.
 \end{remark}
 \subsubsection{Mirror extension across the specular-reflection boundary.}

 After flattening the boundary, we then ``flip over'' $\tilde{f}$ to the upper half space by setting
 \begin{equation}
 \bar{f}(t,y',w') \stackrel{\text{def}}{=} \left\{ 
 \begin{array}{rcl}
 & \tilde{f}(t,y',w'), & \text{if}\;\; y'\!\in \overline{\mathbb{H}}_{-} \\[4pt]
 & \tilde{f}(t,Ry',Rw'),\quad & \text{if}\;\; y'\!\in \overline{\mathbb{H}}_{+}
 \end{array}\right.,
 \label{f-bar}
 \end{equation}
 where  $R\!\eqdef \!{\rm diag}\{1,1,-1\}$. Combined with the corresponding partition of unity, we are able to define our solutions in the whole space.
 \begin{remark}
  	The above construction of extension coincides with the specular reflection boundary condition, which in turn makes it a well-defined and continuous extension across the boundary. This observation suggests that, unfortunately, we cannot apply the same kind of extension to other boundary condition cases. 
 \end{remark}
 Also, it is worth pointing out the necessity of ``continuity of $\bar{f}$ across the boundary'' lies in that, on one hand, it ensures $\bar{f}$ is indeed a solution (at least) in the weak sense in the whole space (see Section \ref{Well-definedness}); on the other hand, $\bar{g}\eqdef \bar{f}^{(n)}$ will appear in the coefficients of the linearized equation through the iteration argument process, and we require some kind of continuity of the second-order coefficient for the $S^p$ estimate (see Section \ref{Continuity-Coefficients}).  
 \subsubsection{Transformed equations.} \label{Transformed-Equation}
 By using the chain rule with our definition (\ref{Phi}) and (\ref{f-trans}) of the transformation $\vec{\Phi}$, we first compute the transformed equation satisfied by $\tilde{f}$ in the lower half space:\footnote{We use the column vector convention in the following matrix operation expressions.}
 \begin{equation*} 
 \partial_t f = \partial_t \tilde{f},
 \end{equation*}
 \begin{align*} 
 v\cdot\nabla_{\!x} f &= \big(A^{-1}w\big)^T \left\{A^T  \nabla_{\!y}\tilde{f} + \big[\frac{_{\partial w}}{^{\partial x}}\big]^T \nabla_{\!w}\tilde{f}\right\} \\
 &= w^T\big(A^{-T}A^T\big)\nabla_{\!y}\tilde{f} + \big(A^{-1}w\big)^T \big(A\big[\frac{_{\partial v}}{^{\partial y}}\big]A\big)^T \nabla_{\!w}\tilde{f} \\
 &= w\cdot\nabla_{\!y} \tilde{f} + \big(ABw\big)\cdot\nabla_{\!w}\tilde{f},
 \end{align*}
 \begin{equation*} 
 a_g\cdot\nabla_{\!v}f = \widetilde{a_g}\cdot\big(A^T  \nabla_{\!w}\tilde{f} \big) = \big(A \widetilde{a_g}\big)\cdot\nabla_{\!w}\tilde{f},
 \end{equation*}and
 \begin{equation*} 
 \nabla_v\cdot\big(\sigma_{\!G}\nabla_{\!v} f\big) = \nabla_w\cdot\left(\big[A \widetilde{\sigma}_{\!G}A^T\big] \nabla_{\!w} \tilde{f} \right).
 \end{equation*}
 See (\ref{DPhi}) for explicit definition of $A, B$, and
 \begin{align*} 
 \widetilde{a_g}(t,y,w) &\eqdef  a_g\big(t,\vec{\Phi}^{-1}(y,w)\big) = a_g(t,x,v), \\
 \widetilde{\sigma}_{\!G}(t,y,w) &\eqdef  \sigma_{\!G}\big(t,\vec{\Phi}^{-1}(y,w)\big) = \sigma_{\!G}(t,x,v).
 \end{align*}
 
 Based on our construction of the extension (\ref{f-bar}), we then go on deriving the equation satisfied by $\bar{f}$ for the upper half space:
 \begin{equation*} 
 \partial_t \tilde{f} = \partial_t \bar{f}
 \end{equation*}
 \begin{equation*} 
 w\cdot\nabla_{\!y} \tilde{f} = \big(Rw'\big)^T R^T  \nabla_{\!y'}\bar{f} = w'^{ T}\big(R^T\!R^T\big)\nabla_{\!y'}\bar{f} = w'\cdot\nabla_{\!y'} \bar{f}  
 \end{equation*}
 \begin{equation*} 
 \big(ABw\big)\cdot\nabla_{\!w}\tilde{f} = \big(\bar{A}\bar{B}Rw'\big)\cdot R^T\nabla_{\!w'}\bar{f} = \big(R\bar{A}\bar{B}Rw'\big)\cdot \nabla_{\!w'}\bar{f}  
 \end{equation*}
 \begin{equation*} 
 \big(A \widetilde{a_g}\big)\cdot\nabla_{\!w}\tilde{f} = \big(\bar{A} \overline{a}_g\big)\cdot R^T\nabla_{\!w'}\bar{f} = \big(R\bar{A} \overline{a}_g\big)\cdot \nabla_{\!w'}\bar{f}  
 \end{equation*}
 \begin{equation*}  
 \nabla_w\cdot\left(\big[A \widetilde{\sigma}_{\!G}A^T\big] \nabla_{\!w} \tilde{f} \right) = \nabla_{\!w'}\cdot\Big(\big[R\bar{A} \overline{\sigma}_{\!G}\bar{A}^T\!R\big] \nabla_{\!w'} \bar{f} \Big)
 \end{equation*}
 where
 \begin{align*} 
 \bar{A}(y') &\eqdef  A(Ry') = A(y), \\
 \bar{B}(y',w') &\eqdef  B(Ry',Rw') = B(y,w),
 \end{align*}
 and $\overline{a}_g$, $\overline{\sigma}_{\!G}$ are $a_g$, $\sigma_{\!G}$ defined with $(t,y',w')$, respectively.
 
 Summing up the above computations, we now obtain that $\bar{f}$ satisfies (pointwisely) the following equation(s) in the lower and upper space, respectively:
 \begin{equation}
 \partial_t \bar{f} + w'\cdot\nabla_{\!y'} \bar{f} = \nabla_{w'}\cdot\big(\mathbb{A} \nabla_{\!w'} \bar{f} \big)+\mathbb{B}\cdot\nabla_{\!w'}\bar{f},
 \label{f-bar-eq}
 \end{equation}
 where the coefficients $\mathbb{A}$ and $\mathbb{B}$ are piecewise-defined:
 \begin{equation} 
 \mathbb{A}(t,y',w') \eqdef  \left\{ 
 \begin{array}{rcl}
 & \widetilde{\mathbb{A}} \eqdef  A \widetilde{\sigma}_{\!G}A^T, & \text{if}\;\; y'\!\in \mathbb{H}_{-} \\[4pt]
 & \overline{\mathbb{A}} \eqdef  R\bar{A} \overline{\sigma}_{\!G}\bar{A}^T\!R,\quad & \text{if}\;\; y'\!\in \mathbb{H}_{+}
 \end{array}\right.,
 \label{coeff-A}
 \end{equation}
 \begin{equation} 
 \mathbb{B}(t,y',w') \eqdef  \left\{ 
 \begin{array}{rcl}
 & \widetilde{\mathbb{B}} \eqdef  ABw' + A \widetilde{a_g}, & \text{if}\;\; y'\!\in \mathbb{H}_{-} \\[4pt]
 & \overline{\mathbb{B}} \eqdef  R\bar{A}\bar{B}Rw' + R\bar{A} \overline{a}_g,\quad & \text{if}\;\; y'\!\in \mathbb{H}_{+}
 \end{array}\right..
 \end{equation}
 
 \begin{remark}
 Thanks to our design of the form of transformation (\ref{Phi}) and extension (\ref{f-bar}), the transport operator of the equation remains {\em invariant} after change of variables, which is vital for our future analysis.
 \end{remark}
 It is also worth noting that the new second-order coefficient $\mathbb{A}$ preserves the positivity of $\sigma_{\!G}$, and thus the hypo-ellipticity of the equation, since $A$ and $R$ are non-degenerate. In fact, the second-order term would also be invariant if $A\!=\!D\vec{\phi}$ was an orthogonal matrix; i.e.,$A^T\! A\!=\!AA^T\!=\!I$, but that is not necessary in general and actually unrealizable. 

 \subsubsection{Adding the weights.}
 Since the velocity weight plays an important role in closing the $L^\infty$ esitimate, we need to consider the weighted equation (\ref{weighted-eq-h}), which has just one more first-order term. Thus we can simply replace its coefficient $a_g$ with $a_g^{ \theta}$, and the rest apply the same way as above.
 \subsubsection{Recovering the inhomogeneous term $\bar{K}_g^{ \theta} f$.}
 To avoid technical tediousness, we prefer to work in the whole space all the way to the end and obtain estimates for extended solutions to  (\ref{linearized-eq}) as well, using Duhamel's principle introduced in \eqref{Eq : Duhamel principle}. So we need to transform the term $\bar{K}_g^{ \theta} f$ also. And it is conceivable that the estimates for $\bar{K}_g^{ \theta} f$ after transformation should still hold valid as in Lemma 2.9 of \cite{2016arXiv161005346K}, which is all that desired by later arguments.
 
 \subsection{Well-definedness of extended solutions in the whole space} \label{Well-definedness}
 After doing the extension, it is important to make sure that across the boundary the equation(s) (\ref{f-bar-eq}) are satisfied by $\bar{f}$ in some proper sense (at least in the weak sense). That means $\bar{f}$ should satisfy the following weak formation of equation (\ref{f-bar-eq}) in the whole space: 
 \begin{align*} 
 & \iint_{\mathbb{R}^3\!\times\mathbb{R}^3} \left[(\bar{f}\varphi)(t)-(\bar{f}\varphi)(0) \right] dy'dw' \\
 & = \int_0^t\!\iint_{\mathbb{R}^3\!\times\mathbb{R}^3} \Big\{\bar{f}  \big[(\partial_s+w'\!\cdot\!\nabla_{\!y'})\varphi-\mathbb{B}\cdot\nabla_{\!w'}\varphi\big] - \nabla_{\!w'}\bar{f}\cdot(\mathbb{A} \nabla_{\!w'}\varphi) \Big\} dy'dw'ds
 \end{align*}
 
 Normally a weak formation is obtained by multiplying the equation by some suitable test function $\varphi$ and then integrating by parts over the domain where the equation(s) are defined i.e.,$(0,t)\times(\mathbb{H}_{-}\!\cup\mathbb{H}_{+})\times\mathbb{R}^3$. This process yields  
 \begin{align*} 
 & \iint_{\widetilde{\Omega}\times\mathbb{R}^3} \left[(\bar{f}\varphi)(t)-(\bar{f}\varphi)(0) \right] dy'dw' \\
 & = \int_0^t\!\iint_{\widetilde{\Omega}\times\mathbb{R}^3} \Big\{\bar{f}  \big[(\partial_s+w'\!\cdot\!\nabla_{\!y'})\varphi-\mathbb{B}\cdot\nabla_{\!w'}\varphi\big] - \nabla_{\!w'}\bar{f}\cdot(\mathbb{A} \nabla_{\!w'}\varphi) \Big\} dy'dw'ds \\
 &\quad -\int_0^t\!\!\int_{\tilde{\gamma}} \bar{f}\varphi  d\tilde{\gamma} ds
 \end{align*}
 Here  $\widetilde{\Omega} \eqdef  \mathbb{H}_{-}\!\cup\mathbb{H}_{+}$, $\tilde{\gamma} \eqdef  \partial\widetilde{\Omega}\times\mathbb{R}^3 = (\partial\mathbb{H}_{-}\!\cup\partial\mathbb{H}_{+})\times\mathbb{R}^3$, and $d\tilde{\gamma} \eqdef  (w'\!\cdot n_{y'}) dS_{y'}dw'$.
 \begin{remark}
 	 	The only boundary-integral term $I_{\tilde{\gamma}} \eqdef  \int_0^t\!\int_{\tilde{\gamma}} \bar{f}\varphi  d\tilde{\gamma} ds$ above comes from integration by parts in $y'$. Note that integration by parts in $w'$ does not produce any boundary terms. 
 \end{remark}
 
 Compared with the above definition, this is equivalent to saying that we have to be sure the boundary term vanishes:
 \begin{equation*}
 \int_{\tilde{\gamma}} \bar{f}\varphi  d\tilde{\gamma} = \left(\iint_{\partial\mathbb{H}_{-}\!\times\mathbb{R}^3} \!+ \iint_{\partial\mathbb{H}_{+}\!\times\mathbb{R}^3}\right) \bar{f}\varphi  (w'\!\cdot n_{y'}) dS_{y'}dw' = 0,
 \end{equation*}
 which is indeed true since 
 \begin{equation*}
 \bar{f}(t ;y_1',y_2',0- ;w') = \bar{f}(t ;y_1',y_2',0+ ;w')
 \end{equation*}
 due to continuity of $\bar{f}$ across the boundary, while the normal vectors at same point of outer and inner boundary are of opposite directions
 \begin{equation*}
 n_{y'}(y_3'\!=\!0-)|_{\partial\mathbb{H}_{-}} = - n_{y'}(y_3'\!=\!0+)|_{\partial\mathbb{H}_{+}},
 \end{equation*}
 plus the coincidence of $y'$-derivative term (transport operator) on two sides.
 
 Therefore, we can now conclude that $\bar{f}$ is a (weak) solution to the equation (\ref{f-bar-eq}) in the whole space.
 
 \subsection{Continuity of the coefficients across the boundary} \label{Continuity-Coefficients}
 The $L^\infty$ estimate, H\"{o}lder estimate, and $S^p$ estimate are based on a reformulation (\ref{linearized-eq}) or (\ref{weighted-eq}) of the linearized equation, which is of the form of a class of kinetic Fokker-Planck equations (also called hypoelliptic   or ultraparabolic   of Kolmogorov type) with rough coefficients (see \cite{golse2016harnack}): 
 \begin{equation*}
  \partial_t f + v\cdot\nabla_{\!x} f = \nabla_v\cdot(\mathbf{A} \nabla_{\!v} f )+\mathbf{B}\cdot\nabla_{\!v}f + \mathbf{C}f  .
 \end{equation*}
 
 The properties of the coefficients used in \cite{2016arXiv161005346K} for the estimates to hold are as follows: if $\|g\|_{\infty}$ is sufficiently small, \\[3pt] 
 $\mathbf{A}(t,x,v)\eqdef \sigma_{\!G}:\ 3\!\times\!3$ non-negative matrix, but \textit{not uniformly} elliptic, $0<(1\!+\!|v|)^{-3}I \lesssim \mathbf{A}(v) \lesssim (1\!+\!|v|)^{-1}I$\; (Lemma 2.4 in \cite{2016arXiv161005346K}); uniformly H\"{o}lder continuous if $g$ is so\; (Lemma 7.5 in \cite{2016arXiv161005346K}). \\[3pt]
 $\mathbf{B}(t,x,v)\eqdef a_g$ : \;essentially bounded $3d$-vector, $\|\mathbf{B}[g]\|_{\infty}\lesssim \|g\|_{\infty}^{2/3}\ll 1$\; (Appendix A in \cite{golse2016harnack}). \\[3pt]
 $\mathbf{C}\eqdef \bar{K}_g$ :  $L^\infty \!\rightarrow\! L^\infty$\! operator, $\|\mathbf{C}\|_{L^\infty \!\rightarrow L^\infty}\!\lesssim\! 1$ \;(Lemma 2.9 in \cite{2016arXiv161005346K}).
 
 The ellipticity and boundedness of the new coefficients after extension are easy to check (look back the transformed equations in Section \ref{Transformed-Equation}). 
 
 So we are left with one main task -- checking the H\"{o}lder continuity of the second-order coefficient $\mathbb{A}$ across the boundary, which is necessary only for the $S^p$ estimate (see Theorem 7.2 and Lemma 7.5 in \cite{2016arXiv161005346K}). 
 A direct computation on (\ref{coeff-A}) gives
 \begin{align*} 
 \widetilde{\sigma}_{\!G} &= \left|\det\!\left(A^{-1}\right)\right|\cdot \widetilde{\phi} \ast \left(\tilde{\mu}\!+\!\tilde{\mu}^{1\!/2}\tilde{g}\right),  \\[2pt]
 \phi(v) &= |v|^{-1}\!\cdot\! I - |v|^{-3}\!\cdot\!\left(vv^T\right), \\[2pt]
 \widetilde{\phi}(y,w) &= (w^T\!A^{-T}\!A^{-1}w)^{-1\!/2}\!\cdot\! I - (w^T\!A^{-T}\!A^{-1}w)^{-3/2}\!\cdot\!\left[A^{-1}ww^TA^{-T}\right], \\
 &= (w^T\! Cw)^{-1\!/2}\!\cdot\! I - (w^T\! Cw)^{-3/2}\!\cdot\!\left[A^{-1}ww^TA^{-T}\right],\ \text{and}\  \\[2pt]
 \tilde{\mu}(y,w) &= e^{-w^T\! Cw},\quad \bar{\mu}(y',w') = e^{-w'^T\!R \bar{C}Rw'}.
 \end{align*}
 Then we have
 \begin{align*}
 \widetilde{\mathbb{A}} &= A \widetilde{\sigma}_{\!G}A^T \\ 
 &= \left|\det\!\left(A^{-1}\right)\right|\!\cdot \Big\{(w^T\! Cw)^{-1\!/2}\!\cdot\! \left[AA^T\right] - (w^T\! Cw)^{-3/2}\!\cdot\!\left[ww^T\right]\!\Big\}
 \ast \left(\tilde{\mu}\!+\!\tilde{\mu}^{1\!/2}\tilde{g}\right), \\[3pt]
 \overline{\mathbb{A}} &= R\bar{A} \overline{\sigma}_{\!G}\bar{A}^T\!R \\ 
 &= \left|\det\!\left(\bar{A}^{-1}\right)\right|\!\cdot \Big\{(w'^T\!R \bar{C}Rw')^{-1\!/2}\!\cdot\! \left[R\bar{A}\bar{A}^T\!R\right] - (w'^T\!R \bar{C}Rw')^{-3/2}\!\cdot\!\left[w'\!w'^T\right]\!\Big\} \\
 &\qquad\qquad\qquad \ast \left(\bar{\mu}\!+\!\bar{\mu}^{1\!/2}\bar{g}\right).
 \end{align*}
 Since $\mathbb{A}$ is already H\"{o}lder continuous in both lower and upper spaces, it suffices to ensure that it is continuous across $\{y_3'\!=\!0\}$:
 \begin{equation*}
 \widetilde{\mathbb{A}} (t ;y_1',y_2',0- ;w') = \overline{\mathbb{A}} (t ;y_1',y_2',0+ ;w'),
 \end{equation*}
 for which we only need to verify
 \begin{enumerate}
 	\item[(1)] $A (y_1',y_2',0-) = \bar{A} (y_1',y_2',0+)$
 	\item[(2)] For $\lambda(y,w) \eqdef  w^T\! Cw,\; \bar{\lambda}(y',w') \eqdef  w'^T\!R \bar{C}Rw'$, 
 	$$\lambda (y_1',y_2',0- ;w') = \bar{\lambda} (y_1',y_2',0+ ;w') $$
 	 
 	\item[(3)] For $\Lambda(y) \eqdef  AA^T = C^{-1},\; \bar{\Lambda}(y') \eqdef  R\bar{A}\bar{A}^T\!R = R\bar{C}^{-1}\!R$, 
 	$$\Lambda (y_1',y_2',0-) = \bar{\Lambda} (y_1',y_2',0+) .$$
 \end{enumerate}
 The first claim is quite straightforward from the expression of $A$, while the latter two are less obvious. But it should become clear by writing out the matrices $C$ and $C^{-1}$, with the observation that there are no constant terms for $y_3$ in four entries $c_{13}, c_{23}, c_{31}, c_{32}$ of matrix $C$, so that at $y_3\!=\!0$ we have $c_{13}\!=\!c_{23}\!=\!c_{31}\!=\!c_{32}\!=\!0$. This means $C$ will remain unchanged after both left-multiplying and right-multiplying by an $R$. The same argument applies to $C^{-1}$ (corresponding to $\Lambda$), which also reads (when $y_3\!=\!0$)
 \begin{equation*}
 C^{-1}|_{y_3=0} 
 = \left(
 \begin{array}{c c|c}
 &  \;&\; 0 \\[-1pt]
 &  \;&\; 0 \\
 \hline 
 0 &\; 0 \;&\; \ast
 \end{array}
 \right)
 = R\bar{C}^{-1}\!R |_{y_3=0}.
 \end{equation*}
 
 Of course, we also need the assumption that $\bar{g}$ is H\"{o}lder continuous, which will be satisfied by applying the H\"{o}lder estimate to $\bar{g}\!\eqdef \!\bar{f}^{(n)}$ in the iteration process.

\begin{remark}
	The H\"{o}lder continuity of $\mathbb{A}$ across the boundary plays a key role in successfully closing our arguments. It is not trivial at all and again due to our particular choice of the transformation as well as the structure of the Landau collision kernel.
\end{remark}
 
 \begin{remark}
 	However, the first-order coefficient $\mathbb{B}$ is actually discontinuous at the boundary $\{y_3'\!=\!0\}$. Fortunately we do not require this condition in our case thanks to the presence of a higher-order term (cf. the one for the \textit{Boltzmann equation} \cite{MR2679358} and \cite{MR3592757}).
 \end{remark} 

 \subsection{Modifications of the original proofs}
 With everything prepared, we finally provide the specific ways to modify the original proofs in Section 5 - 8 of \cite{2016arXiv161005346K} for the whole space problem: 
 \begin{itemize}
 	\item[$\bullet$] The arguments in Section 5 - 8 still hold for the whole-space case by replacing each $\mathbb{T}^3$ with $\mathbb{R}^3$, as the original De Giorgi-Nash-Moser iteration for kinetic Fokker-Planck equation in \cite{mouhot2015holder} also works for the choice of cylinders $Q_n\eqdef [-t_n,0]\times \rth\times B(0;R_n)$.
 	\item[$\bullet$] For the local $L^2\!\rightarrow\!L^\infty$ estimate in Section 5.1, we can just let the cutoff function be independent of $x$, since only the truncation in $v$ is essentially required.
 	\item[$\bullet$] That being so, we also need another version of Lemma 5.8 with cylinders $Q_n$ nested in all three variables as in Section 6 (see Lemma 11 of \cite{golse2016harnack}) to deduce Lemma 6.5.
 	\item[$\bullet$] After obtaining the estimates and well-posedness results for $\bar{f}$ in the whole space, we come back to $f$ in the end by taking restrictions.
 \end{itemize}
 
 \subsection{Conclusion: applicability of the previous arguments to the whole-space case}
 By designing a suitable ``boundary-flattening'' transformation, we are able to extend our solutions to the whole space for the specular reflection boundary condition case while preserving the form of transformed equation to the largest extent. Now that the well-definedness of extended solutions and conditions on the coefficients of new equation are checked, we can conclude that the whole-space problem fits into the existing $L^2\!\rightarrow\!L^\infty$ framework without obstacles. This allows us to apply various techniques in the paper \cite{2016arXiv161005346K} and finally go back to obtain the desired results for our bounded-domain problem.

  \section*{Acknowledgements} Yan Guo's research is supported in part by NSF grants DMS-1611695, DMS-1810868, Chinese NSF grant 10828103, BICMR.  Hyung Ju Hwang's research is supported by the Basic Science Research Program through the National Research Foundation of Korea NRF-2017R1E1A1A03070105. Jin Woo Jang's research is supported by the Korean IBS project IBS-R003-D1.

\bibliographystyle{hplain}

\bibliography{Landau}{}

\end{document}